\documentclass{elsarticle}[11pt]

\usepackage{amsmath,amssymb,amsthm}
\usepackage[retainorgcmds]{IEEEtrantools}
\usepackage{pifont,mathtools}
\usepackage{hyperref}

\newtheorem{theorem}{Theorem}[section]
\newtheorem{lemma}[theorem]{Lemma}
\newtheorem{observation}[theorem]{Observation}
\newtheorem{proposition}[theorem]{Proposition}
\newtheorem{corollary}[theorem]{Corollary}

\numberwithin{equation}{section}

\theoremstyle{remark}
\newtheorem{remark}[theorem]{Remark}

\theoremstyle{definition}
\newtheorem{definition}[theorem]{Definition}
\newtheorem{notation}[theorem]{Notation}

\newcommand\R{\mathbb{R}}
\newcommand\C{\mathbb{C}}

\newcommand\N{\mathbb{N}}

\newcommand\h{\mathbb{H}}

\newcommand\xx{\text{\ding{75}}}

\DeclareMathOperator{\Tr}{Tr}
\DeclareMathOperator{\im}{Im}

\DeclareMathOperator{\id}{id}
\DeclareMathOperator{\Boch}{Boch}
\DeclareMathOperator{\diag}{diag}

\DeclareMathOperator{\inv}{inv}
\DeclareMathOperator{\rad}{rad}
\DeclareMathOperator{\alg}{alg}

\DeclarePairedDelimiter{\norm}{\lVert}{\rVert}
\DeclarePairedDelimiter{\ip}{\langle}{\rangle}

\begin{document}

\title{Operator-Valued Chordal Loewner Chains \\ and Non-Commutative Probability}

\author{David Jekel}
\address{Department of Mathematics, University of California, Los Angeles}
\ead{davidjekel@math.ucla.edu}
\ead[url]{http://www.math.ucla.edu/$\sim$davidjekel/}

\begin{abstract}
We adapt the theory of chordal Loewner chains to the operator-valued matricial upper-half plane over a $C^*$-algebra $\mathcal{A}$.  We define an $\mathcal{A}$-valued chordal Loewner chain as a subordination chain of analytic self-maps of the $\mathcal{A}$-valued upper half-plane, such that each $F_t$ is the reciprocal Cauchy transform of an $\mathcal{A}$-valued law $\mu_t$, such that the mean and variance of $\mu_t$ are continuous functions of $t$.

We relate $\mathcal{A}$-valued Loewner chains to processes with $\mathcal{A}$-valued free or monotone independent independent increments just as was done in the scalar case by Bauer \cite{Bauer2004} and Schei{\ss}inger \cite{Schleissinger2017}.

We show that the Loewner equation $\partial_t F_t(z) = DF_t(z)[V_t(z)]$, when interpreted in a certain distributional sense, defines a bijection between Lipschitz mean-zero Loewner chains $F_t$ and vector fields $V_t(z)$ of the form $V_t(z) = -G_{\nu_t}(z)$ where $\nu_t$ is a generalized $\mathcal{A}$-valued law.

Based on the Loewner equation, we derive a combinatorial expression for the moments of $\mu_t$ in terms of $\nu_t$.  We also construct non-commutative random variables on an operator-valued monotone Fock space which realize the laws $\mu_t$.  Finally, we prove a version of the monotone central limit theorem which describes the behavior of $F_t$ as $t \to +\infty$ when $\nu_t$ has uniformly bounded support.
\end{abstract}


\maketitle

\section{Introduction} \label{sec:intro}

\subsection{Chordal Loewner Chains}

Loewner chains were introduced by Karl Loewner in 1923 \cite{Loewner1923} and further developed by Kufarev and Pommerenke \cite{Pommerenke1975}.  One of the main applications of the theory was to use differential equations to prove estimates on the power series coefficients of univalent analytic functions on the unit disk (the Bieberbach conjecture).  Loewner used this technique to prove a special case of the Bieberbach conjecture, and it was a key ingredient in the full conjecture's eventual resolution \cite{deBranges1985}.

We shall focus here on \emph{chordal} Loewner chains, that is, Loewner chains defined on the upper half-plane $\h = \{\im z > 0\}$.  Our terminology follows Bauer's general treatment of chordal Loewner chains \cite{Bauer2005}.  A \emph{chordal Loewner chain} is a family of conformal maps $f(\cdot,t): \h \to \Omega_t \subseteq \h$ satisfying $F(z,t) = z - t/z + O(1/z^2)$ and $\Omega_s \supseteq \Omega_t$ for $s \leq t$.  Every chordal Loewner chain satisfies the (generalized) Loewner equation
\begin{equation} \label{eq:basicLoewner}
\partial_t F(z,t) = \partial_z F(z,t) \cdot V(z,t) \text{ for a.e.\ } t,
\end{equation}
where $V(\cdot,t): \h \to \h$ is an analytic function satisfying $V(z,t) = -1/z + O(1/z^2)$ \cite[Theorem 5.3]{Bauer2005} (the function $V(z,t)$ is known as a \emph{Herglotz vector field}).  Conversely, given a Herglotz vector field $V(z,t)$, there exists a unique chordal Loewner chain satisfying the Loewner equation \cite[Theorem 5.6]{Bauer2005}.

Since we will be working exclusively with chordal Loewner chains, for the sake of brevity, we will drop the adjective ``chordal'' when discussing Loewner chains and the Loewner equation.

If $F_t(z) = F(z,t)$ is a Loewner chain, then $F_t$ is \emph{analytically subordinated to} $F_s$ for $s < t$, that is, there exists $F_{s,t}: \h \to \h$ such that $F_t = F_s \circ F_{s,t}$.  Conversely, if $F_t$ satisfies $F_t(z) = z - t/z + O(1/z^2)$, then $(F_t)_{t \geq 0}$ is a Loewner chain if and only if $F_t$ is subordinated to $F_s$ for $s < t$.  These subordination functions satisfy $F_{s,t} \circ F_{t,u} = F_{s,u}$ for $s \leq t \leq u$.  The standard way to solve the Loewner equation is first to construct the subordination functions $F_{s,t}$ by solving the ODE $\partial_s F_{s,t} = V(F_{s,t},s)$ as $s$ runs backwards from $t$ to $0$ \cite[Theorem 5.5]{Bauer2005}.

The theory of Loewner chains has been extended to domains in $\C^n$ and even reflexive Banach spaces \cite{GHKK2013}, motivated by similar complex-analytic applications.  Our adaptation of chordal Loewner chains to the operator-valued upper half-space is motivated instead by its applications to non-commutative probability.

\subsection{Operator-Valued Non-Commutative Probability}

In non-commutative probability, one considers algebras of random variables which do not commute under multiplication.  These random variables are represented by an algebra $\mathcal{B}$ of operators on a Hilbert space, and the expectation is a linear map $E: \mathcal{B} \to \C$.  Here we take $\mathcal{B}$ to be a $C^*$-algebra; for background, see \S \ref{subsec:Cstaralgebras}.

A non-commutative version of independence called \emph{free independence} was defined by Voiculescu \cite{Voiculescu1986}.  One of the main tools for studying the distributions of non-commutative random variables is an analytic function called the Cauchy transform.  If $X$ is a self-adjoint random variable, then the \emph{Cauchy transform} of (the law of) $X$ is given by $G_X(z) = E[(z - X)^{-1}]$ and the \emph{$F$-transform} $F_X(z)$ is given by $F_X(z) = G_X(z)^{-1}$.  The $F$-transform is a self-map of the upper half-plane.

In 2004, Bauer \cite{Bauer2004} connected Loewner chains to non-commutative probability by observing that if $(X_t)_{t \geq 0}$ is a process with freely independent increments, then $F_{X_t}(z)$ is a chordal Loewner chain. This follows from a fundamental theorem in free probability that if $X$ and $Y$ are freely independent, then $F_{X + Y} = F_X \circ F$ for some analytic function $F: \h \to \h$ (see \cite[Proposition 4.4]{Voiculescu1993}, \cite[Theorem 3.1]{Biane1998}).  Thus, if $(X_t)$ is a process with freely independent increments, then $F_{X_t} = F_{X_s} \circ F_{s,t}$ for some $F_{s,t}: \h \to \h$.  This implies that the evolution of the laws associated to a process with free increments is described by a Loewner chain.  However, not all Loewner chains arise in this way (see \S \ref{subsec:freeincrements} for an explicit counterexample).

Furthermore, in the case of a process with freely independent and stationary increments, if $F_t(z) = F_{X_t}(z)$, then we have
\begin{equation}
\partial_t F_t(z) = -\partial_z F_t(z) \Phi(F_t(z)),
\end{equation}
where $\Phi$ is a function from the upper half-plane to the lower half-plane.  Thus, in this case, the vector field for the Loewner equation is given by $V(z,t) = -\Phi(F_t(z))$.  For further discussion, see \S \ref{subsec:semigroups} as well as \cite[Theorem 4.3]{Voiculescu1986} and \cite[\S 3.5]{Schleissinger2017}.

Loewner chains are more directly connected to another type of non-commutative independence called \emph{monotone independence}, defined by Muraki \cite{Muraki2000}, \cite{Muraki2001}.  If two random variables $X$ and $Y$ are monotone independent, then $F_{X+Y} = F_X \circ F_Y$.  Thus, if $(X_t)_{t \geq 0}$ is a process with monotone independent increments, then $F_{X_t}$ is subordinated to $F_{X_s}$ for $s < t$, and hence $1/G_{X_t}$ is a Loewner chain in this case as well.  Conversely, every Loewner chain arises in this way from a process with monotone independent increments, as shown by Schlei{\ss}inger \cite[Theorem 3.6]{Schleissinger2017}.

In the case of a monotone convolution semigroup (a process with monotone independent and stationary increments), the differential equation for the evolution of $F_t = 1/G_{X_t}$ was studied much earlier by Muraki \cite[\S Theorem 4.5]{Muraki2000} and Hasebe \cite[\S 3.1]{Hasebe2010-1}, who showed that
\begin{equation} \label{eq:monotonesemigroup}
\partial_t F_t(z) = F_t'(z) A(z)
\end{equation}
for a function $A(z)$ which is a Herglotz vector field in the case where $X_t$ has mean zero.  The function $A(z)$ serves as a generating function for the monotone cumulants \cite[Concluding remark]{HS2011-2} \cite[Concluding remark]{HS2014}.  Schlei{\ss}inger \cite{Schleissinger2017} recognized \eqref{eq:monotonesemigroup} as a special case of the Loewner equation.

We refer to \cite{Schleissinger2017} for a summary of the connections between Loewner chains and non-commutative probability in the scalar case.  Our goal is to generalize these results to the operator-valued setting, in which the scalar field $\C$ is replaced by a $C^*$-algebra $\mathcal{A}$ and the expectation is $\mathcal{A}$-valued; the setup is described in detail in \S \ref{subsec:Cstaralgebras} - \ref{subsec:Avaluedlaws}.

The development of operator-valued non-commutative probability, initiated in \cite{Voiculescu1995}, has several motivations.  First, even scalar-valued non-commutative probability would motivate us to consider the case $\mathcal{A} = M_n(\C)$ because the law of a tuple of operators can be analyzed by putting them together into a single matrix.  Similarly, a polynomial $p$ applied to an operator $X$ can be expressed as one entry of $P(X \otimes 1_n)$ where $P$ is a non-commutative polynomial with coefficients in $M_n(\C)$ and $X \otimes 1_n$ is the diagonal matrix of $X$'s.  For background on these linearization tricks, see \cite{Speicher1998}, \cite[\S 2]{HT2005}, \cite[\S 2.6]{Anderson2013}, \cite[Lemma 3.2]{Williams2017}, \cite{HMS2015}.  Furthermore, operator-valued non-commutative independence is a natural non-commutative analogue of conditional independence.

In the operator-valued setting, in order to get an analytic characterization of operator-valued Cauchy transforms, one must work not only with functions taking values in the algebra $\mathcal{A}$, but with so-called \emph{fully matricial} or \emph{non-commutative} functions (see \S \ref{subsec:matricial} - \S \ref{subsec:Ftransform}).  In other words, the Cauchy transform $G_\mu(z)$ must be viewed as a function which is defined not only when $z$ is in $\mathcal{A}$, but also when $z$ is an $n \times n$ matrix over $\mathcal{A}$.  Voiculescu introduced the matricial Cauchy transform into free probability \cite{Voiculescu2000} \cite{Voiculescu2004} \cite{Voiculescu2010}, and this was later recognized as a special case of non-commutative function theory (see \cite{KVV2014}).  An analytic characterization of operator-valued Cauchy transforms was recently given by Williams \cite[Theorem 3.1]{Williams2017}; this will be a key ingredient in our analysis.

We refer to \cite{Voiculescu1995,VDN1992,Speicher1998,MS2013,BMS2013,PV2013} for background on operator-valued free independence (a.k.a.\ free independence with amalgamation), and to \cite{Popa2008a,BPV2013,HS2014,AW2016} for background on operator-valued monotone independence.

\subsection{Main Results}

Suppose that $\mathcal{A}$ is a $C^*$-algebra (see \S \ref{subsec:Cstaralgebras} for definition).  An \emph{$\mathcal{A}$-valued fully matricial function} is, roughly speaking, a collection of analytic functions $F^{(n)}(z)$ from an open set $U^{(n)}$ of $M_n(\mathcal{A})$ to another open set of $M_n(\mathcal{A})$, such that the domains $U^{(n)}$ and the functions $F^{(n)}$ behave consistently under direct sums of matrices and conjugation by scalar matrices (see \S \ref{subsec:matricial} for precise definitions).  The matricial upper half-plane $\h(\mathcal{A})$ consists of the matrices $z \in M_n(\mathcal{A})$ such that $\im z := -\frac{1}{2} i(z - z^*)$ satisfies $\im z \geq \epsilon$ for some $\epsilon > 0$ depending on $z$ (see Definition \ref{def:halfspace}).

We define a \emph{Loewner chain} as a family $(F_t)_{t \in [0,T]}$ of fully matricial functions $\h(\mathcal{A}) \to \h(\mathcal{A})$ such that
\begin{enumerate}[(1)]
	\item $F_0 = \id$.
	\item $F_t(z) = F_{\mu_t}(z)$ for some law $\mu_t$ with ``bounded support''.
	\item For each $s < t$, we have $F_t = F_s \circ F_{s,t}$ for some matricial-analytic $F_{s,t}: \h(\mathcal{A}) \to \h(\mathcal{A})$
	\item The mean $\mu_t(X)$ and the variance $\mu_t(X^2)$ are continuous functions of $t$.
\end{enumerate}
Furthermore, we say that a Loewner chain $(F_t)$ is \emph{normalized} if $\mu_t(X) = 0$, and it is \emph{Lipschitz} if $\mu_t(X^2)$ is a Lipschitz function of $t$.  (See Definitions \ref{def:CLC}, \ref{def:LNCLC})  We have the following results concerning operator-valued Loewner chains:

{\bf Regularity of Loewner Chains:} Proposition \ref{prop:biholomorphic} shows that if $(F_t)_{t \in [0,T]}$ is a Loewner chain, then each $F_t$ is a biholomorphic map with a fully matricial inverse.  We also show that the subordination map $F_{s,t}$ is automatically a reciprocal Cauchy transform, and the laws $\mu_t$ automatically have ``support radius'' which is uniformly bounded for $s, t \in [0,T]$, in fact, less than or equal to a constant times the support radius of $\mu_T$.

{\bf Loewner Chains and Monotone Probability:} Theorem \ref{thm:processLoewnerchain} shows that $F_t(z) = G_{X_t}(z)^{-1}$ provides a correspondence between Loewner chains $F_t$ and processes $X_t$ with $\mathcal{A}$-valued monotone independent increments.  This generalizes \cite[Theorem 3.5]{Schleissinger2017}.

{\bf Loewner Chains and Free Probability:} Theorem \ref{thm:freeLoewnerchain} shows that if $X_t$ is a process with $\mathcal{A}$-valued free increments, then $F_{X_t}(z)$ is an $\mathcal{A}$-valued Loewner chain.  This result is motivated by \cite{Bauer2004} and \cite[\S 3.5]{Schleissinger2017}.  We give a simple example to show that not every Loewner chain arises from a process with free increments.

{\bf Differentiation of Loewner Chains:} Theorem \ref{thm:Loewnerdifferentiation} shows that every Lipschitz normalized Loewner chain is a locally Lipschitz function of $t$, and it satisfies
\begin{equation}
\partial_t F(z,t) = DF(z,t)[V(z,t)],
\end{equation}
where $DF(z,t)$ is the Fr{\'e}chet derivative with respect to $z$ and $V(z,t) = -G_{\nu_t}(z)$ for some generalized law $\nu_t$.  Here the differentiation with respect to $t$ occurs in a distributional sense defined in \S \ref{subsec:distderiv}.  The vector field $V(z,t)$ and the generalized law $\nu_t$ also depend on $t$ in a distributional sense (see Definitions \ref{def:Herglotz} and \ref{def:distributionallaw}).  This is the operator-valued analogue of \cite[Theorem 5.3]{Bauer2005}.

{\bf Integration of the Loewner Equation:} Theorem \ref{thm:Loewnerintegration} shows that conversely, given a vector field $V(z,t)$, there exists a unique Loewner chain $F(z,t)$ satisfying the Loewner equation, generalizing \cite[Theorem 5.6]{Bauer2005}.

{\bf Combinatorial Formula for $F(z,t)$:} Theorem \ref{thm:combinatorics} provides a combinatorial formula for the coefficients of the power series of $F(z,t)$ at $\infty$ in terms of the measures $\nu_t$.  This generalizes moment formulas from monotone probability theory, e.g.\ \cite[p.\ 33-34]{Muraki2000} \cite[Corollary 5.2, Theorem 5.3, Remark 6.4]{HS2011-2}, \cite[concluding paragraph]{HS2014}, \cite[Theorem 2.5]{BPV2013}, \cite[Definition 4.4 and Proposition 4.8]{AW2016}.

{\bf Fock-space Realization of the Laws $\mu_t$:} Theorem \ref{thm:Fockspace} describes how to realize the laws $\mu_t$ corresponding to a Lipschitz Loewner chain $F_t$.  We construct operators $Y_t$ on a monotone Fock space (related to constructions in \cite{Lu1997}, \cite{Muraki1997}, \cite[\S 4.7]{Speicher1998}) such that $(Y_t)$ is a process with monotone independent increments and $Y_t$ has the law $\mu_t$.

{\bf Central Limit Theory for Loewner Chains:} Theorem \ref{thm:CLTcoupling} proves a central limit theorem describing the behavior of $F(z,t)$ for large $t$, assuming that the measures $\nu_t$ maintain bounded support.  Using the Fock space of Theorem \ref{thm:Fockspace}, we explicitly construct a coupling between the process $Y_t$ and another process $Z_t$ with monotone independent increments, such that $Z_t$ has the central-limit arcsine distribution.  Theorem \ref{thm:CLT} is another version of the central limit theorem which estimates the difference between $G_{Y_t}$ and $G_{Z_t}$ using the Loewner equation.

To complete our overview of the paper, let us also summarize the earlier sections which lay the technical groundwork for the theory of operator-valued Loewner chains.

{\bf \S \ref{sec:locallyLipschitz} Locally Lipschitz Families:} This section will discuss analytic functions over Banach spaces, and develop a theory of distributional differentiation (with respect to $t$) of a family $F(z,t)$ which is analytic in $z$ and Lipschitz in $t$.

{\bf \S \ref{sec:NCP} $C^*$-valued Non-Commutative Probability:} This section reviews background on operator-valued probability spaces and laws in \S \ref{subsec:Cstaralgebras} - \S \ref{subsec:Avaluedlaws}.  Next, in \S \ref{subsec:matricial} - \S \ref{subsec:Cauchytransform}, we describe fully matricial functions and explain Williams and Anshelevich's analytic characterization of operator-valued Cauchy transforms \cite[Theorem 3.1]{Williams2017} \cite[Theorem A.1]{AW2016}.  We also state some basic estimates for Cauchy transforms which we will frequently use.  In \S \ref{subsec:Ftransform}, we analyze the behavior of $F$-transforms (reciprocal Cauchy transforms) at $\infty$, showing that if $F_3 = F_1 \circ F_2$ and two of the functions are $F$-transforms, then so is the third.  We also relate the support radii of the associated laws.  We also observe that for every law $\mu$, there exists a generalized $\nu$ such that $F_\mu(z) = z - a - G_\nu(z)$.

This paper is restricted to the case of self-adjoint random variables for brevity, but further research should investigate analogous questions for unitary operators and Loewner chains on the operator-valued matricial unit ball (see \cite{Bauer2004} for motivation).  We have not tried to remove the assumption that our laws have ``bounded support'' because a good theory of Cauchy transforms for unbounded laws is not yet available.  We have treated processes with free increments briefly, but we plan to discuss them in greater detail in later work.

\section{Locally Lipschitz Families of Analytic Maps} \label{sec:locallyLipschitz}

For a scalar-valued Loewner chain $F(z,t)$, the dependence on $z$ is analytic and the dependence on $t$ is locally Lipschitz.  Loewner theory relies on the regularity theory for such locally Lipschitz families of analytic functions.  For instance, a locally Lipschitz family can be differentiated a.e.\ with respect to $t$, and we have $\partial_t \partial_z^k F = \partial_z^k \partial_t F$.  Moreover, for two such families $F$ and $G$, the chain rule holds for computation of $\partial_t [G(F(z,t),t)]$.  Our goal in this section is to prove analogous results for locally Lipschitz familes of analytic functions $F: \mathcal{X} \times [0,T] \to \mathcal{X}$, where $\mathcal{X}$ is a Banach space.

\subsection{Analytic Functions between Banach Spaces}

We recall some standard definitions and facts about analytic functions between Banach spaces.  For background, see \cite{Zorn1945a}, \cite{Zorn1945b}, \cite{Zorn1946}.  For the reader's convenience, we include only the level of generality that will be used in this paper and give sketches of the proofs.

\begin{notation}
For a Banach space $\mathcal{X}$, we denote by $B_{\mathcal{X}}(0,R)$ the open ball in $\mathcal{X}$ of radius $R$ centered at $x$.  We denote by $\mathcal{L}(\mathcal{X}, \mathcal{Y})$ the space of bounded linear transformations from $\mathcal{X}$ to $\mathcal{Y}$.
\end{notation}

\begin{definition} \label{def:analytic}
Let $\mathcal{X}$ and $\mathcal{Y}$ be Banach spaces and let $\mathcal{U} \subseteq \mathcal{X}$ be open.  We say that $F: \mathcal{U} \to \mathcal{Y}$ is \emph{analytic} if
\begin{enumerate}[(1)]
	\item $F$ is \emph{locally bounded}, that is, for each $x \in \mathcal{U}$, there exists $R > 0$ and $M > 0$ such that $\norm{x - x_0} < R$ implies $y \in \mathcal{U}$ and $\norm{f(x)} \leq M$.
	\item $F$ is \emph{G{\^a}teaux-differentiable}, that is, for each $x_0 \in \mathcal{U}$ and $h \in \mathcal{X}$,
	\[
	\lim_{\zeta \to 0} \frac{F(x + \zeta h)}{\zeta} \text{ exists.}
	\]
\end{enumerate}
\end{definition}

\begin{theorem} \label{thm:analytic}
Let $F: \mathcal{U} \to \mathcal{Y}$ be analytic.  For each $x_0 \in \mathcal{U}$ and $h \in \mathcal{X}$ and $k \in \N$, the iterated complex G{\^a}teaux derivative
\[
\delta^k F(x_0;h) := \frac{d^k}{d \zeta^k}\bigr|_{\zeta = 0} F(x_0 + \zeta h)
\]
is defined.  Moreover, $\delta^k F(x_0; \alpha h) = \alpha^k \delta^k F(x_0,h)$ for $\alpha \in \C$.  If $F(x)$ is defined and $\norm{F(x)} \leq M$ for $x \in B_{\mathcal{X}}(x_0,R)$, then we have the Cauchy estimate
\begin{equation} \label{eq:Cauchyestimate}
\frac{1}{k!} \norm{\delta^k F(x_0;h)} \leq \frac{\norm{h}^kM}{R^k}.
\end{equation}
and the local power series expansion
\begin{equation} \label{eq:powerseries}
F(x_0+h) = \sum_{k=0}^\infty \delta^k F(x_0;h) \text{ for } \norm{h} < R.
\end{equation}
\end{theorem}

\begin{proof}
Assume that $F(x)$ is defined and bounded by $M$ on $B_{\mathcal{X}}(x_0,R)$ and let $h \in \mathcal{X}$.  Suppose that $y^* \in \mathcal{Y}^*$ and $\norm{y^*} \leq 1$.  Then $y^*[F(x_0+\zeta h)]$ is a scalar-valued complex-differentiable function on $B_{\C}(0, R / \norm{h})$ which is bounded by $M$.  Thus, by Goursat's theorem, $y^*[F(x_0+\zeta h)]$ is analytic, and we have a power series expansion
\begin{equation}
y^*[F(x_0+\zeta h)] = \sum_{k=0}^\infty \frac{1}{k!} \alpha_{x,h,y^*k} \zeta^k,
\end{equation}
where $\alpha_{x_0,h,y^*} \in \C$.  Using the Cauchy integral formula and the resulting Cauchy estimates, we have
\begin{equation} \label{eq:initialCauchyestimate}
\frac{1}{k!} |\alpha_{x_0,h,y^*}| \leq \frac{M \norm{h}^k}{R^k}.
\end{equation}
Next, one argues by induction on $k$ that there exists $f_{x,h,n} \in \mathcal{Y}$ such that $\alpha_{x_0,h,y^*,n} = y^*[f_{x,h,n}]$ for all $y^* \in B_{\mathcal{Y}^*}(0,1)$.  The base case $n = 0$ is trivial since $f_{x_0,h,0} = F(x_0)$.  For the inductive step, note that by inductive hypothesis,
\[
F(x+\zeta h) = \sum_{k=0}^{n-1} y^*[f_{x_0,h,k}] \zeta^k + \alpha_{x_0,h,y^*,n} \zeta^n + O_{M,R,h}(\zeta^{n+1}),
\]
where, because of \eqref{eq:initialCauchyestimate}, the error estimate is uniform for $y^* \in B_{\mathcal{Y}^*}(0,1)$.  Thus,
\[
\alpha_{x_0,h,y^*,n} = \lim_{\zeta \to 0} \frac{1}{\zeta^n} y^* \left[ F(x+\zeta h) - \sum_{k=0}^{n-1} f_{x_0,h,k} \zeta^k \right]
\]
with the rate of convergence independent of $y^*$.  It follows that $\lim_{\zeta \to 0} (F(x+\zeta h) - \sum_{k=0}^{n-1} f_{x_0,h,k} \zeta^k)$ exists in $\mathcal{Y}$, and we call this limit $f_{x_0,h,n}$.  Then we have $y^*[f_{x_0,h,n}] = \alpha_{x_0,h,y^*,n}$.  Because the Cauchy estimate \eqref{eq:initialCauchyestimate} holds for all $y^* \in B_{\mathcal{Y}^*}(0,1)$, we obtain $(1/k!) \norm{f_{x_0,h,n}} \leq M \norm{h}^n / R^n$.  From this, one checks that for $|\zeta| < R / \norm{h}$,
\[
F(x_0+\zeta h) = \sum_{k=0}^\infty \frac{1}{k!} f_{x,h,k} \zeta^k.
\]
It follows that $(d / d \zeta)^k F(x_0 + \zeta h)|_{\zeta = 0}$ exists and equals $f_{x_0,h,k}$, so that $\delta^k F(x_0;h)$ is defined and equals $f_{x_0,h,k}$.  The theorem now follows from the claims we proved about $f_{x_0,h,k}$.
\end{proof}

As a consequence of Theorem \ref{thm:analytic} we have the following local continuity estimates.  If $F: B_{\mathcal{X}}(x_0,R) \to \mathcal{Y}$ is analytic and bounded by $M$, then we have
\begin{equation}
\norm{F(x) - F(x_0)} \leq \frac{M \norm{x - x_0}}{R - \norm{x - x_0}},
\end{equation}
which follows from estimating the power series \eqref{eq:powerseries} term by term.  Moreover, $F$ is Lipschitz on $B_{\mathcal{X}}(x_0,r)$ for every $r < R / 2$.  Indeed, if $x_1$, $x_2 \in B_{\mathcal{X}}(x_0,r)$, then we can apply the previous estimate with $x_1$ in place of $x_0$ and $R - r$ in place of $R$ to obtain
\begin{equation}  \label{eq:aprioriLipschitz}
\norm{F(x_1) - F(x_2)} \leq \frac{M\norm{x_1 - x_2} }{R - r - \norm{x_1 - x_2}} \leq \frac{M\norm{x_1 - x_2}}{R - 2r} \text{ whenever } r < R/2 \text{ and } x, x' \in B_\mathcal{X}(x_0,r).
\end{equation}
In particular, $F$ is continuous and locally Lipschitz.

Continuing with $x_0$, $x_1$, $x_2$, $r$, $R$ as above, we may apply the Cauchy estimate to $x \mapsto F(x_1 + x) - F(x_2 + x)$ on $B_{\mathcal{X}}(0,R-r)$ to obtain
\begin{equation}  \label{eq:aprioriLipschitz2}
\frac{1}{k!} \norm{\delta^k F(x_1;h) - \delta^k F(x_2;h)} \leq \frac{M\norm{x_1 - x_2} \norm{h}^k}{(R - 2r)(R - r)^k} \text{ whenever } r < R/2 \text{ and } x, x' \in B_\mathcal{X}(x_0,r).
\end{equation}
In particular, $\delta F(x;h)$ is continuous in $x$.

\begin{lemma}
A function $F: \mathcal{U} \to \mathcal{Y}$ is analytic on $\mathcal{U}$ if and only if it is complex Fr{\'e}chet differentiable at each point $x_0 \in \mathcal{U}$, that is, there exists a bounded linear transformation $DF(x_0): \mathcal{X} \to \mathcal{Y}$ such that
\[
F(x) = F(x_0) + DF(x_0)[x - x_0] + o(\norm{x - x_0}).
\]
Moreover, in this case $DF(x_0)[h] = \delta F(x_0;h)$.
\end{lemma}

\begin{proof}[Sketch of proof]
We claim that $\delta F(x_0;h)$ is complex linear in $h$.  We have already shown it is homogeneous.  Fix $x_0$ and suppose that $F$ is defined and bounded by $M$ on $B_{\mathcal{X}}(x_0,R)$.  Fix $h_1$ and $h_2$ in $\mathcal{X}$.  Then
\[
F(x_0 + \zeta h_1 + \zeta h_2) = F(x_0) + \delta F(x_0; h_1 + h_2) \zeta + O_{M,R}\left( |\zeta|^2 \norm{h_1 + h_2}^2  \right),
\]
where the error estimate only depends on $M$ and $R$.  On the other hand, if $r < R / 2$, we also have
\[
F(x_0 + \zeta h_1 + \zeta h_2) = F(x_0) + \zeta \delta F(x_0;h_1) + \zeta \delta F(x_0+\zeta h_1; h_2) + O_{M,r}\left( |\zeta|^2 \norm{h_1}^2 + |\zeta|^2 \norm{h_2}^2 \right).
\]
provided that $|\zeta| < 2 \min(\norm{h_1}, \norm{h_2}, \norm{h_1 + h_2}) / r$.  Because of \eqref{eq:aprioriLipschitz2}, $\delta F(x_0 + \zeta h_1; h_2) \to \delta F(x_0; h_2)$ as $\zeta \to 0$.  Thus, $\delta F(x_0;h_1 + h_2) = \delta F(x_0;h_1) + \delta F(x_0;h_2)$.  If we define $DF(x_0)[h] = \delta F(x_0;h)$, then $DF(x_0)$ is a bounded linear transformation and given that $F(x_0 + h) \approx F(x_0) + \delta F(x_0;h) + O_R(\norm{h}^2)$, it follows that $F$ is Fr{\'e}chet differentiable.  Conversely, if $F$ is Fr{\'e}chet differentiable, then it is locally bounded and G{\^a}teaux differentiable, and hence analytic.
\end{proof}

As a consequence of the last lemma and the chain rule for Fr{\'e}chet differentiation, we have the following.

\begin{lemma}
The composition of two analytic functions is analytic.
\end{lemma}

\begin{lemma} \label{lem:analyticlimit}
Suppose that $(F_n)$ is a sequence of analytic functions $\mathcal{U} \to \mathcal{Y}$ where $\mathcal{U} \subseteq \mathcal{X}$ is open.  Assume that $\sup_n \norm{F_n}$ is locally bounded and suppose for each $x_0$, we have $F_n(x) \to F(x)$ uniformly on some neighborhood of $x_0$.  Then $F$ is analytic.
\end{lemma}

\begin{proof}[Sketch of proof]
In the case of scalar analytic functions, this result is a standard theorem of Weierstrass.  We can apply the scalar result to $\zeta \mapsto y^*[F_n(x + \zeta h)]$ for each $y^* \in \mathcal{Y}^*$, $x \in \mathcal{U}$, and $h \in \mathcal{X}$.
\end{proof}

\begin{lemma} \label{lem:analyticcontinuation}
Let $\mathcal{U}$ be open and connected.  If $F$ and $G: \mathcal{U} \to \mathcal{Y}$ are equal on an open subset of $\mathcal{U}$, then $F = G$ on $\mathcal{U}$.
\end{lemma}

\begin{proof}[Sketch of proof]
Consider the set
\[
\mathcal{V} = \{x \in \mathcal{U}: \delta^k F(x;h) = \delta^k G(x;h) \text{ for all } k \in \N \text{ and } h \in \mathcal{X}\}.
\]
This is closed relative to $\mathcal{U}$ by continuity of $\delta^k F$ and $\delta^k G$ and it is open because of the local power series expansion \eqref{eq:powerseries}.  Hence, $\mathcal{V} = \mathcal{U}$ by connectedness.
\end{proof}

\begin{lemma} \label{lem:analyticcontinuationofconvergence}
Suppose that $F_n$ is analytic $\mathcal{U} \to \mathcal{Y}$, where $\mathcal{U} \subseteq \mathcal{X}$ is open and connected.  Assume that $\sup_n \norm{F_n(x)}$ is locally bounded.  If $F_n \to F$ with respect to $\norm{\cdot}_\mathcal{Y}$ uniformly on $B_\mathcal{X}(x_0,R)$ for some $x_0 \in \mathcal{U}$ and some $R > 0$, then $F$ extends to be analytic on $\mathcal{U}$ and $F_n \to F$ locally uniformly on $\mathcal{U}$.
\end{lemma}

\begin{proof}[Sketch of proof]
Consider the set
\[
\mathcal{V} = \{x \in \mathcal{U}: \lim_{n \to \infty} \delta^k F_n(x;h) \text{ exists for all } k \in \N \text{ and } h \in \mathcal{X}\}.
\]
Because $\sup_n \norm{F_n(x)}$ is locally bounded, \eqref{eq:aprioriLipschitz2} implies that $\delta^k F_n(x;h)$ is equicontinuous in $x$ with respect to $n$ for each fixed $k$, hence $\mathcal{V}$ is closed.  On the other hand, if $x \in \mathcal{V}$, one can use the local boundedness of $F_n$, which is uniform in $n$, together with \eqref{eq:powerseries} and \eqref{eq:Cauchyestimate}, to show that for $x'$ in some open neighborhood of $x$, the sequence $\delta^k F_n(x';h)$ converges for each $k$ and $h$.  So $\mathcal{V}$ is both open and closed, hence $\mathcal{V} = \mathcal{U}$.  The limit function is analytic by Lemma \ref{lem:analyticlimit}.
\end{proof}

\subsection{Banach-valued Measurability and Integration}

We recall some terminology regarding Bochner integration of Banach-valued functions on an interval $[0,T]$.  For background, see \cite{DP1940} \cite{Phillips1940} \cite{MU1971} \cite{DU1977}.

Let $\mathcal{X}$ be a Banach space.  A function $\gamma: [0,T] \to \mathcal{X}$ is said to be \emph{norm measurable} or \emph{strongly measurable} if it is a measurable function with respect to Borel $\sigma$-algebra on $\mathcal{X}$ in the norm topology.  It is said to be \emph{weakly measurable} if $(\phi, \gamma(t))$ is measurable for every $\phi \in \mathcal{X}^*$, where $(\cdot,\cdot)$ denotes the dual pairing.  A map $\gamma: [0,T] \to \mathcal{X}^*$ is \emph{weak-$*$ measurable} if $(\gamma(t), x)$ is measurable for every $x \in \mathcal{X}$.

A \emph{simple function} $[0,T] \to \mathcal{X}$ is a function of the form $\sum_{j=1}^\infty x_j \cdot \chi_{E_j}(t)$, where $y_j \in \mathcal{Y}$ and the $E_j$'s are disjoint and measurable, and its $L^p$ norm is
\begin{align}
\norm*{\sum_{j=1}^\infty x_j \cdot \chi_{E_j} }_{L^p} &= \left(\sum_{j=1}^\infty |E_j| \norm{x_j}^p\right)^{1/p} \text{ for } 1 \leq p < \infty \\
\norm*{\sum_{j=1}^\infty x_j \cdot \chi_{E_j} }_{L^\infty} &= \sup_{j, |E_j| > 0} \norm{x_j},
\end{align}
where $|E_j|$ denotes the Lebesgue measure.  The \emph{Bochner $L^p$ space} $L_{\text{Boch}}^p([0,T],\mathcal{X})$ is the completion with respect to this norm of the space of simple functions with finite $L_{\Boch}^p$ norm modulo equality almost everywhere.

The space $L^1([0,T], \mathcal{X})$ can equivalently be characterized as the space of norm-measurable functions $\gamma: [0,T] \to \mathcal{X}$ such that $\int \norm{\gamma(t)}\,dt < +\infty$ and $f$ is \emph{almost separably valued}, that is, there exists a separable subspace $\mathcal{Y} \subseteq \mathcal{X}$ such that $f(t) \in \mathcal{Y}$ for a.e.\ $t$.

A (norm-) continuous function $\gamma: [0,T] \to \mathcal{X}$ is in $L_{\Boch}^1([0,T], \mathcal{X})$.  Moreover, continuous functions and step functions are both dense in $L_{\Boch}^1([0,T], \mathcal{X})$.

\subsection{Distributional Derivatives} \label{subsec:distderiv}

This section examines distributional derivatives of a Lipschitz functions $\gamma: [0,T] \to \mathcal{X}$, where $\mathcal{X}$ is some Banach space.  In particular, we will describe how to perform various ``pointwise'' operations with elements of $\mathcal{L}(L^1[0,T], \mathcal{X})$, including nonlinear operations involving composition.

As motivation, recall that if $\gamma: [0,T] \to \C$ is Lipschitz, then the distributional derivative $\dot{\gamma}: C_c^\infty(0,T) \to \C$ is represented by a function in $L^\infty(0,T) = L^1(0,T)^*$.  In general, if $\gamma: [0,T] \to \mathcal{X}$ is Lipschitz, then the distributional derivative $\dot{\gamma}: C_c^\infty(0,T) \to \mathcal{X}$ is \emph{not} necessarily represented by a function in $L_{\Boch}^\infty([0,T], \mathcal{X})$.  However, we claim that $\dot{\gamma}$ does extend to a bounded map $L^1[0,T] \to \mathcal{X}$.  In the following, we denote by $\mathcal{L}(L^1[0,T], \mathcal{X})$ the space of bounded linear maps $L^1[0,T] \to \mathcal{X}$.

\begin{observation} \label{obs:weakderivative}
If $\gamma: [0,T] \to \mathcal{X}$ is Lipschitz, then there exists a unique $\dot{\gamma} \in \mathcal{L}(L^1[0,T], \mathcal{X})$ satisfying
\begin{equation}
\dot{\gamma}[\chi_{[a,b]}] = \gamma(b) - \gamma(a).
\end{equation}
Conversely, if $\rho \in \mathcal{L}(L^1[0,T], \mathcal{X})$, then the function
\begin{equation}
\gamma(t) = \rho[\chi_{[0,t]}]
\end{equation}
is Lipschitz and satisfies $\dot{\gamma} = \rho$.  Also, $\norm{\dot{\gamma}}_{\mathcal{L}(L^1,\mathcal{X})}$ equals the Lipschitz seminorm of $\gamma$.
\end{observation}

\begin{proof}
Suppose $\gamma: [0,T] \to \mathcal{X}$ is $C$-Lipschitz.  The action of $\dot{\gamma}$ on step functions is defined by $\dot{\gamma}[\chi_{[a,b]}] = \gamma(b) - \gamma(a)$.  For any step function $\phi$, we have $\norm{\dot{\gamma}[\phi]} \leq C \norm{\phi}_{L^1[0,T]}$, hence the $\dot{\gamma}$ extends to bounded linear map $L^1[0,T] \to \mathcal{X}$.  The other claims are left as exercises.
\end{proof}

The following fact will be handy for proving identities and estimates involving distributional derivatives.

\begin{lemma} \label{lem:smallerror}
If $\rho \in \mathcal{L}(L^1[0,T], \mathcal{X})$, then
\begin{equation}
\norm{\rho}_{\mathcal{L}(L^1, \mathcal{X})} = \sup_{0 \leq a < b \leq T} \frac{\norm{\rho[\chi_{a,b}]}}{b - a} = \lim_{\epsilon \to 0} \sup_{0 < b - a \leq \epsilon} \frac{\norm{\rho[\chi_{a,b}]}}{b - a}.
\end{equation}
As a consequence, if $\rho$ and $\tilde{\rho}$ are bounded maps $L^1[0,T] \to \mathcal{X}$ and $\rho[\chi_{[a,b]}] = \tilde{\rho}[\chi_{[a,b]}] + o(|b - a|)$, then $\rho = \tilde{\rho}$.
\end{lemma}

\begin{proof}
The nontrivial part of the proof is to show that
\begin{equation}
\norm{\rho}_{\mathcal{L}(L^1, \mathcal{X})} \leq \liminf_{\epsilon \to 0} \sup_{0 < b - a \leq \epsilon} \frac{\norm{\rho[\chi_{a,b}]}}{b - a}.
\end{equation}
If $C$ is the right hand side, then it is sufficient to show that $\norm{\rho[\phi]} \leq C \norm{\phi}_{L^1[0,T]}$ when $\phi$ is continuous.  This can be proved by approximating $\phi$ uniformly by a sequence of step functions, such that mesh size of the partition also approaches zero.
\end{proof}

\begin{remark}
Note that by the previous lemma and some basic results on $L_{\Boch}^\infty$, there is an isometric inclusion $\iota: L_{\Boch}^\infty([0,T], \mathcal{X}) \to \mathcal{L}(L^1[0,T], \mathcal{X})$ given by
\[
\iota(\rho): \phi \mapsto \int_0^T \rho(t) \phi(t)\,dt,
\]
for $\rho \in L_{\Boch}^\infty([0,T],\mathcal{X})$, so in the sequel we will regard $L_{\Boch}^\infty([0,T],\mathcal{X})$ as a subspace of $\mathcal{L}(L^1[0,T],\mathcal{X})$.
\end{remark}

If we had a bounded function $R: [0,T] \times [0,T] \to \mathcal{X}$ denoted $R(s,t)$, then could define the diagonal restriction $R(t,t)$.  We claim that under appropriate hypotheses, this operation still makes sense when $R(s,\cdot)$ is an element of $\mathcal{L}(L^1[0,T], \mathcal{X})$ rather than a bounded function $[0,T] \to \mathcal{X}$.  For this to be rigorous, we must view $R$ as a map $[0,T] \to \mathcal{L}(L^1[0,T], \mathcal{X})$.

\begin{lemma}[Diagonal restriction] \label{lem:diagonal}
There exists a unique linear map
\[
\diag: L_{\Boch}^\infty([0,T], \mathcal{L}(L^1[0,T], \mathcal{X})) \to \mathcal{L}(L^1[0,T],\mathcal{X})
\]
such that
\begin{enumerate}[(1)]
	\item If $R(s) = \sum_{j=1}^\infty \chi_{E_j}(s) \cdot \rho_j$ where the sets $E_j$ are disjoint measurable sets and $\sup_j \norm{\rho_j}_{\mathcal{L}(L^1[0,T],\mathcal{X})} < +\infty$, and if $\phi \in L^1[0,T]$, we have
	\begin{equation} \label{eq:diagsimple}
	(\diag R)[\phi] = \sum_{j=1}^\infty \rho_j[\chi_{E_j} \phi].
	\end{equation}
	\item We have
	\begin{equation} \label{eq:stupidestimate0}
	\norm{\diag R}_{\mathcal{L}(L^1[0,T], \mathcal{X})} \leq \norm{R}_{L_{\Boch}^\infty([0,T],\mathcal{L}(L^1[0,T], \mathcal{X}))}.
	\end{equation}
\end{enumerate}
Furthermore, this map $\diag$ satisfies the estimate
\begin{equation} \label{eq:stupidestimate}
\norm*{ (\diag R)[\phi] } \leq \int_0^T |\phi(t)| \norm{R(t,\cdot)}_{\mathcal{L}(L^1[0,T],\mathcal{X})}\,dt.
\end{equation}
\end{lemma}

\begin{proof}
For a simple function $R$, we can define $\diag R$ unambiguously by \eqref{eq:diagsimple}, that is, it is independent of the decomposition of the simple function.  We check that \eqref{eq:stupidestimate} and hence \eqref{eq:stupidestimate0} hold for simple functions.  Then \eqref{eq:stupidestimate0} implies that $\diag R$ has a unique extension to $L_{\Boch}^\infty([0,T], \mathcal{L}(L^1[0,T], \mathcal{X}))$.  The inequality \eqref{eq:stupidestimate} extends to $L_{\Boch}^\infty([0,T], \mathcal{L}(L^1[0,T], \mathcal{X}))$ because both sides are continuous in the Bochner $L^\infty$ norm.
\end{proof}

In the rest of the paper, we will often use more suggestive notation which treats the elements of $\mathcal{L}(L^1[0,T], \mathcal{X})$ like pointwise defined functions.  Although using function notation for distributions has some drawbacks, the ultimate benefit will be a more intuitive statement of identities such Lemma \ref{lem:chainrule} below, and more generally a compact notation for constructing and transforming such distributions.

\begin{notation}
For a function $\rho \in \mathcal{L}(L^1[0,T], \mathcal{X})$, we will use the notation $\rho(t)$ where $t$ is formal or ``dummy'' variable.  For $\phi \in L^1[0,T]$, we will denote
\begin{equation}
\int_0^T \phi(t) \rho(t)\,dt := \rho[\phi]
\end{equation}
as well as
\begin{equation}
\int_a^b \rho(t)\,dt := \rho[\chi_{[a,b]}].
\end{equation}
\end{notation}

To obviate potential confusion, when we apply $\rho$ as a linear map to a function $\phi$ in $L^1[0,T]$, we will use square brackets and not write the dummy variable $t$.  For instance, the application of $\rho$ to the identity function $t$ on $[0,T]$ would be denoted by $\rho[\id_{[0,T]}]$ or $\int_0^T \rho(t) t\,dt$ and \emph{not} by $\rho(t)$ or $\rho[t]$.  Similarly, $\rho(2t)$ would denote the element of $\mathcal{L}(L^1[0,T/2], \mathcal{X})$ defined by
\[
\int_0^{T/2} \rho(2t) \phi(t)\,dt := \frac{1}{2} \int_0^T \rho(t) \phi(2t)\,dt
\]
but on the other hand $\int \rho(t) \cdot 2t \,dt$ would denote the application of $\rho$ as a linear map to the function $2t$ on $[0,T]$.

\begin{notation}
If $R$ is in $L_{\Boch}^\infty([0,T],\mathcal{L}(L^1[0,T], \mathcal{X}))$, then we will write $R$ formally as a function of two variables $(s,t)$, where the $s$ corresponds to the first ``$[0,T]$'' and the $t$ corresponds to the second ``$[0,T]$'' in ``$L_{\Boch}^\infty([0,T],\mathcal{L}(L^1[0,T], \mathcal{X}))$''; in other words, the distributional dependence occurs in the second variable $t$.  We will denote $(\diag R)(t)$ as $R(t,t)$.
\end{notation}

Thus, for example, if $R(s,t) = \sum_{j=1}^\infty \chi_{E_j}(s) \rho_j(t)$, then \eqref{eq:diagsimple} becomes
\begin{equation}
\int_0^T \phi(t) R(t,t)\,dt = \sum_{j=1}^\infty \int_{E_j} \phi(t) \rho_j(t)\,dt,
\end{equation}
and hence in a formal sense
\begin{equation}
R(t,t) = \sum_{j=1}^\infty \chi_{E_j}(t) \rho_j(t).
\end{equation}
Also, \eqref{eq:stupidestimate} becomes
\begin{equation} \label{eq:stupidestimate2}
\norm*{ \int_0^T \phi(t) R(t,t)\,dt } \leq \int_0^T |\phi(t)| \norm{R(t,\cdot)}_{\mathcal{L}(L^1[0,T],\mathcal{X})}\,dt.
\end{equation}

We will mainly use two special cases of the diagonal restriction.

\begin{definition} \label{def:multiplication}
Suppose that $\rho \in \mathcal{L}(L^1[0,T], \mathcal{X})$ and $A \in L_{\Boch}^\infty([0,T], \mathcal{L}(\mathcal{X}, \mathcal{Y}))$.  Then we define $(A \rho)(t) = A(t) \rho(t)$ in $\mathcal{L}(L^1[0,T], \mathcal{Y})$ as the diagonal restriction $R(t,t)$ of the function $R(s,t) = A(s) \rho(t)$, which is in $L_{\Boch}^\infty([0,T], \mathcal{L}(L^1[0,T], \mathcal{Y}))$.
\end{definition}

\begin{observation}~ \label{obs:multiplicationproperties}
\begin{enumerate}[(1)]
	\item The product $A \cdot \rho$ defined above is bilinear in $A$ and $\rho$.
	\item If $I$ is a subinterval of $[0,T]$, then we have $A|_I \cdot \rho|_I = (A \rho)|_I$.
	\item $\norm{A \cdot \rho}_{\mathcal{L}(L^1[0,T], \mathcal{Y})} \leq \norm{A}_{L_{\Boch}^\infty([0,T], \mathcal{L}(\mathcal{X}, \mathcal{Y}))} \norm{\rho}_{\mathcal{L}(L^1[0,T], \mathcal{X})}$.
\end{enumerate}
\end{observation}

\begin{definition} \label{def:composition}
Suppose that $W$ is a metric space, $F(w,t)$ is a continuous map $W \to \mathcal{L}(L^1[0,T], \mathcal{X})$, and $w: [0,T] \to W$ is continuous.  Then $R(s,t) = F(w(s),t)$ is a continuous map $[0,T] \to \mathcal{L}(L^1[0,T], \mathcal{X})$.  We define $F(w(t),t)$ to be the diagonal restriction of $R$.
\end{definition}

\begin{observation} \label{obs:compositionerror}
Suppose that $F$ is uniformly continuous as a map $W \to \mathcal{L}(L^1[0,T], \mathcal{X})$ with modulus of continuity $\omega_F$, and let $d_\infty$ be the supremum metric on $C([0,T], W)$.  Then for $w, \tilde{w} \in C([0,T],W)$, we have
\begin{align*}
\norm{F(w(t),t) - F(\tilde{w}(t),t)}_{\mathcal{L}(L^1[0,T],\mathcal{X})_{dt}} &\leq \norm{F \circ w - F \circ \tilde{w}}_{L_{\Boch}^\infty([0,T], \mathcal{L}(L^1[0,T], \mathcal{X}))} \\
&\leq \omega_F( d_\infty(w,\tilde{w}) ).
\end{align*}
\end{observation}

\subsection{Locally Lipschitz Families} \label{subsec:LLF}

\begin{definition}[Locally Lipschitz Family] \label{def:Lipschitzfamily}
Let $\mathcal{X}$ and $\mathcal{Y}$ be Banach spaces.  Let $\mathcal{U} \subseteq \mathcal{X}$ be open and $T > 0$.  A map $F: \mathcal{U} \times [0,T] \to \mathcal{Y}$ for $t \in [0,T]$ is called a \emph{locally Lipschitz family of analytic maps} if $F(\cdot,t)$ is analytic for each $t$, and for each $x_0 \in \mathcal{U}$ there exist $r > 0$ and $L > 0$ such that
\begin{equation} \label{eq:LLF}
\norm{f(x,s) - f(x,t)} \leq L|s - t| \text{ for all } s, t \in [0,T] \text{ for all } x \in B_{\mathcal{X}}(x_0,r).
\end{equation}
\end{definition}

Here the word ``locally'' refers to the variable $x$ but ``Lipschitz'' refers to the variable $t$, analytic functions being automatically locally Lipschitz in the space variable by \eqref{eq:aprioriLipschitz}.

\begin{lemma}[Differentiation and Integration]~
\begin{enumerate}[(1)]
	\item Let $\partial_t F(x,\cdot) \in \mathcal{L}(L^1[0,T],\mathcal{Y})$ denote the distributional time derivative.  Then $x \mapsto \partial_t F(x,\cdot)$ is an analytic map $\mathcal{U} \to \mathcal{L}(L^1[0,T], \mathcal{Y})$.
	\item Conversely, if $f: \mathcal{U} \to \mathcal{L}(L^1[0,T], \mathcal{Y})$ is analytic, then we can define a locally Lipschitz family of analytic functions $F(x,t)$ by
	\begin{equation*}
	F(x,t) = \int_0^t f(x,s)\,ds.
	\end{equation*}
\end{enumerate}
\end{lemma}

\begin{proof}
For (1), to check that $x \mapsto \partial_t F(x,\cdot)$ is locally bounded, suppose $x_0 \in \mathcal{X}$.  Then there exist $L$ and $R > 0$ such that $\norm{F(x,s) - F(x,t)} \leq L|s - t|$ for $\norm{x - x_0} \leq R$. This implies that $\norm{\partial_t F(x,t)}_{\mathcal{L}(L^1[0,T],\mathcal{Y})} \leq L$.

To prove analyticity of $x \mapsto \partial_t F(x,\dot)$, it suffices to show that $\int_0^T \phi(t) \partial_t F(x,t)\,dt$ is an analytic function $\mathcal{X} \to \mathcal{Y}$ for each $\phi \in L^1[0,T]$.  This clearly holds when $\phi$ is a step function, and hence it holds for all $\phi \in L^1[0,T]$ by approximation.

The verification of (2) is left to the reader (see Observation \ref{obs:weakderivative}).
\end{proof}

\begin{lemma}[Mixed Partials]
Let $F: \mathcal{U} \times [0,T] \to \mathcal{Y}$ be a locally Lipschitz family and let $h \in \mathcal{X}$.  Then $\delta^k F(x,t;h)$ is a locally Lipschitz family and we have $\partial_t \delta^k F(\cdot,\cdot;h) = \delta^k \partial_t F(\cdot,\cdot;h)$ for every $k > 0$.
\end{lemma}

\begin{proof}
The fact that $\delta^kF(x,t;h)$ is a locally Lipschitz family follows from Theorem \ref{thm:analytic}.  The equality
\begin{equation}
\int_0^T \phi \partial_t D^k F\,dt = \int_0^T \phi D^k \partial_t F\,dt
\end{equation}
is immediate when $\phi$ is the indicator function of an interval, hence holds when $\phi$ is a step function, and therefore it holds for every $\phi \in L^1[0,T]$ by density.
\end{proof}

\begin{lemma}[Chain Rule] \label{lem:chainrule}
Let $\mathcal{X}$, $\mathcal{Y}$, $\mathcal{Z}$ be Banach spaces and let $\mathcal{U} \subseteq \mathcal{X}$ and $\mathcal{V} \subseteq \mathcal{Y}$ be open.  Let $F: \mathcal{U} \times [0,T] \to \mathcal{V} \subseteq \mathcal{Y}$ and $G: \mathcal{V} \times [0,T] \to \mathcal{Z}$ be locally Lipschitz families.  Then $F(G(x,t),t)$ is a locally Lipschitz family.  Moreover,
\begin{equation} \label{eq:chainrule}
\partial_t [F(G(x,t),t))] = DF(G(x,t),t)[\partial_t G(x,t)] + \partial_t F(G(x,t),t).
\end{equation}
Here $DF(G(x,t),t)[\partial_t G(x,t)]$ is given by Definition \ref{def:multiplication} with $A(t) = DF(G(x,t),t)$ and $\rho(t) = \partial_t F(x,t)$.  The other term $\partial_t F(G(x,t),t)$ is given by Definition \ref{def:composition} by taking $W$ to be an appropriate open subset of $\mathcal{V}$ and setting $w(t) = G(x,t)$.
\end{lemma}


\begin{proof}
Using the a priori Lipschitz estimate for analytic functions \eqref{eq:aprioriLipschitz} together with equation \eqref{eq:LLF}, we see that for each $x_0 \in \mathcal{X}$, there exist $R$, $L_1$, and $L_2 > 0$ such that
\begin{multline}
\norm{F(x,t) - F(x',t')} \leq L_1\norm{x - x'} + L_2|t - t'| \\
\text{ for } x, x' \in B_\mathcal{X}(x_0,R) \text{ and } t, t' \in [0,T].
\end{multline}
From here, a straightforward argument using compactness of $[0,T]$ shows that there exist constants such that
\begin{multline}
\norm{F(G(x,s),t) - F(G(x',s'),t')} \leq L_1^* \norm{x - x'} + L_2^* |s - s'| + L_3^*|t - t'| \\ \text{ for } x, x' \in B_\mathcal{X}(x_0,R^*).
\end{multline}
In particular, by taking $s = t$ and $s' = t'$, we see that $F(G(x,t),t)$ is a locally Lipschitz family.

Now let us prove the chain rule identity \eqref{eq:chainrule}.  Fix $x_0$ and consider an interval $[a,b] \subseteq [0,T]$.  Then 
\begin{align}
F&(G(x_0,b),b) - F(G(x_0,a),a) \nonumber \\
=& [F(G(x_0,b),b) - F(G(x_0,a),b)] + [F(G(x_0,a),b) - F(G(x_0,a),a)] \nonumber \\
=& DF(G(x_0,a),b)[G(x_0,b) - G(x_0,a)] + [F(G(x_0,a),b) - F(G(x_0,a),a)] \nonumber \\
&\qquad + O(|b - a|^2)
\end{align}
In other words,
\begin{align}
\int_a^b &\partial_t [F(G(x_0,t),t))]\,dt \nonumber \\
&= \int_a^b DF(G(x_0,a),b) \partial_t G(x_0,t)\,dt + \int_a^b \partial_t F(G(x_0,a),t)\,dt + O(|b - a|^2). \label{eq:chainrulecomputation1}
\end{align}
Consider the first term on the right-hand side of \eqref{eq:chainrulecomputation1}.  We showed earlier that $F(G(x,s),t)$ is Lipschitz with respect to $(s,t)$ for $x$ in an open neighborhood of $x_0$, and the same holds for $DF(G(x,s),t)$ by \eqref{eq:Cauchyestimate}.  Therefore,
\begin{equation}
\sup_{t \in [a,b]} \norm{DF(G(x,a),b) - DF(G(x,t),t)} = O(|b - a|).
\end{equation}
By Observation \ref{obs:multiplicationproperties}, this implies
\begin{equation} \label{eq:chainrulecomputation2}
\int_a^b DF(G(x_0,a),b) \partial_t G(x_0,t)\,dt = \int_a^b DF(G(x_0,t),t) \partial_t G(x_0,t)\,dt + O(|b - a|^2),
\end{equation}
where the error bound comes from multiplying $O(|b - a|)$ by $\norm{\chi_{[a,b]}}_{L^1}$.

Now consider the second term on the right hand side of \eqref{eq:chainrulecomputation1}.  Because $\partial_t G$ is an analytic function $\mathcal{V} \to \mathcal{L}(L^1[0,T],\mathcal{Z})$, we have
\begin{equation}
\norm{\partial_t F(y, \cdot) - \partial_t F(y',\cdot)}_{\mathcal{L}(L^1[0,T],\mathcal{Z})} \leq C \norm{y - y'}
\end{equation}
for $y$ and $y'$ in an open neighborhood $\mathcal{W}$ of $F(x_0,a)$.  Thus, by applying Observation \ref{obs:compositionerror} on the interval $[a,b]$, we have that if $b$ is sufficiently small, then for all $t \in [a,b]$,
\begin{align}
\norm{\partial_t F(G(x_0,t),t) - \partial_t F(G(x_0,a),t)}_{\mathcal{L}(L^1[a,b],\mathcal{Z})_{dt}} &\leq C \sup_{t \in [a,b]} \norm{G(x_0,t) - G(x_0,a)} \nonumber \\
&= O(|b - a|),
\end{align}
where the error estimate is independent of $t$ and only depends on $x_0$, $a$, $F$, and $G$.  Hence,
\begin{equation}
\int_a^b \partial_t F(G(x_0,a),t)\,dt = \int_a^b \partial_t F(G(x_0,t),t)\,dt + O(|b - a|^2).  \label{eq:chainrulecomputation3}
\end{equation}

Overall, by substituting \eqref{eq:chainrulecomputation2} and \eqref{eq:chainrulecomputation3} into \eqref{eq:chainrulecomputation1}, we have
\begin{align}
\int_a^b &\partial_t [F(G(x_0,t),t))]\,dt \nonumber \\
&= \int_a^b DF(G(x_0,t),t) \partial_t G(x_0,t)\,dt + \int_a^b \partial_t F(G(x_0,t),t)\,dt + O(|b - a|^2).
\end{align}
By Lemma \ref{lem:smallerror}, the chain rule identity holds.
\end{proof}

Next, we will check that the two terms on the right hand side of the chain rule equation \eqref{eq:chainrule} are analytic functions $\mathcal{U} \to \mathcal{L}(L^1[0,T], \mathcal{Z})$.

\begin{lemma} \label{lem:multiplicationanalytic}
Let $g$ be an analytic function $\mathcal{U} \to \mathcal{L}(L^1[0,T], \mathcal{Y})$, and let $A: \mathcal{U} \times [0,T] \to \mathcal{L}(\mathcal{Y}, \mathcal{Z})$ be a locally Lipschitz analytic family.  Then $A(x,t) g(x,t)$ is an analytic function $\mathcal{U} \to \mathcal{L}(L^1[0,T], \mathcal{Z})$.
\end{lemma}

\begin{proof}
Fix $x_0 \in \mathcal{U}$.  Then there exists $R > 0$ such that for $x \in B_\mathcal{X}(x_0,R)$ we have
\begin{align}
\norm{A(x,t)} &\leq C_1 \\
\norm{A(x,t) - A(x,t')} &\leq L|t - t'| \\
\norm{g(x,\cdot)}_{\mathcal{L}(L^1[0,T], \mathcal{Z})} &\leq C_2.
\end{align}
Consider a partition $P$ given by $0 = t_0 < t_1 < \dots < t_m = T$, and let
\begin{equation}
h_P(x,t) = \sum_{j=1}^m A(x,t_{j-1}) \chi_{[t_{j-1},t_j)}(t) g(x,t)
\end{equation}
is a function $B_\mathcal{X}(x_0,R) \to \mathcal{L}(L^1[0,T], \mathcal{Z})$ which is bounded by $C_1C_2$. One checks easily that $\int_0^T h_P(x,t) \phi(t)\,dt$ is analytic for each $\phi \in L^1[0,T]$ and hence $x \mapsto h_P(x,\cdot)$ is analytic.  On the other hand, we have
\[
\norm{A(x,t) g(x,t) - h_P(x,t)}_{\mathcal{L}(L^1[0,T], \mathcal{Z})_{dt}} \leq LC_2 \max |t_j - t_{j-1}|,
\]
which can be made arbitrarily small by choosing a fine enough partition.  Thus, $A(x,t) g(x,t)$ is the uniform limit of a sequence of analytic functions on $B_\mathcal{X}(x_0,R)$ and hence is analytic.
\end{proof}

\begin{lemma} \label{lem:compositionanalytic}
Let $G: \mathcal{U} \times [0,T] \to \mathcal{V} \subseteq \mathcal{Y}$ be a locally Lipschitz analytic family, and let $\theta: \mathcal{V} \to \mathcal{L}(L^1[0,T], \mathcal{Z})$ be analytic.  Then $\theta(G(x,t),t)$ defines an analytic function $\mathcal{U} \to \mathcal{L}(L^1[0,T], \mathcal{Z})$.
\end{lemma}

\begin{proof}
The argument is similar to the previous lemma.  We use the approximation
\begin{equation}
h_P(x,t) = \sum_{j=1}^m \chi_{[t_{j-1},t_j)}(t) f(G(x,t_{j-1}),t),
\end{equation}
and each term in the sum is analytic because it is the composition of two analytic functions.
\end{proof}

\begin{remark}
We suspect that the results of this section may have other applications (e.g.\ to PDE), even the author did not find these results in the literature.
\end{remark}

\section{$C^*$-valued Non-Commutative Probability} \label{sec:NCP}

\subsection{Basics of $C^*$-algebras} \label{subsec:Cstaralgebras}

For the readers who are not familiar with operator algebras, we review without proof some standard results in the theory of $C^*$-algebras.  We refer to Blackadar \cite[Chapter II]{Blackadar2006} for an encyclopedic list of results, proof sketches, and references.

\begin{definition} \label{def:Cstaralgebra}
A \emph{$C^*$-algebra} $\mathcal{A}$ is a Banach space equipped with a multiplication operation and a conjugate-linear involution $a \mapsto a^*$ such that
\[
\norm{xy} \leq \norm{x} \norm{y}
\]
and
\[
\norm{x^*x} = \norm{x}^2 = \norm{x^*}^2.
\]
A $C^*$-algebra is called \emph{unital} if it has a multiplicative identity $1 \neq 0$.  In this case, we have $1^* = 1$ automatically.
\end{definition}

\begin{definition}
Given two $C^*$-algebras, $\mathcal{A}$ and $\mathcal{B}$, a $*$-homomorphism $\pi: \mathcal{A} \to \mathcal{B}$ is $\C$-linear map which respects the multiplication and $*$-operations.  If $\mathcal{A}$ and $\mathcal{B}$ are unital, then we say $\pi$ is \emph{unital} if $\pi(1) = 1$.
\end{definition}

\begin{definition}
We denote by $B(\mathcal{H})$ the space of bounded operators on a Hilbert space $\mathcal{H}$, which is a unital $C^*$-algebra with respect to the operator norm and the involution given by taking the adjoint.
\end{definition}

\begin{theorem} ~ \label{thm:Cstarbasics}
Let $\mathcal{A}$ and $\mathcal{B}$ be $C^*$-algebras.
\begin{enumerate}
	\item If $\pi: \mathcal{A} \to \mathcal{B}$ is a $*$-homomorphism, then $\norm{\pi(x)} \leq \norm{x}$.
	\item If $\pi: \mathcal{A} \to \mathcal{B}$ is an injective $*$-homomorphism, then $\pi$ is an isometry.
	\item For every $C^*$-algebra $\mathcal{A}$, there exists a Hilbert space $\mathcal{H}$ and an injective $*$-homomorphism $\mathcal{A} \to B(\mathcal{H})$.
\end{enumerate}
\end{theorem}

As a consequence of (1) and (2), if $\mathcal{A}$ is a $C^*$-algebra, then there is only one norm on $\mathcal{A}$ satisfying Definition \ref{def:Cstaralgebra}.

We claim that if $\mathcal{A}$ is a $C^*$-algebra, then $M_n(\mathcal{A})$ is also a $C^*$-algebra.  Note that $M_n(\mathcal{A})$ can naturally by identified with $\mathcal{A} \otimes M_n(\C)$ as a vector space.  We equip $M_n(\mathcal{A})$ with the multiplication and $*$-operations given by tensoring those of $\mathcal{A}$ with those of $M_n(\C)$.

We define the norm on $M_n(\mathcal{A})$ by representing it as an algebra of operators on a Hilbert space.  By Theorem \ref{thm:Cstarbasics} (3), there is an injective $*$-homomorphism $\pi: \mathcal{A} \to B(\mathcal{H})$ for some Hilbert space $\mathcal{H}$.  Moreover, matrix multiplication defines a $*$-isomorphism $\rho_n: M_n(\C) \to B(\C^n)$.  Then $\pi \otimes \rho_n$ defines a $*$-homomorphism $\mathcal{A} \otimes M_n(\C) \to B(\mathcal{H} \otimes \C^n)$ or in other words $M_n(\mathcal{A}) \to B(\mathcal{H}^{\oplus n})$.  If $A = (a_{i,j}) \in M_n(\mathcal{A})$, then we define $\norm{A}$ to be the operator norm of $\pi \otimes \rho_n(A)$.  One can check that $\max_{i,j} \norm{a_{i,j}} \leq \norm{A} \leq \sum_{i,j} \norm{a_{i,j}}$ and hence $M_n(\mathcal{A})$ is already complete in this norm and hence is a $C^*$-algebra.  Moreover, as remarked earlier, the norm on $M_n(\mathcal{A})$ is unique.

\begin{definition}
An element $a$ in a $C^*$-algebra $\mathcal{A}$ is \emph{positive} if $a = b^*b$ for some $b \in \mathcal{A}$, and in this case we write $a \geq 0$.  A bounded linear functional $\phi: \mathcal{A} \to \C$ is called \emph{positive} if $\phi(a^*a) \geq 0$ for every $\mathcal{a} \in \mathcal{A}$.  A \emph{state} is a positive linear functional with $\norm{\phi} = 1$.
\end{definition}

A state is viewed as a non-commutative analogue of a Borel probability measure on a locally compact Hausdorff space.  As motivation, note that if $\Omega$ is a compact Hausdorff space, then $C(\Omega)$ is a $C^*$-algebra (using the supremum norm and complex conjugation).  Moreover, states on $C(\Omega)$ are equivalent to Borel probability measures on $\Omega$.  Thus, a $C^*$-algebra $\mathcal{A}$ with a state $\phi$ can be viewed as a non-commutative analogue of a probability space.

The following theorem shows that every state can be represented concretely by taking the inner product with some vector $\xi$ in some representation of $\mathcal{A}$.  Concrete representations of $C^*$-algebras and states are an important tool in non-commutative probability, and since we will use the operator-valued version of this technique in \S \ref{subsec:Fockspace1} - \ref{subsec:Fockspaceindependence}, we include for motivation a brief sketch of the scalar-valued case here and of the operator-valued case in the next subsection.

\begin{theorem} \label{thm:scalarGNS}
Let $\mathcal{A}$ be a unital $C^*$-algebra.  If $\pi: \mathcal{A} \to B(\mathcal{H})$ is a $*$-homomorphism and $\xi$ is a vector in $\mathcal{H}$, then $\phi(a) = \ip{\xi, \pi(a)\xi}$ is a positive linear functional.  Conversely, every positive linear functional can be represented in this form for some representation $\pi$ and unit vector $\xi$.
\end{theorem}

The nontrivial direction is the converse.  Suppose that $\phi$ is a state on $\mathcal{A}$.  For $a, a' \in \mathcal{A}$, define $\ip{a,a'}_\phi = \phi(a^*a')$.  Then $\ip{\cdot,\cdot}$ is a pre-inner product and thus satisfies the Cauchy-Schwarz inequality.  Thus, $\mathcal{A} / \{a: \phi(a^*a) = 0\}$ can be completed to a Hilbert space $\mathcal{H}_\phi$.  We want to define $\pi_\phi: \mathcal{A} \to B(\mathcal{H}_\phi)$ by $\pi(a) b = [ab]$.  To show that the left multiplication action of $\mathcal{A}$ on itself produces a bounded action on the completed quotient $\mathcal{H}_\phi$, it suffices to show that $\ip{ab,ab}_\phi \leq \norm{a}^2 \ip{b,b}_\phi$.  Now $\norm{a}^2 - a^*a$ is positive and hence $b^*(\norm{a} - a^*a)b$ is positive, so that $\phi(b^*(\norm{a}^2 - a^*a) b) \geq 0$ and thus $\phi(b^*a^*ab) \leq \norm{a}^2 \phi(b^*b)$ as desired.  Thus, $\pi_\phi$ is well-defined.  Letting $\xi$ be the vector $[1] \in \mathcal{H}_\phi$, we have
\[
\ip{\xi, \pi_\phi(a)\xi} = \phi(a),
\]
which is the desired representation of $\phi$.  This procedure is known as the \emph{Gelfand-Naimark-Segal (GNS) construction}.  In the theorem above, note that $\phi$ is a state if and only if $\xi$ is a unit vector.

\subsection{$\mathcal{A}$-valued Probability Spaces} \label{subsec:Avaluedprobability}

Let $\mathcal{A}$ be a given unital $C^*$-algebra.  Then $\mathcal{A}$-valued non-commutative probability is, roughly speaking, an analogue of non-commutative probability theory in which the scalars $\C$ are replaced by the algebra $\mathcal{A}$.  Thus, we study $\mathcal{A}$-valued ``positive functionals'' and ``states'' on $\mathcal{B}$ and even ``Hilbert spaces'' with $\mathcal{A}$-valued inner products.  The appropriate replacement for positivity in this context is complete positivity.

\begin{definition}
We say a linear map $\Phi: \mathcal{B} \to \mathcal{A}$ is \emph{positive} if $b \geq 0$ implies $\Phi(b) \geq 0$.  Given a map $\Phi: \mathcal{B} \to \mathcal{A}$, we denote by $\Phi^{(n)}: M_n(\mathcal{B}) \to M_n(\mathcal{A})$ the function given by applying $\Phi$ entrywise.  We say that $\Phi$ is \emph{completely positive} if $\Phi^{(n)}$ is positive for every $n$.
\end{definition}

We next define the $\mathcal{A}$-valued version of a Hilbert space, which is a right Hilbert $\mathcal{A}$-module with an $\mathcal{A}$-valued inner product.  For background, see \cite{Lance1995} and \cite[\S II.7.1 - II.7.2]{Blackadar2006} and the references therein.  Just as in the scalar case a right Hilbert $\mathcal{A}$-modules can be constructed from a space with a pre-inner product by taking a completed quotient.

\begin{definition}
Let $\mathcal{A}$ be a unital $C^*$-algebra.  If $\mathcal{H}$ is a right $\mathcal{A}$-module, then an \emph{$\mathcal{A}$-valued pre-inner product} is a map $\ip{\cdot,\cdot}: \mathcal{H} \times \mathcal{H} \to \mathcal{A}$ such that for $h$, $h_1$, $h_2 \in \mathcal{H}$.
\begin{enumerate}
	\item $h_2 \mapsto \ip{h_1,h_2}$ is a right $\mathcal{A}$-module map.
	\item $\ip{h_2,h_1} = \ip{h_1,h_2}^*$.
	\item $\ip{h,h} \geq 0$.
\end{enumerate}
Note that $\mathcal{H}$ is a vector space over $\C$ since $\C \subseteq \mathcal{A}$ for unital $\mathcal{A}$.  For such an $\mathcal{H}$ and $\ip{\cdot,\cdot}$, we say that $\mathcal{H}$ is a \emph{right Hilbert $\mathcal{A}$-module} if $\mathcal{H}$ is a Banach space with respect to $\norm{h} := \norm{\ip{h,h}}_{\mathcal{A}}^{1/2}$.  In this case, we call $\ip{\cdot,\cdot}$ an \emph{$\mathcal{A}$-valued inner product}.
\end{definition}

\begin{lemma} \label{lem:rightHilbertmodule}
Suppose that $\mathcal{H}$ is a right $\mathcal{A}$-module with an $\mathcal{A}$-valued pre-inner product.  Then
\begin{enumerate}
	\item We have the CBS inequality $\ip{h_1,h_2}^*\ip{h_1,h_2} \leq \norm{\ip{h_1,h_1}} \ip{h_2,h_2}$ in $\mathcal{A}$, and in particular,
	$\norm{\ip{h_1,h_2}} \leq \norm{\ip{h_1,h_1}}^{1/2} \norm{\ip{h_2,h_2}}^{1/2}$.
	\item The function $\norm{h} = \norm{\ip{h,h}}^{1/2}$ defines a semi-norm on $\mathcal{H}$.
	\item We have $\norm{ha} \leq \norm{h} \norm{a}$ for $h \in \mathcal{H}$ and $a \in \mathcal{A}$.
	\item The completion of $\mathcal{H} / \{h: \norm{h} = 0\}$ is a right Hilbert $\mathcal{A}$-module with the right $\mathcal{A}$-action and the $\mathcal{A}$-valued inner product induced in the natural way from those of $\mathcal{H}$.
\end{enumerate}
\end{lemma}

Next, we define $B(\mathcal{H})$ for a right Hilbert $\mathcal{A}$-module $\mathcal{H}$.  If $\mathcal{H}_1$ and $\mathcal{H}_2$ be Hilbert $\mathcal{A}$-modules, we say that a linear map $T: \mathcal{H}_1 \to \mathcal{H}_2$ is \emph{right $\mathcal{A}$-linear} if $T(ha) = (Th)a$ for $h \in \mathcal{H}_1$ and $a \in \mathcal{A}$.  We say that $T$ is \emph{adjointable} if there exists a map $T^*: \mathcal{H}_2 \to \mathcal{H}_1$ such that
\[
\ip{Th_1,h_2} = \ip{h_1,T^*h_2} \text{ for all } h_1 \in \mathcal{H}_1 \text{ and } h_2 \in \mathcal{H}_2.
\]
We denote by $B(\mathcal{H})$ the space of bounded, right $\mathcal{A}$-linear, adjointable operators on a right Hilbert $\mathcal{A}$-module $\mathcal{H}$.  One can check that $B(\mathcal{H})$ is a $C^*$-algebra \cite[p.\ 8]{Lance1995}.  We are now ready to give the $\mathcal{A}$-valued version of the GNS construction.

\begin{proposition} \label{prop:operatorvaluedGNS}
Let $\mathcal{A}$ and $\mathcal{B}$ be unital $C^*$-algebras and $\Phi: \mathcal{B} \to \mathcal{A}$ a linear map.  The following are equivalent:
\begin{enumerate}
	\item $\Phi$ is completely positive.
	\item There exists a right Hilbert $\mathcal{A}$-module $\mathcal{H}$ and a $*$-homomorphism $\pi: \mathcal{B} \to B(\mathcal{H})$ such that $\Phi(b) = \ip{\xi, \pi(b)\xi}$.
\end{enumerate}
\end{proposition}

\begin{proof}[Sketch of proof]
(2) $\implies$ (1).  To see that $\Phi$ is positive, note that $\Phi(b^*b) = \ip{\pi(b) \xi, \pi(b) \xi} \geq 0$.  By considering $\pi^{(n)}: M_n(\mathcal{B}) \to B(\mathcal{H}^{\oplus n})$ one can show that for $B \in M_n(\mathcal{B})$ and $v = [a_1, \dots, a_n]^t$, we have $v^* \Phi(B^*B) v \geq 0$.  Finally, one argues that if this holds for all $v$, then $\Phi^{(n)}(B^*B) \geq 0$.

(1) $\implies$ (2).  We define an $\mathcal{A}$-valued pre-inner product on $\mathcal{B} \otimes_{\alg} \mathcal{A}$ by
\[
\ip{b_1 \otimes a_1, b_2 \otimes a_2} = a_1^* \Phi(b_1^*b_2) a_2.
\]
To check positivity, consider a sum of simple tensors $h = \sum_{j-1}^n b_j \otimes a_j$.  Note that
\[
\ip{h,h} = \begin{bmatrix} a_1^* & \dots & a_n^* \end{bmatrix} \Phi^{(n)} \left( \begin{bmatrix} b_1^*b_1 & \dots & b_1^*b_n \\ \vdots & \ddots & \vdots \\ b_n^*b_1 & \dots & b_n^*b_n \end{bmatrix} \right) \begin{bmatrix} a_1 \\ \vdots \\ a_n \end{bmatrix}.
\]
The matrix $(b_i^*b_j)_{i,j}$ is positive in $M_n(\mathcal{B})$ since it equals $[b_1, \dots, b_n]^* [b_1, \dots, b_n]$.  Thus, by complete positivity, $(\Phi(b_i^*b_j))_{i,j}$ is positive in $M_n(\mathcal{A})$ and this implies positivity of $\ip{h,h}$.  We then define the space $\mathcal{B} \otimes_\Phi \mathcal{A}$ to the completed quotient of $\mathcal{B} \otimes_{\alg} \mathcal{A}$ with respect to this pre-inner product.

For $b \in \mathcal{B}$, we want to define $\pi(b) \in B(\mathcal{H})$ by $\pi(b) b' \otimes a = bb' \otimes a$.  To see that this is well-defined, it suffices to show that $\norm{bh} \leq \norm{b} \norm{h}$ for $h \in \mathcal{B} \otimes_{\alg} \mathcal{A}$.  Since $\norm{b}^2 - b^*b$ is positive, we can write it as $\norm{b}^2 - b^*b = c^*c$.  But $\ip{h,c^*ch} = \ip{ch,ch} \geq 0$ and hence $\norm{\ip{bh,bh}} \leq \norm{b}^2 \norm{\ip{h,h}}$.  Letting $\xi = 1 \otimes 1 \in \mathcal{B} \otimes_\Phi \mathcal{A}$, we get $\Phi(b) = \ip{\xi, \pi(b) \xi}$.
\end{proof}

\begin{remark}
If $\Phi$ is completely positive and is represented as $\ip{\xi, \pi(\cdot)\xi}$ as in Proposition \ref{prop:operatorvaluedGNS}, then
\[
\norm{\Phi(b)} = \norm{\ip{\xi, \pi(b)\xi}} \leq \norm{b} \norm{\xi}^2 = \norm{b} \norm{\Phi(1)}.
\]
Because the same reasoning can be applied to $\Phi^{(n)}$ and $\Phi^{(n)}(1) = 1 \otimes 1_n$, we see that for $z \in M_n(\mathcal{B})$, we have
\begin{equation} \label{eq:CPmapbound}
\norm{\Phi^{(n)}(z)} \leq \norm{z} \norm{\Phi(1)},
\end{equation}
which is an estimate we will use frequently in the rest of the paper.
\end{remark}

Having described the $\mathcal{A}$-valued analogue of positive linear functionals, we now turn to the $\mathcal{A}$-valued analogue of states (and of probability measures).

\begin{definition}
Let $\mathcal{A} \subseteq \mathcal{B}$ be a unital inclusion of $C^*$-algebras.  A map $E: \mathcal{B} \to \mathcal{A}$ is called a \emph{conditional expectation} or \emph{$\mathcal{A}$-valued expectation} if
\begin{enumerate}[(1)]
	\item $E$ is completely positive;
	\item $E|_{\mathcal{A}} = \id$;
	\item $E$ is an $\mathcal{A}$-bimodule map, that is, $E(a_1ba_2) = a_1 E(b) a_2$ whenever $a_1$ and $a_2$ are in $\mathcal{A}$.
\end{enumerate}
An \emph{$\mathcal{A}$-valued probability space} is a pair $(\mathcal{B},E)$ where $\mathcal{B} \subseteq \mathcal{A}$ is a unital $C^*$-algebra and $E: \mathcal{B} \to \mathcal{A}$ is a conditional expectation (note that the inclusion map $\mathcal{A} \to \mathcal{B}$ is implicitly part of the data).
\end{definition}

Property (3) is an analogue of the property in classical probability theory that $E[f(X) Y | X] = f(X) E[Y | X]$ for bounded random variables and bounded Borel-measurable $f: \R \to \R$.  We can characterize $\mathcal{A}$-valued expectations in terms of the representing vector $\xi$ from Proposition \ref{prop:operatorvaluedGNS} as follows.  The proof is a routine computation.

\begin{lemma}
Let $\mathcal{B} \subseteq \mathcal{A}$ and let $\Phi: \mathcal{B} \to \mathcal{A}$ be a completely positive map.  Suppose that $\mathcal{H}$ is a right Hilbert $\mathcal{A}$-module and $\pi: \mathcal{B} \to B(\mathcal{H})$ is a $*$-homomorphism such that $\Phi(b) = \ip{\xi, b\xi}$.  Then $\Phi$ is an $\mathcal{A}$-valued expectation if and only if the following conditions hold:
\begin{enumerate}
	\item $\xi$ is a \emph{unit vector}, that is $\ip{\xi,\xi} = 1$ in $\mathcal{A}$.
	\item $\xi$ is \emph{$\mathcal{A}$-central}, that is $\pi(a) \xi = \xi a$ for every $a \in \mathcal{A}$.
\end{enumerate}
\end{lemma}

\subsection{$\mathcal{A}$-valued Laws} \label{subsec:Avaluedlaws}

Next, we describe the non-commutative $\mathcal{A}$-valued analogue of the law of a self-adjoint random variable.  Classically, a compactly supported measure $\mu$ on $\R$ can be viewed as a linear functional $C_0(\R) \to \C$.  Alternatively, since it is uniquely determined by its moments, it can be viewed as a map $\C[x] \to \C$, where $\C[x]$ is the polynomial algebra.

Let $\mathcal{A}\ip{X}$ denote the non-commutative polynomial algebra in a single variable $X$.  In other words, $\mathcal{A}\ip{X}$ is the universal algebra over $\C$ generated by $\mathcal{A}$ and a variable $X$.  It is spanned by the non-commutative monomials $a_0 X a_1 \dots X a_k$ for $k \in \N$ and $a_j \in \mathcal{A}$.  The multiplication is given by
\[
(a_0 X a_1 \dots Xa_k)(b_0Xb_1 \dots Xb_j) = a_0 X a_1 \dots X(a_kb_0)Xb_1 \dots Xb_j.
\]
It is also has a $*$-operation given by
\[
(a_0 X a_1 \dots X a_k)^* = a_k^* X \dots a_1^* X a_0^*.
\]

\begin{definition} \label{def:law}
An \emph{$\mathcal{A}$-valued law} is a linear map $\sigma: \mathcal{A}\ip{X} \to \mathcal{A}$ such that
\begin{enumerate}[(1)]
	\item it is completely positive in the sense that for every $P(X) \in M_n(\mathcal{A}\ip{X})$, we have $\sigma^{(n)}(P(X)^*P(X)) \geq 0$ in $M_n(\mathcal{A})$;
	\item there exists $M > 0$ and $C > 0$ such that for every $k$ and every $a_0$, \dots, $a_k \in \mathcal{A}$,
	\[
	\norm{\sigma(a_0Xa_1\dots X a_k)} \leq C M^k \norm{a_0} \dots \norm{a_n}.
	\]
	\item we have $\sigma |_{\mathcal{A}} = \id$.
	\item we have $\sigma(a_1 p(X) a_2) = a_1 \sigma(p(X)) a_2$ when $a_1$, $a_2 \in \mathcal{A}$.
\end{enumerate}
We call $\sigma$ a \emph{generalized law} if it satisfies (1) and (2), but not necessarily (3) or (4).  We denote
\begin{equation}
\rad(\sigma) = \inf \{M > 0 \text{ such that } (2) \text{ holds for some } C\}.
\end{equation}
\end{definition}

\begin{notation}
We denote the set of laws by $\Sigma_0(\mathcal{A})$.
\end{notation}

If $x$ is a self-adjoint random variable in the $\mathcal{A}$-valued probability space $(\mathcal{B},E)$, then the \emph{law of $x$} is the map $\mu_x: \mathcal{A} \ip{X} \to \mathcal{A}$ given by $p(X) \mapsto E[p(x)]$.  Note that $\mu_x$ is a law according to Definition \ref{def:law}; indeed, (1), (3), and (3) follow from properties of the $\mathcal{A}$-valued expectation $E$, while (2) follows from the fact that $x$ is a bounded operator, where we take $C = 1$ and $M = \norm{x}$.  More generally, if $x$ is self-adjoint in $\mathcal{B}$, $\pi: \mathcal{B} \to B(\mathcal{H})$ is a representation of $\mathcal{B}$ on a right Hilbert $\mathcal{A}$-module $\mathcal{H}$, and $\xi \in \mathcal{H}$, then the map
\[
\sigma: \mathcal{A}\ip{X} \to \mathcal{A}: p(X) \mapsto \ip{\xi, p(x) \xi}
\]
is a generalized law.

The following proposition shows the converse, namely, that every generalized law defined on the formal polynomial algebra $\mathcal{A}\ip{X}$ comes from a completely positive map $\Phi: \mathcal{B} \to \mathcal{A}$ and a self-adjoint $x \in \mathcal{B}$.  This is analogous to the classical statement that every probability measure on $\R$ is the law of some real random variable.  This result is an adaptation of Popa-Vinnikov \cite[Proposition 1.2]{PV2013} and Williams \cite[Proposition 2.9]{Williams2017}.

\begin{proposition} \label{prop:CPmap}
Suppose that $\sigma$ is a generalized law $\mathcal{A} \ip{X} \to \mathcal{A}$.  Then there exists a $C^*$ algebra $\mathcal{B}$, a unital $*$-homomorphism $\pi: \mathcal{A} \ip{X} \to \mathcal{B}$, and a completely positive map $\widehat{\sigma}: \mathcal{B} \to \mathcal{A}$ such that $\sigma = \widehat{\sigma} \circ \pi$ and $\norm{\pi(X)}_{\mathcal{B}} = \rad(\sigma)$.

Furthermore if $\sigma$ is a law, then $\pi$ is an embedding, and $(\mathcal{B},\widehat{\sigma})$ is an $\mathcal{A}$-valued probability space.
\end{proposition}

The proof goes by way of the GNS construction.  We define $\mathcal{A}\ip{X} \otimes_\sigma \mathcal{A}$ to be the completed quotient of $\mathcal{A}\ip{X} \otimes_{\alg} \mathcal{A}$ with the pre-inner product $\ip{p_1(X) \otimes a_1, p_2(X) \otimes a_2}_\sigma = a_1^* \sigma(p_1(X)^*p_2(X)) a_2$.  We want to define $\pi: \mathcal{A}\ip{X} \to B(\mathcal{A}\ip{X} \otimes_\sigma \mathcal{A})$ to be the action of left multiplication, define $\mathcal{B}$ to be $C^*$-algebra generated by $\pi(\mathcal{A}\ip{X})$, and define $\widehat{\sigma}: \mathcal{B} \to \mathcal{A}$ by $\widehat{\sigma}(b) = \ip{(1 \otimes 1), b(1 \otimes 1)}$.  However, to show that the left multliplication action $\pi$ is well-defined and bounded requires additional argument since $\mathcal{A}\ip{X}$ is not a $C^*$-algebra.  We refer to the papers cited above for the complete proof.

\begin{remark}
If $\sigma$, $\pi$, and $\widehat{\sigma}$ are as in the preceding proposition, then by \eqref{eq:CPmapbound}, we have for $A_j \in M_n(\mathcal{A})$ that
\begin{align}
\norm{\sigma^{(n)}(A_0 X A_1 \dots X A_k)} &\leq \norm{\widehat{\sigma}(1)} \norm{\pi^{(n)}(A_0 X A_1 \dots X A_k)} \nonumber \\
&\leq \norm{\sigma(1)} \rad(\sigma)^k \norm{A_0} \dots \norm{A_k},  \label{eq:improvedradiusbound}
\end{align}
which is a sharpening of the estimate (2) assumed in Definition \ref{def:law}.
\end{remark}

\begin{lemma} \label{lem:maxradius}
If $\sigma_1$ and $\sigma_2: \mathcal{A}\ip{X} \to \mathcal{A}$ are generalized laws, then $\sigma_1 + \sigma_2$ is a generalized laws and satisfies
\begin{equation} \label{eq:maxradius}
\rad(\sigma_1 + \sigma_2) = \max(\rad(\sigma_1), \rad(\sigma_2)).
\end{equation}
\end{lemma}

\begin{proof}
It is immediate that $\sigma_1 + \sigma_2$ is completely positive.  Let $M_1 = \rad(\sigma_1)$ and $M_2 = \rad(\sigma_2)$ and $M = \max(M_1,M_2)$.  Then for some constants $C_1$ and $C_2$,
\begin{align*}
\norm{(\sigma_1 + \sigma_2)(a_0 X a_1 \dots X a_k)} &\leq (C_1 M_1^n + C_2 M_2^n) \norm{a_0} \dots \norm{a_k} \\
&\leq (C_1 + C_2) M^n \norm{a_0} \dots \norm{a_k}.
\end{align*}
Thus, $\sigma_1 + \sigma_2$ satisfies Definition \ref{def:law} (2), so it is a generalized law, and $\rad(\sigma_1 + \sigma_2) \leq M$.  It remains to show the reverse inequality.  By positivity of the inner product on $\mathcal{A}\ip{X} \otimes_{\sigma_j} \mathcal{A}$, we have
\[
0 \leq \ip{p(X) \otimes 1 - 1 \otimes \sigma_j(p(X)), p(X) \otimes 1 - 1 \otimes \sigma_j(p(X))}_{\sigma_j} = \sigma_j(p(X)^*p(X)) - \sigma_j(p(X))^* \sigma_j(p(X)).
\]
Therefore,
\[
0 \leq \sigma_j(p(X))^* \sigma_j(p(X)) \leq \sigma_j(p(X)^*p(X)) \leq (\sigma_1 + \sigma_2)(p(X)^*p(X)).
\]
Hence,
\[
\norm{\sigma_j(p(X))} \leq \norm{(\sigma_1 + \sigma_2)(p(X)^* p(X))}^{1/2}.
\]
In particular, taking $p(X) = a_0 X a_1 \dots X a_n$, we have
\begin{align*}
\norm{\sigma_j(a_0 X a_1 \dots X a_k)} &\leq \norm{(\sigma_1 + \sigma_2)((a_0Xa_1 \dots Xa_k)^*(a_0Xa_1 \dots X a_k))}^{1/2} \\
&\leq \norm{\sigma_1(1) + \sigma_2(1)} \rad(\sigma_1 + \sigma_2)^k \norm{a_0} \dots \norm{a_k}.
\end{align*}
Hence, $\rad(\sigma_j) \leq \rad(\sigma_1 + \sigma_2)$.
\end{proof}

\subsection{Fully Matricial Functions} \label{subsec:matricial}

The Cauchy transform of a probability measure plays an important role in non-commutative probability theory, comparable to the role of the Fourier transform in classical probability theory and analysis.  Moreover, the study of Cauchy transforms as analytic functions is essential for the study of Loewner chains in the upper half-plane.

There is an analytic theory of $\mathcal{A}$-valued Cauchy transforms that is strikingly similar to the scalar-valued theory, including the explicit estimates and analytic characterization for Cauchy transforms that we will explain in the next section and use heavily in the rest of the paper.  These results rely on viewing the Cauchy transform $G_\sigma$ not merely as a Banach-valued analytic function on a subset of $\mathcal{A}$, but rather as a sequence of functions $G_\sigma^{(n)}$ defined on $n \times n$ matrices over $\mathcal{A}$ for every $n$, that is, a fully matricial function.

The theory of fully matricial functions (also known as non-commutative functions) was first developed in \cite{Taylor1972,Taylor1973}.  Its application to non-commutative probability is due to \cite[\S 6 - 7]{Voiculescu2004}.  A systematic treatment can be found in \cite{KVV2014}.  We restrict our attention the setting studied in \cite[\S 7.3]{KVV2014} and \cite{Williams2017}, which in their terminology would be uniformly analytic non-commutative functions on uniformly open non-commutative sets.  We also restrict ourselves to matrices over a $C^*$-algebra rather than general operator spaces.

\begin{notation}
If $\mathcal{A}$ is an algebra and $z \in M_n(\mathcal{A})$, we denote
\[
z^{(m)} = z \otimes 1_m = \begin{bmatrix} z & 0 & \dots & 0 & 0 \\ 0 & z & \dots & 0 & 0 \\ \vdots & \vdots & \ddots & \vdots & \vdots \\ 0 & 0 & \dots & z & 0 \\ 0 & 0 & \dots & 0 & z \end{bmatrix} \in M_{mn}(\mathcal{A}).
\]
Moreover, for $z \in M_n(\mathcal{A})$ and $w \in M_m(\mathcal{A})$, we denote
\[
z \oplus w = \begin{bmatrix} z & 0 \\ 0 & w \end{bmatrix} \in M_{n+m}(\mathcal{A}).
\]
\end{notation}

\begin{definition} 
Let $\mathcal{A}$ be a unital $C^*$-algebra.  Let $M_\bullet(\mathcal{A})$ denote the sequence of sets $(M_n(\mathcal{A}))_{n \in \N}$.  A \emph{fully matricial domain} in $M_\bullet(\mathcal{A})$ is a sequence of sets $U = (U^{(n)})_{n \in \N}$, where
\begin{enumerate}[(1)]
	\item $U^{(n)}$ is a connected subset of $M_n(\mathcal{A})$.
	\item $U$ is \emph{uniformly open}, that is, for every $z \in U^{(n)}$, there exists $r > 0$ such that for every $m \in \N$, we have $B_{M_{nm}(\mathcal{A})}(z^{(m)},r) \subseteq U^{(mn)}$.
	\item $U$ \emph{respects direct sums}, that is, if $z \in U^{(n)}$ and $w \in U^{(m)}$, then $z \oplus w \in U^{(n+m)}$.
\end{enumerate}
If $U = \{U^{(n)}\}$ and $\tilde{U} = \{\tilde{U}^{(n)}\}$ are fully matricial domains, write $U \subseteq \widehat{U}$ to mean that $U^{(n)} \subseteq \widehat{U}^{(n)}$ for each $n$.
\end{definition}

\begin{definition}
Let $U = (U^{(n)})_{n \in \N}$ be a matricial domain.  A fully matricial function $F: U \to M_\bullet(\mathcal{A})$ is a sequence of maps $F^{(n)}: U^{(n)} \to M_n(\mathcal{A})$ such that
\begin{enumerate}[(1)]
	\item $F$ \emph{respects direct sums}, that is, if $z \in U^{(n)}$ and $w \in U^{(m)}$, then $F^{(n+m)}(z \oplus w) = F^{(n)}(z) \oplus F^{(m)}(w)$.
	\item $F$ \emph{respects similarities}, that is, $s \in M_n(\C) \subseteq M_n(\mathcal{A})$ is invertible, if $z \in U^{(n)}$ and $szs^{-1} \in U^{(n)}$, then $F^{(n)}(szs^{-1}) = s F^{(n)}(z) s^{-1}$.
	\item $F$ is \emph{uniformly locally bounded}, that is, for every $n$ and every $z_0 \in U^{(n)}$, there exists an $r > 0$ and $C > 0$ such that for every $m \in \N$, we have
	\[
	z \in B_{M_{mn}(\mathcal{A})}(z_0^{(m)},r) \implies z \in U^{(mn)} \text{ and } \norm{F^{(mn)}(z)} \leq C.
	\]
\end{enumerate}
\end{definition}

As shown in \cite[\S 7.1]{KVV2014}, these assumptions imply that $F^{(n)}$ is an analytic function $U^{(n)} \to M_n(\mathcal{A})$.  Moreover, the non-commutative difference-differential operators defined in \cite{KVV2014} relate in a natural way to the Banach-valued derivatives $\delta^k F^{(n)}(z;h)$.  As we will not need the difference-differential calculus in our proofs, the interested reader may see \cite[\S 2.3]{Williams2017} for a convenient summary.

\begin{notation}
In this paper, $(F^{(n)})^{-1}(z)$ will denote the inverse function of $F^{(n)}$ (when defined), while $F^{(n)}(z)^{-1}$ will denote the inverse in $M_n(\mathcal{A})$ of the element $F(z)$ (when defined).  We will denote by $DF^{(n)}(z)$ the Fr{\'e}chet derivative of $F^{(n)}$, which is a linear transformation $M_n(\mathcal{A}) \to M_n(\mathcal{A})$.  Meanwhile, $DF^{(n)}(z)^{-1}$ will denote the inverse linear transformation $M_n(\mathcal{A}) \to M_n(\mathcal{A})$ (when defined).
\end{notation}

\subsection{Cauchy Transforms} \label{subsec:Cauchytransform}

\begin{definition} \label{def:halfspace}
We define the \emph{$\mathcal{A}$-valued upper-half plane} $\h(\mathcal{A})$ as follows: Let
\[
\h_\epsilon^{(n)}(\mathcal{A}) = \{z \in M_n(\mathcal{A}), \im z \geq \epsilon\},
\]
where $\im z = -\frac{1}{2}i(z - z^*)$.  Then let
\begin{equation}
\h^{(n)}(\mathcal{A}) = \bigcup_{\epsilon > 0} \h_\epsilon^{(n)}(\mathcal{A}).
\end{equation}
Let $\h(\mathcal{A})$ be the matricial domain $\{\h^{(n)}(\mathcal{A})\}_{n \in \N}$.  We use the notation
\begin{equation}
\overline{\h}(\mathcal{A}) := (\{z \in M_n(\mathcal{A}): \im z \geq 0\})_{n \in \N}.
\end{equation}
\end{definition}

Let $\sigma: \mathcal{A}\ip{X} \to \mathcal{A}$ be a generalized law (with $\rad(\sigma) < +\infty$ as always in this paper).  By Theorem \ref{prop:CPmap}, there is a $C^*$-algebra $\mathcal{B}$ and a $*$-homomorphism $\pi: \mathcal{A} \ip{X} \to \mathcal{B}$ and a completely positive map $\widehat{\sigma}: \mathcal{B} \to \mathcal{A}$ satisfying $\sigma = \widehat{\sigma} \circ \pi$.  Then (as in \cite{Voiculescu2004}) we define the \emph{Cauchy transform} $G_\sigma: \h(\mathcal{A}) \to M_\bullet(\mathcal{A})$ by
\begin{equation} \label{eq:Cauchytransformdef}
G_\sigma^{(n)}(z) := \widehat{\sigma}^{(n)}[(\pi^{(n)}(z) - \pi(X)^{(n)})^{-1}].
\end{equation}
Note that $G_\sigma$ maps $\h(\mathcal{A})$ into $-\overline{\h}(\mathcal{A})$ is that for an invertible operator $z$, we have $\im z \geq 0$ if and only if $\im z^{-1} \leq 0$.  We will verify below (after Lemma \ref{lem:Cauchytransformseries}) that the function denoted $G_\sigma$ only depends on $\sigma$ and not on $\pi$ and $\widehat{\sigma}$.

\begin{lemma} \label{lem:estimate1}
Let $\sigma$ be a generalized law $\mathcal{A} \ip{X} \to \mathcal{A}$ and $\norm{X}_\sigma \leq M$ and let $G_\sigma$ be as above.  Let
\begin{equation} \label{eq:Rndef}
R_\epsilon^{(n)} = \{z: \im z \geq \epsilon\} \cup \{\norm{z^{-1}} \leq 1/(M + \epsilon)\}.
\end{equation}
Then for $z \in R_\epsilon^{(n)}$, we have
\begin{equation}
\norm{G_\sigma^{(n)}(z)} \leq \frac{\norm{\sigma(1)}}{\epsilon} \label{eq:estimate1}
\end{equation}
\end{lemma}

\begin{proof}
Choose $\mathcal{B}$, $\pi$, and $\widehat{\sigma}$ be as above.  Since $\pi$ is a unital $*$-homomorphism, it is completely positive.  Hence, if we assume $\im z \geq \epsilon$, then $\im \pi^{(n)}(z - X^{(n)}) \geq \epsilon$ and hence $\norm{\pi^{(n)}(z - X^{(n)})^{-1}} \leq 1/\epsilon$.  Then because $\widehat{\sigma}^{(n)}$ is completely positive, we obtain
\begin{equation}
\norm{\widehat{\sigma}^{(n)}(\pi^{(n)}(z - X^{(n)})^{-1})} \leq \norm{\widehat{\sigma}^{(n)}(1^{(n)})} \norm{\pi^{(n)}(z - X^{(n)})^{-1}} \leq \frac{\norm{\sigma(1)}}{\epsilon}.
\end{equation}
Similarly, if $\norm{z^{-1}} \leq (M + \epsilon)^{-1}$, then $\norm{\pi(z)^{-1}} \leq (M + \epsilon)^{-1}$.  Using the Neumann series trick and that the fact that $\norm{\pi(X)^{(n)}} \leq M$, we know that $\pi(z - X^{(n)})$ is invertible and
\begin{equation}
\norm{\pi^{(n)}(z - X^{(n)})^{-1}} \leq \frac{1}{\epsilon}.
\end{equation}
Again, by complete positivity of $\widehat{\sigma}$, we obtain $\norm{\widehat{\sigma}^{(n)} \pi^{(n)}(z - X^{(n)})^{-1}} \leq \norm{\sigma(1)} / \epsilon$.
\end{proof}

\begin{lemma}
The object $G_\sigma$ defined above is a fully matricial function on $\h(\mathcal{A})$.
\end{lemma}

\begin{proof}
It is straightforward to check that $G_\sigma$ respects direct sums and similarities.  To check uniform local boundedness, assume that $z_0 \in M_n(\mathcal{A})$ with $\im z_0 \geq \epsilon$.  Then for every $m \in \N$, we have
\[
B_{M_{mn}(\mathcal{A})}(z_0^{(m)}, \epsilon / 2) \subseteq \h_{\epsilon / 2}(\mathcal{A}),
\]
and hence for $z \in B_{M_{mn}(\mathcal{A})}(z_0^{(m)}, \epsilon/2)$, we know $G_\sigma^{(mn)}(z)$ is defined and it is bounded by $2 \norm{\sigma(1)} / \epsilon$.
\end{proof}

Recall that in the scalar case, the Cauchy transform of a compactly supported finite measure on $\R$ extends to be analytic in a neighborhood of $\infty$ and the power series coefficients at $\infty$ are given by the moments of the measure.  The operator-valued analogue is as follows (see e.g.\ \cite[Proposition 2.17]{Williams2017}).

\begin{lemma} \label{lem:Cauchytransformseries}
Let $\sigma$ and $G_\sigma$ be as above with $\rad(\sigma) \leq M$.  Let $\tilde{G}_\sigma(z) = G_\sigma(z^{-1})$ for $z \in -\h(\mathcal{A})$.  Then $\tilde{G}_\sigma^{(n)}|_{B_{M_n(\mathcal{A})}(0,1/M) \cap -\h^{(n)}(\mathcal{A})}$ has a unique analytic extension $\tilde{G}$ to $B_{M_n(\mathcal{A})}(0,1/M)$ given by
\begin{equation}
\tilde{G}^{(n)}(z) = \sum_{k=0}^\infty \sigma^{(n)}(z(X^{(n)}z)^k),
\end{equation}
The sequence $(G^{(n)})$ is fully matricial and satisfies the estimate
\begin{equation} \label{eq:estimatenearinfinity}
\norm{\tilde{G}^{(n)}(z)} \leq \frac{\norm{\sigma(1)}}{\norm{z}^{-1} - M}.
\end{equation}
\end{lemma}

\begin{proof}
Let $\pi$ and $\widehat{\sigma}$ be as above.  Observe that
\[
(z^{-1} - \pi(X)^{(n)})^{-1} = z(1 - \pi(X)^{(n)}z)^{-1} = \sum_{k=0}^\infty z (\pi(X)^{(n)} z)^k.
\]
Thus, for $z \in B_{M_n(\mathcal{A})}(0,1/M) \cap -\h^{(n)}(\mathcal{A})$, we have
\begin{equation}
\tilde{G}_\sigma^{(n)}(z) = \sum_{k=0}^\infty \widehat{\sigma}^{(n)} \circ \pi^{(n)} [z(X^{(n)}z)^k] = \sum_{k=0}^\infty \sigma^{(n)}[z(X^{(n)}z)^k].
\end{equation}
Here convergence of the series follows from the fact that $\norm{z(\pi(X)^{(n)}z)^k} \leq \norm{z}^{k+1} M^k$ and $\norm{\sigma^{(n)}} = \norm{\sigma(1)}$, and this also justifies \eqref{eq:estimatenearinfinity}.  We leave it as an exercise to check that $\tilde{G}$ is fully matricial.
\end{proof}

One consequence of this is that $G_\sigma$ is well-defined, independent of the choice of $(\mathcal{B},\pi,\widehat{\sigma})$ realizing the law $\sigma$.  Indeed, it follows from Lemma \ref{lem:Cauchytransformseries} that two choices of $\pi$ will yield the same function $\tilde{G}_\sigma$ for $z \in B_{M_n(\mathcal{A})}(0,1/M) \cap -\h^{(n)}(\mathcal{A})$ and hence everywhere by analytic continuation (Lemma \ref{lem:analyticcontinuation}).

Thus, the \emph{Cauchy transform} $G_\sigma$ of a generalized law $\sigma$ is well-defined by \eqref{eq:Cauchytransformdef}.  Moreover, we denote the fully matricial extension $\tilde{G}$ constructed in Lemma \ref{lem:Cauchytransformseries} by $\tilde{G}_\sigma$ since no confusion will result.  Just as in the scalar case, the operator-valued Cauchy transform satisfies the following local Lipschitz estimate.

\begin{lemma} \label{lem:estimates}
Let $\sigma$ be a generalized law $\mathcal{A} \ip{X} \to \mathcal{A}$ and $\rad(\sigma) \leq M$.  Let $R_\epsilon^{(n)}$ be given by \eqref{eq:Rndef} as in Lemma \ref{lem:estimate1}.  Then for $z$ and $z' \in R_\epsilon^{(n)}$, we have
\begin{align}
\norm{G(z) - G(z')} &\leq \frac{\norm{\sigma(1)}}{\epsilon^2} \norm{z - z'} \label{eq:estimate2} \\
\norm{DG(z)} &\leq \frac{\norm{\sigma(1)}}{\epsilon^2} \label{eq:estimate3}
\end{align}
\end{lemma}

\begin{proof}
Relation \eqref{eq:estimate2} follows from a similar argument to Lemma \ref{lem:estimate1} using the resolvent identity
\begin{equation}
\pi(z - X^{(n)})^{-1} - \pi(z' - X^{(n)})^{-1} = - \pi(z - X^{(n)})^{-1}\pi(z - z') \pi(z' - X^{(n)})^{-1}.
\end{equation}
Relation \eqref{eq:estimate3} follows from \eqref{eq:estimate2}.
\end{proof}

The next lemma shows the existence of a fully matricial inverse function for $\tilde{G}_\mu$ in a neighborhood of zero when $\mu$ is a law.  Implicit and inverse function theorems for fully matricial functions have been studied in \cite[\S 11.5]{Voiculescu2004}, \cite{AKV2013}, \cite{AKV2015}, \cite{AM2016}.  In particular, the following lemma is a special case of \cite[Theorem 1.4]{AKV2015}.  We sketch the proof here for the sake of exposition and to justify our explicit estimates on the radius of the ball where $\tilde{G}_\mu^{-1}$ is defined.

\begin{lemma} \label{lem:Cauchytransforminverse}
Let $G$ be the Cauchy transform of a law $\mu$ with $\rad(\mu) \leq M$.  Then $\tilde{G}^{(n)}$ has a fully matricial inverse $(G^{(n)})^{-1}: B^{(n)}(0,R_2) \to B^{(n)}(0,R_1)$, where $R_1 = M^{-1}(1 - 1/\sqrt{2})$ and $R_2 = M^{-1}(3 - 2 \sqrt{2})$.  In particular, the radii $R_1$ and $R_2$ are independent of $n$.
\end{lemma}

\begin{proof}
Write $\tilde{G}^{(n)}(z) = z + P^{(n)}(z)$.  A power series manipulation shows that when $\norm{z}$, $\norm{z'} \leq R < 1/M$, we have
\begin{equation}
\norm*{P^{(n)}(z)} \leq \frac{MR^2}{1 - MR}
\end{equation}
and
\begin{equation}
\norm*{P^{(n)}(z) - P^{(n)}(z')} \leq \left( \frac{1}{(1 - MR)^2} - 1 \right) \norm{z - z'}.
\end{equation}
Now $z$ satisfies $\tilde{G}^{(n)}(z) = w$ if and only if $z$ is a fixed point of the map $Q^{(n)}(z) = w - P^{(n)}(z)$.  When $R < R_1$, we have $1/(1 - MR)^2 - 1 < 1$, so that $P^{(n)}$ and hence $Q^{(n)}$ are contractions for $\norm{z} \leq R$.  Observe that $Q^{(n)}$ maps $\overline{B}^{(n)}(0,R)$ into $\overline{B}^{(n)}(0,R)$ provided that
\begin{equation}
\norm{w} \leq R' := R - \frac{MR^2}{1 - MR};
\end{equation}
this follows from
\begin{equation}
\norm*{Q^{(n)}(z)} = \norm*{w - P^{(n)}(z)} \leq \norm*{w} + \frac{MR^2}{1 - MR}.
\end{equation}
Therefore, by the Banach fixed point theorem, $Q^{(n)}$ has a unique fixed point in $\overline{B}^{(n)}(0,R)$, so that $(\tilde{G}^{(n)})^{-1}$ is a well-defined map $B^{(n)}(0,R') \to B^{(n)}(0,R)$.

As $R \to R_1$, we have $R' \to R_2$, and hence $(\tilde{G}^{(n)})^{-1}$ is defined on the asserted domain.  By the standard proof of the inverse function theorem, $(\tilde{G}^{(n)})^{-1}$ is Fr{\'e}chet-differentiable in the complex sense, and hence analytic, and it clearly preserves direct sums and similarities by uniqueness of the fixed point of $Q^{(n)}$ in $B^{(n)}(0,R_2)$.
\end{proof}

We will rely on the following result of Williams \cite[Theorem 3.1]{Williams2017} and Williams-Anshelevich \cite[Theorem A.1]{AW2016} that gives an analytic characterization of matricial Cauchy transforms.

\begin{theorem} \label{thm:Cauchytransform}
Let $G = (G^{(n)})_{n \geq 1}$ be a sequence of functions $G^{(n)}: \h^{(n)}(\mathcal{A}) \to -\overline{\h}^{(n)}(\mathcal{A})$.  Then $G$ is the Cauchy transform of some $\sigma \in \Sigma_0(\mathcal{A})$ with $\rad(\sigma) \leq M$ if and only if the following hold:
\begin{enumerate}[(1)]
	\item $G$ is a fully matricial function.
	\item For each $n$, $\tilde{G}^{(n)}(z) = G^{(n)}(z^{-1})$ extends to be analytic on $B_{M_n(\mathcal{A})}(0,1/M)$.
	\item For each $\epsilon > 0$, we have $\norm{\tilde{G}^{(n)}(z)} \leq C_\epsilon$ for $\norm{z} < 1/(M + \epsilon)$, where $C_\epsilon$ is independent of $n$.
	\item $\tilde{G}^{(n)}(z^*) = \tilde{G}^{(n)}(z)^*$.
	\item $\tilde{G}^{(n)}(0) = 0$.
\end{enumerate}
In this case, $G$ is the Cauchy transform of a \emph{law} if and only if for each $n$,
\begin{equation}
\lim_{z \to 0} z^{-1} \tilde{G}^{(n)}(z) = 1^{(n)},
\end{equation}
where the limit is taken over invertible $z$ and $z \to 0$ in operator norm.  Moreover, for each $\delta \in (0,1/M)$, the law $\mu$ is uniquely determined by the fully matricial function $\tilde{G}$ restricted to $\norm{z} < \delta$.
\end{theorem}

The chacterization of $G_\sigma(z)$ for laws $\sigma \in \Sigma_0$ was given by \cite[Theorem 3.1]{Williams2017} and it was extended to generalized laws by \cite[Theorem A.1]{AW2016}.  As explained in those papers and previous work, the generalized law $\sigma$ is recovered from looking at the power series expansion of $\tilde{G}$ at $0$.  Indeed, to evaluate $\sigma(a_0 X a_1 \dots X a_n)$, for $a_0$, \dots, $a_n \in \mathcal{A}$, we consider the upper triangular $(n + 2) \times (n + 2)$ matrix
\begin{equation} \label{eq:uppertriangular}
z = \begin{bmatrix}
0 & a_0 & 0 & \dots & 0 & 0 \\
0 & 0 & a_1 & \dots & 0 & 0 \\
0 & 0 & 0 & \dots & 0 & 0 \\
\vdots & \vdots & \vdots & \ddots & \vdots & \vdots \\
0 & 0 & 0 & \dots & a_n & 0 \\
0 & 0 & 0 & \dots & 0 & 0
\end{bmatrix}
\end{equation}
Then for sufficiently small $\zeta \in \C$, we have by Lemma \ref{lem:Cauchytransformseries} that
\[
\tilde{G}^{(n+2)}(\zeta z) = \sum_{k=0}^\infty \zeta^{k+1} \sigma^{(n+2)}[(zX^{(n+2)})^kz]
\]
Note $zX^{(n+2)}$ is an upper triangular nilpotent matrix, the series terminates at $k = n + 1$, and the upper left entry of $\tilde{G}^{(n+2)}(\zeta z)$ is
\begin{equation} \label{eq:Cauchytransformmoment}
[\tilde{G}^{(n+2)}(\zeta z)]_{1,n+2} = \zeta^{n+2} \sigma(a_0 X a_1 \dots X a_n).
\end{equation}
Hence, $\sigma(a_0 X a_1 \dots Xa_n)$ is determined by evaluating $\tilde{G}$ at $\zeta z$.

We have stated the theorem here in a more precise form than \cite[Theorem 3.1]{Williams2017}, \cite[Theorem A.1]{AW2016} by including the characterization of when $\rad(\mu) \leq M$.  For a generalized law $\sigma$, the fact that $\rad(\sigma) \leq M$ implies (2) and (3) follows from Lemma \ref{lem:Cauchytransformseries}.  Conversely, if (2) and (3) hold, then for an arbitrary $\epsilon > 0$, we can show that $\rad(\sigma) \leq M + \epsilon$ by applying \eqref{eq:Cauchytransformmoment} with arbitrary elements $a_1$, \dots, $a_n$ normalized so that $\norm{a_j} = 1$ and with $\zeta = 1 / (M + \epsilon)$.

\subsection{$F$-Transforms} \label{subsec:Ftransform}

If $\mu: \mathcal{A}\ip{X} \to \mathcal{A}$ is an $\mathcal{A}$-valued law, then the \emph{$F$-transform} of $\mu$ is given by $F_\mu(z) = G_\mu(z)^{-1}$ where defined.  Because Loewner theory deals with $F$-transforms, we will next establish some basic properties of $F$-transforms and in particular their behavior under composition.

As in the scalar case, reciprocal Cauchy transforms of laws can be expressed as the identity minus a Cauchy transform.  This is similar to a result proved in \cite[Theorem 5.6]{PV2013} and \cite[Corollary 3.3]{Williams2017}, which gives a Nevanlinna-type representation for $F_\mu(z) - z$.   However, we prefer to work with the Cauchy transform representation of $F_\mu(z) - z$ given by \cite[Remark 5.7]{PV2013}.   The form of the proposition stated here can be proved by manipulating the power series at $\infty$ in the obvious ways and applying Theorem \ref{thm:Cauchytransform}.  We leave the details of the computation to the reader.

\begin{proposition} \label{prop:Ftransformbijection} ~
\begin{enumerate}[(1)]
	\item If $\mu$ is a bounded law, then there exists a unique generalized law $\sigma: \mathcal{A}\ip{X} \to \mathcal{A}$ with $\norm{X}_\sigma < +\infty$ and a self-adjoint $a_0 \in \mathcal{A}$ such that
\begin{equation} \label{eq:Ftransformbijection}
F_\mu^{(n)}(z) = z - G_\sigma^{(n)}(z) - a_0^{(n)}.
\end{equation}
	\item Conversely, given such a $\sigma$ and $a_0$, there exists a unique law $\mu$ satisfying \eqref{eq:Ftransformbijection}.
	\item Letting $H(z) := F_\mu(z) - z$ and $\tilde{H}(z) := H(z^{-1})$, we have $a_0 = -\tilde{H}^{(1)}(0) = \mu(X)$ and $\sigma(b) = -D\tilde{H}^{(1)}(0)[b] = \mu(X b X) - \mu(X) b \mu(X)$ for $b \in \mathcal{A}$.

	\item We have the estimates
	\begin{equation}
	\rad(\sigma) \leq 2 \rad(\mu), \qquad \rad(\mu) \leq (1 + \norm{\sigma(1)}) \rad(\sigma) + \norm{a_0}.
	\end{equation}
\end{enumerate}
\end{proposition}

Our next results concern the behavior of $F$-transforms under composition.

\begin{definition}
Let $U$ be a matricial domain in $M(\mathcal{A})$.  A family $\mathcal{F}$ of fully matricial functions $U \to U$ will be said to have the \emph{two-out-of-three property} if
\begin{enumerate}[(1)]
	\item If $F_1$ and $F_2$ are in $\mathcal{F}$, then $F_1 \circ F_2 \in \mathcal{F}$.
	\item If $F_1$ and $F_1 \circ F_2$ are in $\mathcal{F}$, then $F_2 \in \mathcal{F}$.
	\item If $F_2$ and $F_1 \circ F_2$ are in $\mathcal{F}$, then $F_1 \in \mathcal{F}$.
\end{enumerate}
\end{definition}

\begin{proposition} \label{prop:twooutofthree}
The family of $F$-transforms of $\mathcal{A}$-valued laws has the two-out-of-three property on the domain $\h(\mathcal{A})$.
\end{proposition}

\begin{proof}
Let $\inv$ denote the map $z \mapsto z^{-1}$ where defined.  Note that $F_\mu = \inv \circ \tilde{G}_\mu \circ \inv$, so it suffices to show that the transforms $\tilde{G}_\mu$ satisfy the two-out-of-three property.  By Theorem \ref{thm:Cauchytransform} and Lemma \ref{lem:Cauchytransforminverse}, a fully matricial function $G: \h(\mathcal{A}) \to -\overline{\h}(\mathcal{A})$ is a Cauchy transform if and only if it satisfies
\begin{enumerate}[(1)]
	\item $\tilde{G}^{(n)}$ defines a map $B^{(n)}(0,R_1) \to B^{(n)}(0,R_2)$ for some $R_1$ and $R_2$ independent of $n$, and the inverse function $(\tilde{G}^{(n)})^{-1}$ defines a map $B^{(n)}(0,R_3) \to B^{(n)}(0,R_4)$ for some $R_3$ and $R_4$ independent of $n$.
	\item $\tilde{G}^{(n)}(0) = 0$.
	\item $\tilde{G}^{(n)}(z^*) = \tilde{G}^{(n)}(z)^*$.
	\item $\lim_{z \to 0} z^{-1} \tilde{G}^{(n)}(z) = 1^{(n)}$ for each $n$.
\end{enumerate}
FIrst, we observe that (1) satisfies the two-out-of-three property.  Next, if we restrict our attention to functions satisfying (1), we see that (2) and (3) also satisfy the two-out-of-three property.  Finally, note that if $\tilde{G}_3 = \tilde{G}_1 \circ \tilde{G}_2$, then
\begin{equation}
z^{-1} \tilde{G}_3^{(n)} (z) = [z^{-1}G_1^{(n)}(z)] [G_1^{(n)}(z)^{-1} G_2^{(n)}(G_1^{(n)}(z))].
\end{equation}
Thus, if we restrict our attention to functions satisfying (1), then condition (4) has the two-out-of-three property.  Altogether, we have shown that $F$-transforms satisfy the two-out-of-three property.
\end{proof}

Suppose that $F_{\mu_3} = F_{\mu_1} \circ F_{\mu_2}$.   It follows from Theorem \ref{thm:Cauchytransform} and Lemma \ref{lem:Cauchytransforminverse} that there are constants $C_j$ such that
\begin{equation} \label{eq:radiusestimate0}
\rad(\mu_j) \leq C_j \sum_{k \neq j} \rad(\mu_k).
\end{equation}
But in fact, bounds on $\rad(\mu_3)$ will automatically imply bounds on $\rad(\mu_1)$ and $\rad(\mu_2)$.

\begin{proposition} \label{prop:twooutofthree2}
There exists a constant $C > 0$ such that if $F_{\mu_3} = F_{\mu_1} \circ F_{\mu_2}$, then
\begin{equation}
\max(\rad(\mu_1), \rad(\mu_2)) \leq C \left(\rad(\mu_3) + \norm{\mu_2(X)} \right)
\end{equation}
\end{proposition}

\begin{proof}
Write $F_{\mu_j}^{(n)}(z) = z - \mu_j(X)^{(n)} - G_{\sigma_j}^{(n)}(z)$ as in Proposition \ref{prop:Ftransformbijection}.  Note that
\begin{equation}
\im F_{\mu_2}^{(n)}(z) \leq \im F_{\mu_3}^{(n)}(z).
\end{equation}
Therefore, $G_{\sigma_3} - G_{\sigma_2}$ maps $\h(\mathcal{A})$ into $\h^-(\mathcal{A})$.  Therefore, by Theorem \ref{thm:Cauchytransform}, there exists $\sigma_4$ such that $G_{\sigma_3} = G_{\sigma_2} + G_{\sigma_4}$.  So by Lemma \ref{lem:maxradius} $\rad(\sigma_2) \leq \rad(\sigma_3)$.

But by Proposition \ref{prop:Ftransformbijection}, we have
\begin{align}
\rad(\mu_2) &\leq (1 + \norm{\sigma_2(1)}) \rad(\sigma_2) + \norm{\mu_2(X)} \nonumber \\
&\leq (1 + \norm{\sigma_3(1)}) \rad(\sigma_3) + \norm{\mu_2(X)} \nonumber \\
&\leq (1 + 2 \rad(\mu_3)^2) 2 \rad(\mu_3) + \norm{\mu_2(X)}.
\end{align}
Here we use the fact $\sigma_3(1) = \mu_3(X^2) - \mu_3(X)^2$ is bounded by $2 \rad(\mu_3)^2$.  It follows from a rescaling argument that
\begin{equation}
c \rad(\mu_2) \leq (1 + 2 c \rad(\mu_3)^2) 2 c\rad(\mu_3) + c\norm{\mu_2(X)},
\end{equation}
and by taking $c$ to zero, we obtain $\rad(\mu_2) \leq 2 \rad(\mu_3) + \norm{\mu_2(X)}$.  The estimate for $\rad(\mu_1)$ now follows from \eqref{eq:radiusestimate0}.
\end{proof}

\section{Chordal Loewner Chains} \label{sec:Loewnerchains}

\subsection{Definition and Basic Properties}

\begin{definition} \label{def:subordination}
Let $F_1, F_2: \h(\mathcal{A}) \to \h(\mathcal{A})$.  If there exists a fully matricial function $F: \h(\mathcal{A}) \to \h(\mathcal{A})$ such that $F_1 = F_2 \circ F$, then we say that $F_1$ is \emph{subordinated} to $F_2$ and write $F_1 \preceq F_2$.  (To remember the direction of the $\preceq$ sign, note that the image of $F_1$ is contained in the image of $F_2$.)
\end{definition}

\begin{definition} \label{def:CLC}
A \emph{$\mathcal{A}$-valued chordal Loewner chain} is a collection of fully matricial functions $F_t: \h(\mathcal{A}) \to \h(\mathcal{A})$ such that
\begin{enumerate}[(1)]
	\item $F_0(z) = z$;
	\item $F_t$ is the $F$-transform of some law $\mu_t \in \Sigma_0(\mathcal{A})$;
	\item If $s \leq t$, then $F_s \succeq F_t$;
	\item The functions $t \mapsto \mu_t(X)$ and $t \mapsto \mu_t(X^2)$ are continuous (with respect to $\norm{\cdot}_{\mathcal{A}}$).
\end{enumerate}
\end{definition}

We will often omit the adjectives ``$\mathcal{A}$-valued'' and ``chordal'' because we are only dealing with this type of Loewner chain for the rest of the paper.

\begin{lemma}~ \label{lem:semigroupproperties}
Let $(F_t)_{t \in [0,T]}$ be a Loewner chain.
\begin{enumerate}[(1)]
	\item For $0 \leq s \leq t \leq T$, there exists a unique fully matricial function $F_{s,t}: \h(\mathcal{A}) \to \h(\mathcal{A})$ such that $F_t = F_s \circ F_{s,t}$.
	\item For $0 \leq s \leq t \leq T$, the map $F_{s,t}$ is the $F$-transform of a law $\mu_{s,t} \in \Sigma_0(\mathcal{A})$.
	\item We have $F_{0,t} = F_t$ and $F_{s,u} = F_{s,t} \circ F_{t,u}$ whenever $s \leq t \leq u$.
\end{enumerate}
\end{lemma}

\begin{proof}
(1) and (2).  Fix $s \leq t$.  Then there exists $F: \h(\mathcal{A}) \to \h(\mathcal{A})$ such that $F_t = F_s \circ F$.  By Proposition \ref{prop:twooutofthree}, $F$ is the $F$-transform of some law in $\mu \in \Sigma_0(\mathcal{A})$.  In a neighborhood of $0$, we have
\begin{equation}
\tilde{G}_{\mu_s} \circ \tilde{G}_\mu = \tilde{G}_{\mu_t},
\end{equation}
which shows that there is one and only one law $\mu$ that satisfies $F_t = F_s \circ F_\mu$.  If we denote this law by $\mu_{s,t}$, then $F_{s,t} := F_{\mu_{s,t}}$ is the unique fully matricial function satisying $F_t = F_s \circ F_{s,t}$.  This shows (1) and (2).

(3) We have $F_0 = \id$ and hence $F_t = F_0 \circ F_{0,t} = F_{0,t}$.  Now suppose $s \leq t \leq u$.  Then $F_s \circ F_{s,t} \circ F_{t,u} = F_u = F_s \circ F_{s,u}$, and hence $F_{s,t} \circ F_{t,u} = F_{s,u}$ by the uniqueness claim of (1).
\end{proof}

\begin{lemma} \label{lem:radiusbound}
For $0 \leq s \leq t \leq T$, let $F_{s,t}$ be as above and suppose that $F_{s,t}$ is the $F$-transform of the law $\mu_{s,t} \in \Sigma_0(\mathcal{A})$.  Then
\begin{equation}
\sup_{s < t} \rad(\mu_{s,t}) < +\infty.
\end{equation}
\end{lemma}

\begin{proof}
By Proposition \ref{prop:twooutofthree}, we know that $F_{s,t}$ is an $F$-transform of a law $\mu_{s,t}$.   A uniform bound on $\rad(\mu_{s,t})$ follows from Proposition \ref{prop:twooutofthree2} and the boundedness of $\mu_t(X)$ as a function of $t$.
\end{proof}

\begin{lemma}
Let $F_t$ be a Loewner chain and let $F_{s,t}$ be as above.  For $0 \leq s \leq t \leq T$, we have
\begin{equation} \label{eq:sigmaSTrepresentation}
F_{s,t}(z) = z - \mu_{s,t}(X) - G_{\sigma_{s,t}}(z),
\end{equation}
where $\sigma_{s,t}$ is a generalized law and
\begin{equation}
\sigma_{s,t}(1) = \mu_{s,t}(X^2) = \mu_t(X^2) - \mu_s(X^2) \leq \mu_T(X^2).
\end{equation}
\end{lemma}

\begin{proof}
By Proposition \ref{prop:Ftransformbijection}, we have $F_{s,t}(z) = z - \mu_{s,t}(X) - G_{\sigma_{s,t}}(z)$ for some generalized law $\sigma_{s,t}$.  Using the power series expansion at $\infty$ (Lemma \ref{lem:Cauchytransformseries}), we see that $\sigma_{s,t}(1) = \mu_{s,t}(X^2)$.  Another power series computation shows that
\begin{equation}
\mu_T(X^2) = \mu_{0,s}(X^2) + \mu_{s,t}(X^2) + \mu_{t,T}(X^2).
\end{equation}
Each of these elements is positive in $\mathcal{A}$ because the laws are completely positive.  Hence, $0 \leq \mu_{s,t}(X^2) \leq \mu_T(X^2)$.
\end{proof}

\subsection{Loewner Chains are Biholomorphic Functions}

\begin{proposition} \label{prop:biholomorphic}
Suppose that $(F_t)_{t \in [0,T]}$ is a Loewner chain and $F_t = F_s \circ F_{s,t}$.  Then $F_{s,t}^{(n)}$ is a biholomorphic map from $\h^{(n)}(\mathcal{A})$ onto an open $U_{s,t}^{(n)} \subseteq \h^{(n)}(\mathcal{A})$.  The sequence $\{U_{s,t}^{(n)}\}$ is a fully matricial domain and $F_{s,t}^{-1}$ is a fully matricial function.

Furthermore, given $\epsilon > 0$, there exists $\delta > 0$, depending only on $\epsilon$ and the modulus of continuity of $t \mapsto \mu_t(X^2)$, such that for all $0 \leq s \leq t \leq T$ and $n \in \N$,
\begin{enumerate}[(1)]
	\item whenever $z \in \h_\epsilon^{(n)}(\mathcal{A})$, we have $\norm{DF_{s,t}^{(n)}(z)^{-1}} \leq 1/ \delta$;
	\item whenever $z, z' \in \h_\epsilon^{(n)}(\mathcal{A})$, we have
	\begin{equation*}
	\norm{F_{s,t}^{(n)}(z) - F_{s,t}^{(n)}(z')} \geq \delta \norm{z - z'}.
	\end{equation*}
\end{enumerate}
\end{proposition}

\begin{proof}
We start by proving the explicit estimates.  We continue to use the representation \eqref{eq:sigmaSTrepresentation} for $F_{s,t}$.  Fix $\epsilon > 0$.  By continuity of $\mu_t(X^2)$, there exists a $\gamma > 0$ such that
\begin{equation}
|s - t| < \gamma \implies \norm{\sigma_{s,t}(1)} = \norm{\mu_s(X^2) - \mu_t(X^2)} < \frac{\epsilon^2}{2}.
\end{equation}
Therefore, if $|s - t| < \gamma$, then for $\im z, z' \in \h_\epsilon^{(n)}(\mathcal{A})$, we have by \eqref{eq:estimate2}
\begin{equation}
\norm{G_{\sigma_{s,t}}^{(n)}(z) - G_{\sigma_{s,t}}^{(n)}(z')} \leq \frac{\norm{\sigma_{s,t}(1)}}{\epsilon^2} \norm{z - z'} \leq \frac{1}{2} \norm{z - z'}.
\end{equation}
Since $F_{s,t}(z) = z - \mu_{s,t}(X)^{(n)} - G_{\sigma_{s,t}}^{(n)}(z)$, we thus have
\begin{equation} \label{eq:biholomorphicestimate1}
\norm{F_{s,t}^{(n)}(z) - F_{s,t}^{(n)}(z')} \geq \norm{z - z'} - \frac{1}{2} \norm{z - z'} = \frac{1}{2} \norm{z - z'}.
\end{equation}
This implies that $F_{s,t}^{(n)}$ is injective on $\h_\epsilon^{(n)}(\mathcal{A})$.  Moreover, for $\im z \geq \epsilon$, we have
\begin{equation} \label{eq:biholomorphicestimate2}
\norm{\id - DF_{s,t}^{(n)}(z)} = \norm{DG_{\sigma_{s,t}}^{(n)}(z)} \leq \frac{\mu_{s,t}(X^2)}{\epsilon^2} \leq \frac{1}{2},
\end{equation}
and hence $DF_{s,t}^{(n)}(z)$ is invertible and the norm of its inverse is bounded by $2$.

Now choose any $s < t$.  We can choose a partition $s = t_0 < \dots < t_m = t$ of $[s,t]$ such that $t_j - t_{j-1} < \gamma$ and $m \leq T / \gamma + 1$.  Then we write
\begin{equation}
F_{s,t}^{(n)} = F_{t_0,t_1}^{(n)} \circ \dots \circ F_{t_{m-1},t_m}^{(n)}.
\end{equation}
Because each $F_{t_{j-1},t_j}$ maps $\h_\epsilon^{(n)}(\mathcal{A})$ into $\h_\epsilon^{(n)}(\mathcal{A})$, we can apply our previous estimates \eqref{eq:biholomorphicestimate1} and \eqref{eq:biholomorphicestimate2} iteratively to $F_{t_{j-1},t_j}$ to obtain
\begin{align}
\norm{F_{s,t}^{(n)}(z) - F_{s,t}^{(n)}(z')} &\geq \frac{1}{2^m} \norm{z - z'} \nonumber \\
\norm{DF_{s,t}^{(n)}(z)^{-1}} &\leq 2^m.
\end{align}
Therefore, we have proved (1) and (2) with $\delta = 1/2^m$.

Because $F_{s,t}^{(n)}$ is analytic and $DF_{s,t}^{(n)}(z)$ is invertible, the inverse function theorem for analytic functions between Banach spaces implies that $F^{(n)}(\h^{(n)}(\mathcal{A}))$ is open and the inverse function for $F_{s,t}^{(n)}$ is analytic.  Because $F_{s,t}$ respects direct sums and similarities, so does the region $U^{(n)}$ and so does the inverse function.

In fact, we claim that $(U^{(n)})$ is uniformly open and the function $F_{s,t}^{-1}$ is uniformly locally bounded.  Indeed, the estimates (1) and (2) are independent of $n$.  Moreover, we have uniform bounds on $F_{s,t}^{(n)}|_{\h_\epsilon^{(n)}(\mathcal{A})}$ by Lemma \ref{lem:estimate1} which imply uniform bounds on the derivatives $\delta^k F_{s,t}^{(n)}$ by \eqref{eq:Cauchyestimate}.  Thus, the inverse function theorem will show that for $z \in \h_\epsilon^{(n)}(\mathcal{A})$, the inverse function $(F_{s,t}^{(n)})^{-1}$ is defined in a ball $B_{M_n(\mathcal{A})}(z,R)$ and bounded by $C$, where $R$ and $C$ depend only on $\epsilon$ and the modulus of continuity of $t \mapsto \mu_t(X^2)$.  In particular, the same $R$ and $C$ will work for $z^{(m)}$ for each $m \in \N$.  Thus, $(U^{(n)})$ is fully matricial domain and the function $F_{s,t}^{-1}$ is fully matricial.  Alternatively, one can appeal to the fully matricial inverse function theorem of \cite[Theorem 1.4]{AKV2015}.
\end{proof}

\subsection{Loewner Chains and Monotone Independence}

In this section, we will discuss monotone independence and establish a correspondence between Loewner chains and processes with monotone independent increments.  In the scalar case, monotone independence was defined and studied by Muraki \cite{Muraki2000,Muraki2001}.  For background on operator-valued monotone independence, refer to \cite{Popa2008a,BPV2013,HS2014,AW2016}.  Our treatment will follow \cite{Popa2008a}.  We start with some convenient auxiliary definitions.

\begin{definition}
If $\mathcal{B}$ is a $C^*$-algebra containing $\mathcal{A}$, then we say that $\mathcal{C} \subseteq \mathcal{B}$ is an \emph{$\mathcal{A}$-subalgebra} if it is a $*$-subalgebra and $\mathcal{A} \cdot \mathcal{C} \cdot \mathcal{A} \subseteq \mathcal{C}$.  Note that $\mathcal{C}$ is not assumed to be unital even if $\mathcal{A}$ and $\mathcal{B}$ are.
\end{definition}

\begin{definition}
If $\mathcal{B}$ is a $C^*$-algebra containing $\mathcal{A}$ and $X \in \mathcal{B}$, then we denote by $\mathcal{A} \ip{X}_0$ the $\mathcal{A}$-algebra generated by $X$.  In other words, $\mathcal{A}\ip{X}_0$ is the space of non-commutative polynomials of $X$ with coefficients in $\mathcal{A}$ which have no degree-zero term.
\end{definition}

\begin{definition}
Let $(\mathcal{B}, E)$ be an $\mathcal{A}$-valued non-commutative probability space.  We say that $\mathcal{A}$-subalgebras $\mathcal{B}_1$, \dots, $\mathcal{B}_N$ are \emph{monotone independent} over $\mathcal{A}$ if, given $m \geq 2$ and $b_1 \in \mathcal{B}_{k_1}$, \dots, $b_m \in \mathcal{B}_{k_m}$, we have
\begin{equation}
E[b_1 \dots b_m] = E[b_1 \dots b_{j-1} E[b_j] b_{j+1} \dots b_m]
\end{equation}
provided that
\begin{equation}
\begin{cases} k_j > k_{j-1} \text{ and } k_j > k_{j+1}, & 1 < j < m, \\ k_j > k_{j+1}, & j = 1 \\ k_j > k_{j-1}, & j = m\end{cases}
\end{equation}
\end{definition}

\begin{definition} \label{def:monotoneindependence}
Let $(\mathcal{B}, E)$ be an $\mathcal{A}$-valued non-commutative probability space.  We say that self-adjoint elements $X_1$, \dots, $X_N$ in $\mathcal{B}$ are \emph{monotone independent} if the $\mathcal{A}$-subalgebras $\mathcal{A}\ip{X_j}_0$ are monotone independent.
\end{definition}

The next proposition, due to \cite{Muraki2000} {Muraki2001} in the scalar-valued case, shows that $F$-transform of a sum of monotone independent random variables is the composition of the individual $F$-transforms.  The operator-valued case was proved combinatorially in \cite[Theorem 3.7]{Popa2008a} and we give a similar, but shortened, analytic proof.

\begin{proposition} \label{prop:monoindcomposition}
Let $X$ and $Y$ be self-adjoint elements of the $\mathcal{A}$-valued probability space $(\mathcal{B}, E)$, and suppose that $X$ and $Y$ are monotone independent.  If $F_X$, $F_Y$, and $F_{X+Y}$ are the reciprocal Cauchy transforms of $X$, $Y$, and $X + Y$ respectively, then $F_{X+Y} = F_X \circ F_Y$.
\end{proposition}

The proof relies on the following observation:

\begin{lemma} \label{lem:monoindconst}
If $\mathcal{B}_1$, \dots, $\mathcal{B}_N$ are monotone independent over $\mathcal{A}$, then $\mathcal{B}_1$, \dots, $\mathcal{B}_{n-1}$, $\mathcal{B}_N + \mathcal{A}$ are monotone independent over $\mathcal{A}$.  (Note that $\mathcal{B}_N + \mathcal{A}$ is the unit $\mathcal{A}$-algebra generated by $\mathcal{B}_N$.)
\end{lemma}

\begin{proof}
Straightforward casework left to the reader.
\end{proof}

\begin{proof}[Proof of Proposition \ref{prop:monoindcomposition}]
Suppose that $X$ and $Y$ are monotone independent.  It is straightforward to check that $X^{(n)}$ and $Y^{(n)}$ are monotone independent in the $M_n(\mathcal{A})$-valued probability space $(M_n(\mathcal{B}),E^{(n)})$.  Note that for $z \in M_n(\mathcal{A})$ with $\norm{z} < 1/ \norm{X + Y}$, we have
\begin{align}
(z^{-1} - (X^{(n)} + Y^{(n)}))^{-1} &= ((z^{-1} - X^{(n)}) - Y^{(n)})^{-1} \nonumber \\
&= \sum_{k=0}^\infty [(z^{-1} - Y^{(n)})^{-1} X^{(n)}]^k (z^{-1} - Y^{(n)})^{-1}.
\end{align}
We take the expectation and observe that $(z^{-1} - Y^{(n)})^{-1} - z$ is in the closure of $M_n(\mathcal{A})\ip{Y^{(n)}}_0$.  Thus, $X^{(n)}$ and $(z^{-1} - Y^{(n)})^{-1} - z$ are monotone independent.  But this implies that $X^{(n)}$ and $(z - Y^{(n)})^{-1}$ are monotone independent by Lemma \ref{lem:monoindconst}.  Thus,
\begin{align}
\tilde{G}_{X+Y}^{(n)}(z) &= E^{(n)}[(z^{-1} - X^{(n)} - Y^{(n)})^{-1}] \nonumber \\
&= \sum_{k=0}^\infty E^{(n)}\biggl(\bigl[E^{(n)}[(z^{-1} - Y^{(n)})^{-1}] X^{(n)} \bigr]^k E^{(n)}[(z^{-1} - Y^{(n)})^{-1}] \biggr) \nonumber \\
&= \tilde{G}_X^{(n)} \circ \tilde{G}_Y^{(n)}(z).
\end{align}
Because $F_X = \inv \circ \tilde{G}_X \circ \inv$, where $\inv: z \mapsto z^{-1}$, we have $F_{X + Y} = F_X \circ F_Y$.
\end{proof}

\begin{definition}
If $X$ and $Y$ are monotone independent, and $X$ has law $\mu$ and $Y$ has law $\nu$, then we call the law of $X + Y$ the \emph{monotone convolution} of $\mu$ and $\nu$, and denote it $\mu \rhd \nu$.  Note that $\mu \rhd \nu$ is well-defined because the $F$-transform of $X + Y$ is uniquely determined by the previous proposition.
\end{definition}

Now we can prove the operator-valued analogue of \cite[Theorem 3.1]{Schleissinger2017}, which shows that $\mathcal{A}$-valued Loewner chains are equivalent to processes with monotone independent increments (see also \cite[\S 3.1]{Hasebe2010-2}).

\begin{definition}
Let $(\mathcal{B}, E)$ be an $\mathcal{A}$-valued probability space, and let $(X_t)_{t \in [0,T]}$ be a family of self-adjoint elements of $\mathcal{B}$.  Then $(X_t)$ is said to be a \emph{process with monotone independent increments} if for every $t_0 < t_1 < \dots < t_N$, the random variables $X_{t_0}$, $X_{t_1} - X_{t_0}$, \dots, $X_{t_N} - X_{t_{N-1}}$. are monotone independent, and $t \mapsto E(X_t)$ and $t \mapsto E(X_t^2)$ are continuous functions $[0,T] \to \mathcal{A}$.
\end{definition}

\begin{theorem} \label{thm:processLoewnerchain}
Suppose that $(X_t)_{t \in [0,T]}$ is a process with monotone independent increments in an $\mathcal{A}$-valued probability space $(\mathcal{B}, E)$, then the $F$-transforms $F_{X_t}^{(n)}(z) = (E^{(n)}[(z - X_t)^{-1}])^{-1}$ form a chordal Loewner chain.  Conversely, every chordal Loewner chain arises from such a process with monotone independent increments on some $\mathcal{A}$-valued probability space.
\end{theorem}

\begin{proof}
Suppose $(X_t)_{t \in [0,T]}$ is a process with monotone independent increments.  Then for $s < t$, we have
\begin{equation}
F_{X_t} = F_{X_s} \circ F_{X_t - X_s}.
\end{equation}
We also assumed that $E(X_t)$ and $E(X_t^2)$ are continuous.  Hence, $F_{X_t}$ is a chordal Loewner chain.

Conversely, suppose that $F_t$ is a chordal Loewner chain.  Note that $F_{s,t}$ is the reciprocal Cauchy transform of a law $\mu_{s,t}$.  By Theorem \ref{prop:CPmap}, we can construct a representative random variable $X_{s,t}$ in a $C^*$ probability space $(\mathcal{B}_{s,t}, E_{s,t})$ which has the law $\mu_{s,t}$.  Now consider a partition $P = (t_0,\dots, t_N)$ where $0 = t_0 < t_1 < \dots < t_N = T$.  Let $\mathcal{B}_P$ be the monotone product of the algebras $\mathcal{B}_{0,t_1}$, \dots, $B_{t_{N-1},t_N}$ as constructed in \cite[\S 4]{Popa2008a}.  Then the random variables $X_{t_0,t_1}$, \dots, $X_{t_{N-1}, t_N}$ are monotone independent in $\mathcal{B}_P$.  Moreover, for $i < j$, the variable $X_{t_i,t_j} = X_{t_i,t_{i+1}} + \dots + X_{t_{j-1},t_j}$ in $\mathcal{B}_P$ has the law $\mu_{t_i,t_j}$.

We define $\mathcal{B}$ as the $C^*$ inductive limit (see \cite[\S II.8.2]{Blackadar2006}) of $\mathcal{B}_P$, where the index set is the collection of partitions $P$ of $[0,T]$ ordered by inclusion, and where for $P \subseteq P'$ the inclusion map $\mathcal{B}_P \to \mathcal{B}_{P'}$ is given by the mapping the variable $X_{s,t}$ of $\mathcal{B}_P$ to the variable $X_{s,t}$ of $\mathcal{B}_{P'}$ for every pair $s < t$ in $P$.  Then $X_t = X_{0,t}$ in $\mathcal{B}$ is the desired process with monotone independent increments such that $X_t$ has the law $\mu_t$.
\end{proof}

\begin{remark}
In Theorem \ref{thm:Fockspace}, under additional assumptions on $\mu_t$, we will provide a more direct construction of a process with monotone independent increments that realizes the family of laws $\mu_t$.
\end{remark}

\begin{remark}
Anit-monotone independence is defined the same as monotone independence with the order of the indices reversed (that is, $\mathcal{B}_1$, \dots, $\mathcal{B}_N$ are anti-monotone independent if and only if $\mathcal{B}_N$, \dots, $\mathcal{B}_1$ are monotone independent).  If $F_t$ is a Loewner chain on $[0,T]$, then the time-reversed flow $F_{T-t,T}$ describes a process with anti-monotone independent increments.  We refer to \cite[\S 3.3]{Schleissinger2017} for a treatment of the scalar case, and the generalization to the operator-valued case is straightforward.
\end{remark}

\subsection{Loewner Chains and Free Independence} \label{subsec:freeincrements}

Now we will adapt Bauer's observation that processes with freely independent increments give rise to Loewner chains \cite{Bauer2004} \cite[\S 3.5]{Schleissinger2017}.  We first review the notion of operator-valued free independence from \cite{VDN1992,Speicher1998,MS2013,BMS2013}

\begin{definition}
Let $(\mathcal{B}, E)$ be an $\mathcal{A}$-valued probability space.  Let $\mathcal{B}_1$, \dots, $\mathcal{B}_N$ be subalgebras of $\mathcal{B}$ which contain $\mathcal{A}$.  Then $\mathcal{B}_1$, \dots, $\mathcal{B}_N$ are said to be \emph{freely independent over $\mathcal{A}$} if we have
\begin{equation}
E(b_1, \dots, b_m) = 0
\end{equation}
whenever $E(b_j) = 0$ and $b_j \in \mathcal{B}_{k_j}$ with $k_j \neq k_{j+1}$.
\end{definition}

\begin{definition}
Self-adjoint random variables in $X_1$, \dots, $X_N$ of $(\mathcal{B},E)$ are said to be \emph{freely independent} if the unital $*$-subalgebras $\mathcal{A}\ip{X_1}$, \dots, $\mathcal{A}\ip{X_N}$ are freely independent.
\end{definition}

\begin{definition}
If $X$ and $Y$ are freely independent over $\mathcal{A}$, and $\mu$ and $\nu$ are the laws of $X$ and $Y$ respectlvely, then we denote the law of $X + Y$ by $\mu \boxplus \nu$.  This is known to be well-defined.
\end{definition}

An important consequence of free independence is the analytic subordination of Cauchy transforms.  This was proved in the scalar case by \cite[Proposition 4.3]{Voiculescu1993} \cite[Theorem 3.1]{Biane1998}, \cite{Lenczewski2007}.  Versions of the result were proved in the multivariable setting by \cite{Nica2009} and in the operator-valued setting by \cite{Voiculescu2000}.  The statement here concerning subordination for fully matricial functions defined on all of $\h(\mathcal{A})$ is due to \cite[Theorem 2.2]{BMS2013}.

\begin{theorem} \label{thm:analyticsubordination}
Suppose that $X$ and $Y$ are freely independent in $(\mathcal{B}, E)$.  Then there exists a fully matricial function $F: \h(\mathcal{A}) \to \h(\mathcal{A})$ such that $G_{X+Y}(z) = G_X \circ F(z)$, and hence $F_{X+Y}(z) = F_X \circ F(z)$.
\end{theorem}

\begin{definition}
An \emph{$\mathcal{A}$-valued process with freely independent increments} on $[0,T]$ is a family of self-adjoint elements $\{X_t\}_{t \in [0,T]}$ in $(\mathcal{B},E)$ such that for each $t_0 < t_1 < \dots < t_N$, the increments $X_{t_j} - X_{t_{j-1}}$ are freely independent, and such that $E(X_t)$ and $E(X_t^2)$ are continuous functions of $t$.
\end{definition}

\begin{theorem} \label{thm:freeLoewnerchain}
Suppose that $\{X_t\}_{t \in [0,T]}$ is an $\mathcal{A}$-valued process with freely independent increments in $(\mathcal{B},E)$ such that $X_0 = 0$ and $E(X_t)$ and $E(X_t^2)$ are continuous. Then $F_t := (E[(z - X)^{-1}])^{-1}$ is a Loewner chain.
\end{theorem}

\begin{proof}
This is immediate from Theorem \ref{thm:analyticsubordination}.
\end{proof}

However, even in the scalar-valued case, not every Loewner chain arises from a process with freely independent increments.  We demonstrate this with a concrete counterexample.  First, we construct a counterexample to the converse of the subordination property.  Here we use some standard notation and results from free probability theory concerning the free cumulants and $R$-transforms; for explanation, see e.g.\ \cite[Part II]{NS2006}.

\begin{lemma} \label{lem:counterexample}
Let $\sigma$ be a semicircular distribution, with probability density given by $\frac{1}{2\pi} \sqrt{4 - t^2} \chi_{[-2,2]}(t)$.  There does not exist a law $\nu$ such that $\sigma \rhd \sigma = \sigma \boxplus \nu$.
\end{lemma}

\begin{proof}
Recall that $F_\sigma^{-1}(z)$ and $G_\sigma^{-1}(z)$ are both given by $z + 1/z$ on the appropriate domains.  Note that $G_\mu = G_\sigma \circ F_\sigma$, so that
\begin{equation}
\frac{1}{z} + R_\mu(z) = G_\mu^{-1}(z) = F_\sigma^{-1} \circ G_\sigma^{-1}(z) = z + \frac{1}{z} + \frac{1}{z + 1/z} = \frac{1}{z} + z + \frac{z}{1 + z^2}.
\end{equation}
If we assume that $\mu = \sigma \boxplus \nu$, then because the $R$-transform linearizes free convolution, we have
\begin{equation}
R_\nu(z) = R_\mu(z) - R_\sigma(z) = \frac{z}{1 + z^2} = \sum_{n=0}^\infty (-1)^n z^{2n+1}.
\end{equation}
Thus, the free cumulants of $\nu$ would be
\begin{equation}
\kappa_{2n} = (-1)^{n+1} \text{ for } n \geq 1, \qquad \kappa_{2n+1} = 0.
\end{equation}
When we compute the sixth moment of $\nu$ by the moment-cumulant formula, we obtain the contradiction
\begin{equation}
\nu(t^6) = \kappa_6 + 6 \kappa_4 \kappa_2 + 5 \kappa_2^3 = 0. \qedhere
\end{equation}
\end{proof}

Now we construct a Loewner chain on $[0,2]$ as follows:  Let $\sigma_t$ be the semicircular distribution of variance $t$.  Then we define $F_t$ on $[0,1]$ by $F_t = F_{\sigma_t}$.  This family of reciprocal Cauchy transforms arises from a process $S_t$ whose increments are freely independent, such that $S_t - S_s$ is semicircular of variance $t - s$.  Therefore, $\{F_t\}_{t \in [0,1]}$ is a Loewner chain.  Now, for in $t \in [1,2]$, we define
\begin{equation}
F_t = F_1 \circ F_{t - 1}.
\end{equation}
Then $F_t$ is a Loewner chain on $[0,2]$.  However, it does not arise from a process with free increments; indeed, the law at time $1$ is $\sigma$, the law at time $2$ is $\sigma \rhd \sigma$, and we showed in Lemma \ref{lem:counterexample} that $\sigma \rhd \sigma$ cannot be expressed as $\sigma \boxplus \nu$ for any law $\nu$.

\section{The Loewner Equation} \label{sec:Loewnerequation}

This section will prove the operator-valued version of \cite[Theorems 5.3 and 5.6]{Bauer2005}, that is, we will show that the Loewner equation
\begin{equation} \label{eq:Loewner}
\partial_t F(z,t) = DF(z,t)[V(z,t)]
\end{equation}
defines a bijection between \emph{Lipschitz normalized Loewner chains} $F(z,t) = F_t(z)$ (see Definition \ref{def:LNCLC}) and \emph{distributional Herglotz vector fields} $V(z,t)$ (see Definition \ref{def:Herglotz}).  Here $V^{(n)}(z,\cdot)$ is allowed to be an element of $\mathcal{L}(L^1[0,T], M_n(\mathcal{A}))$ rather than a pointwise defined function of $t$, as in \S \ref{subsec:distderiv}, and the time derivative is computed in a distributional sense.

The section is organized as follows: \S \ref{subsec:LipschitzLoewnerchain} discusses normalized, Lipschitz Loewner chains and distributional Herglotz vector fields, \S \ref{subsec:differentiation} explains how to differentiate a Loewner chain to find the Herglotz vector field $V(z,t)$, and \S \ref{subsec:integration} constructs a solution $F(z,t)$ to the Loewner equation for a given Herglotz vector field $V(z,t)$.  Finally, \S \ref{subsec:semigroups} relates our results to prior work on free and monotone convolution semigroups.

Many of the proofs here follow the approach that \cite{Bauer2005} used for the scalar case (see also \cite[\S 3]{Bauer2004}).  The author of this paper also worked off the lectures by Mario Bonk at UCLA in Fall 2016.  These proofs in turn were inspired by the theory of Loewner chains in the disk of \cite[\S 6.1]{Pommerenke1975} and \cite{RR1994}.

\subsection{Definitions} \label{subsec:LipschitzLoewnerchain}

For simplicity, we will restrict our attention to Loewner chains $F_t = F_{\mu_t}$ such that $\mu_t$ has mean zero (that is, $\mu_t(X) = 0$).  This is justified by the following observation.  Suppose that $F_t$ is an arbitrary Loewner chain.  Let $\Phi_t(z) = F_t(z + \mu_t(X))$.  Then $\Phi_t$ is a Loewner chain.  Indeed, the subordination functions for this Loewner chain are given by $\Phi_{s,t}(z) = F_{s,t}(z + \mu_t(X)) - \mu_s(X)$ for $0 \leq s \leq t \leq T$.  Moreover, the laws corresponding to the Loewner chain $\Phi_t$ have mean zero.

Now consider a Loewner chain $F_t = F_{\mu_t}$ with mean zero.  In order to be able to differentiate with respect to $t$, we make the further technical assumption that after some reparametrization of time, the function $t \mapsto \mu_t(X^2)$ is an absolutely continuous map $[0,T] \to \mathcal{A}$.  Under this assumption, we can define $f: [0,T] \to [0,+\infty)$ as the total variation of $\mu_t(X^2)$ up to time $t$.  Then by a time reparametrization based on $f$, we may assume that $\mu_t(X^2)$ is Lipschitz in $t$.  This leads us to the following definition.

\begin{definition} \label{def:LNCLC}
A \emph{Lipschitz normalized chordal Loewner chain} is a chordal Loewner chain $F_t = F_{\mu_t}$ such that $\mu_t(X) = 0$ and such that there exists $C > 0$ such that $\norm{\mu_s(X^2) - \mu_t(X^2)} \leq C|s - t|$ for all $s$, $t \in [0,T]$.  We will usually abbreviate the name to \emph{Lipschitz Loewner chain} since no confusion will result.
\end{definition}

\begin{remark}
We remark that the absolute continuity condition on $\mu_t(X^2)$ is similar to the ``order $d$'' condition on evolution families in the unit disk \cite[Definition 1.2]{CDG2009}.  In fact, given the setup of \S \ref{sec:Loewnerchains}, absolute continuity is equivalent to the Loewner chain being order $1$, as we show in \ref{sec:order1}.
\end{remark}

\begin{definition} \label{def:Herglotz}
A \emph{distributional Herglotz vector field on $[0,T]$} is a sequence of functions $V^{(n)}: \h^{(n)}(\mathcal{A}) \times L^1[0,T] \to M_n(\mathcal{A})$, denoted
\[
(z, \phi) \mapsto \int V^{(n)}(z,t) \phi(t)\,dt,
\]
which satisfy the following conditions.
\begin{enumerate}[(1)]
	\item For each $z \in \h^{(n)}(\mathcal{A})$, we have $V^{(n)}(z,\dot) \in \mathcal{L}(L^1[0,T], M_n(\mathcal{A}))$.
	\item For each $\phi \in L^1[0,T]$, the sequence of functions $\int V^{(n)}(\cdot,t)\phi(t)\,dt$ is fully matricial.
	\item If $\phi \geq 0$, then $\int V^{(n)}(\cdot,t) \phi(t)\,dt$ maps $\h(\mathcal{A})$ into $\overline{\h}(\mathcal{A})$.
	\item There exist $R > 0$ and $C > 0$ such that for every $\phi \in L^1[0,T]$, the function $\int V^{(n)}(z^{-1},t) \phi(t)$ has a fully matricial extension from $B_{M_n(\mathcal{A})}(0,R) \cap -\h^{(n)}(\mathcal{A})$ to $B_{M_n(\mathcal{A})}(0,R)$, denoted $\int \tilde{V}^{(n)}(z,t) \phi(t)\,dt$,  which satisfies
	\[
	\int \tilde{V}^{(n)}(z^*,t)\phi(t)\,dt = \left( \int \tilde{V}^{(n)}(z,t) \overline{\phi(t)}\,dt \right)^*
	\]
	and
	\[
	\norm*{ \int \tilde{V}^{(n)}(z,t)\phi(t)\,dt } \leq C \norm{\phi}_{L^1[0,T]}.
	\]
\end{enumerate}
\end{definition}

It follows from Theorem \ref{thm:Cauchytransform} that if $\phi \geq 0$, then $\int V(\cdot,t) \phi(t)\,dt$ is minus the Cauchy transform of some generalized law.  In fact, we will view $V$ as minus the Cauchy transform of a family of generalized laws which depends upon $t$ in a distributional sense.

\begin{definition} \label{def:distributionallaw}
A \emph{distributional family of $\mathcal{A}$-valued generalized laws on $[0,T]$} is a function $\nu: \mathcal{A}\ip{X} \times L^1[0,T] \to \mathcal{A}$, denoted
\[
(f(X),\phi) \mapsto \int \nu(f(X),t)\phi(t)\,dt,
\]
such that
\begin{enumerate}
	\item For each $f(X) \in \mathcal{A}\ip{X}$, we have $\nu(f(X),\cdot) \in \mathcal{L}(L^1[0,T],\mathcal{A})$.
	\item For each $\phi \geq 0$ in $L^1[0,T]$, the function $\int \nu(\cdot,t)\phi(t)\,dt$ is a generalized $\mathcal{A}$-valued law.
\end{enumerate}
\end{definition}

It would seem natural to assume that $\rad(\int \nu(\cdot,t)\phi(t)\,dt)$ is bounded independent of $\phi$ for $\phi \geq 0$, but this turns out to be automatic.

\begin{lemma} \label{lem:generalizedlawradius}
Suppose that $\nu$ is a distributional family of generalized laws.  Then for $\phi \geq 0$, we have
\[
\rad\left( \int \nu(\cdot,t) \phi(t)\,dt \right) \leq \rad \left( \int \nu(\cdot,t) 1 \,dt \right).
\]
\end{lemma}

\begin{proof}
Let $\nu_\phi$ denote the generallized law $\int \nu(\cdot,t) \phi(t)\,dt$.  Note that if $\phi \leq \psi$, then $\nu_\psi = \nu_\phi + \nu_{\psi - \phi}$, and hence by Lemma \ref{lem:maxradius}, $\rad(\nu_\phi) \leq \rad(\nu_\psi)$.  In particular, $\rad(\nu_{\chi_{[a,b]}}) \leq \rad(\nu_1)$.  Hence, we have
\begin{align*}
\norm{\nu_{\chi_{[a,b]}}(a_0 X a_1 \dots X a_k)} &\leq \norm{a_0} \dots \norm{a_k} \rad(\nu_1)^k \norm{\nu_{\chi_{[a,b]}}(1)} \\
&\leq \norm{a_0} \dots \norm{a_k} \rad(\nu_1)^k \norm{\nu(1,\cdot)}_{\mathcal{L}(L^1[0,T],\mathcal{A})} \norm{\chi_{[a,b]}}_{L^1[0,T]}.
\end{align*}
It follows by Lemma \ref{lem:smallerror} that for every $\phi \in L^1[0,T]$,
\[
\norm{\nu_\phi(a_0 X a_1 \dots X a_k)} \leq \norm{a_0} \dots \norm{a_k} \rad(\nu_1)^k \norm{\nu(1,\cdot)}_{\mathcal{L}(L^1[0,T],\mathcal{A})} \norm{\phi}_{L^1[0,T]}.
\]
In particular, for $\phi \geq 0$, we have $\rad(\nu_\phi) \leq \rad(\nu_1)$.
\end{proof}

We now state the correspondence between Herglotz vector fields and distributional families of generalized laws, as well as the analogues of Lemmas \ref{lem:estimate1} and \ref{lem:estimates} in this setting.

\begin{proposition} \label{prop:Herglotzbijection}
There is a bijective correspondence $V \leftrightarrow \nu$ between distributional Herglotz vector fields and distributional families of generalized laws, given by the condition that for each $\phi \geq 0$ in $L^1[0,T]$, the function $\int V(\cdot,t)\phi(t)\,dt$ is minus the Cauchy transform of the generalized law $\int \nu(\cdot,t)\phi(t)\,dt$.
\end{proposition}

\begin{proof}
Let $V$ be a distributional Herglotz vector field and let $C$ and $R$ be as in Definition \ref{def:Herglotz}.  Then for $\phi \geq 0$ in $L^1[0,T]$, Theorem \ref{thm:Cauchytransform} implies that $-\int V(\cdot,t)\phi(t)\,dt$ is the Cauchy transform of some generalized law $\nu_\phi$ with $\rad(\nu_\phi) \leq 1 / R$.  For nonnegative $L^1[0,T]$ functions $\phi_1$ and $\phi_2$, we see that the Cauchy transform of $\nu_{\phi_1} + \nu_{\phi_2}$ is $\int V(\cdot,t)(\phi_1 + \phi_2)(t)\,dt$, and thus $\nu_{\phi_1+\phi_2} = \nu_{\phi_1} + \nu_{\phi_2}$.  We also observe that
\[
\norm{\nu_{\phi}(1)} = \norm*{D_z|_{z = 0} \left( \int \tilde{V}^{(1)}(z,t)\phi(t)\,dt \right)} \leq \frac{C}{R} \norm{\phi}_{L^1[0,T]}.
\]
It follows from Theorem \ref{thm:Cauchytransform} and \eqref{eq:improvedradiusbound} that
\[
\norm{\nu_\phi(a_0Xa_1 \dots Xa_k)} \leq \frac{C}{R^{k+1}} \norm{\phi}_{L^1[0,T]}.
\]
Thus, for each $p(X) \in \mathcal{A}\ip{X}$, the function $\phi \mapsto \nu_\phi(p(X))$ extends to a bounded linear function on $L^1[0,T]$, which we can see by breaking a function $\phi \in L^1[0,T]$ into its positive/negative real/imaginary parts as in the standard construction of the integrals in measure theory.  We denote this extension by $\phi \mapsto \int \nu(p(X),t)\phi(t)\,dt$.  In this way, we have defined a generalized law $\nu$ corresponding to $V$, and this $\nu$ is unique because $\int V(\cdot,t)\phi(t)\,dt$ uniquely determines $\int \nu(\cdot,t)\phi(t)\,dt$ for $\phi \geq 0$.

For the converse direction, suppose that $\nu$ is a distributional family of generalized laws.  Then for $\phi \geq 0$ in $L^1[0,T]$, let $V_\phi$ be minus the Cauchy transform of $\int \nu(\cdot,t)\phi(t)\,dt$.  Also, note by Lemma \ref{lem:estimates}
\[
z \in \h_\epsilon^{(n)}(\mathcal{A}) \implies \norm{V_\phi^{(n)}(z)} \leq \frac{\norm{\int \nu(1,t)\phi(t)\,dt}}{\epsilon} \leq \frac{\norm{\nu(1,\cdot)}_{\mathcal{L}(L^1[0,T],\mathcal{A})}}{\epsilon} \norm{\phi}_{L^1[0,T]}.
\]
Since $V_\phi$ satisfies this estimate and $V_{\phi_1+\phi_2} = V_{\phi_1} + V_{\phi_2}$ for $\phi_1$ and $\phi_2 \geq 0$, we can extend the map $\phi \mapsto V_\phi(z)$ to a bounded linear map $L^1[0,T] \to M_n(\mathcal{A})$.  We denote this extension by $\phi \mapsto \int V^{(n)}(z,t)\phi(t)\,dt$.  Then (1) of Definition \ref{def:Herglotz} holds, and (3) also holds by construction.  Now (2) holds for $\phi \geq 0$ and hence for all $\phi$ by linearity.  Finally, it is straightforward from Lemmas \ref{lem:Cauchytransformseries} and \ref{lem:generalizedlawradius} to check that (4) holds if we choose $R$ such that $1 / R > \rad(\int \nu(\cdot,t)\,dt)$ and set $C = \norm{\nu(1,\cdot)}_{\mathcal{L}(L^1[0,T],\mathcal{A})} / (1/R - M)$.  So $V$ is a distributional Herglotz vector field.  It is clear that $V$ is uniquely determined by $\nu$.
\end{proof}

\begin{notation}
If $V$ is a distributional Herglotz vector field and $\nu$ is the corresponding distributional family of generalized laws on $[0,T]$, we denote
\[
\rad(V) := \rad(\nu) := \rad \left( \int \nu(\cdot, t)\,dt \right).
\]
\end{notation}

As a consequence of the Cauchy transform representation, we have the following estimates.

\begin{lemma} \label{lem:estimates2}
Let $V(z,t)$ be a distributional Herglotz vector field with $\rad(V) \leq M$ and let $R_\epsilon^{(n)}$ be the region where $\im z \geq \epsilon$ or $\norm{z^{-1}} \leq 1/(M + \epsilon)$ as in \eqref{eq:Rndef}.  Then for $z$ and $z' \in R_\epsilon^{(n)}$, we have
\begin{align}
\norm{V^{(n)}(z,\cdot)}_{\mathcal{L}(L^1[0,T],M_n(\mathcal{A}))}
&\leq \frac{1}{\epsilon} \norm{D\tilde{V}^{(1)}(0,\cdot)[1]}_{\mathcal{L}(L^1[0,T],\mathcal{A})} \label{eq:estimate2a} \\
\norm{V^{(n)}(z,\cdot) - V^{(n)}(z',\cdot)}_{\mathcal{L}(L^1[0,T],M_n(\mathcal{A}))}
&\leq \frac{1}{\epsilon^2} \norm{D\tilde{V}^{(1)}(0,\cdot)[1]}_{\mathcal{L}(L^1[0,T],\mathcal{A})}  \norm{z - z'} \label{eq:estimate2b} \\
\norm{DV^{(n)}(z,\cdot)}_{\mathcal{L}(L^1[0,T],M_n(\mathcal{A}))}
&\leq \frac{1}{\epsilon} \norm{D\tilde{V}^{(1)}(0,\cdot)[1]}_{\mathcal{L}(L^1[0,T],\mathcal{A})} \label{eq:estimate2c},
\end{align}
where $D\tilde{V}^{(1)}(0,\cdot)[1]$ represents the Fr{\'e}chet derivative of $\tilde{V}^{(1)}$ as a function $B_{\mathcal{A}}(0,1/M) \to \mathcal{L}(L^1[0,T],\mathcal{A})$ evaluated at the point $0$ and applied to the vector $1 \in \mathcal{A}$.
\end{lemma}

\begin{proof}
Let $\nu$ be the distributional family of generalized laws on $[0,T]$ corresponding to $V$.  Suppose that $[a,b] \subseteq [0,T]$.  Since $\int V(\cdot,t) \chi_{[a,b]}(t)\,dt$ is the Cauchy transform of $\int \nu(\cdot,t) \chi_{[a,b]}(t)\,dt$, we have by \eqref{eq:estimate1} that for $z \in R_\epsilon^{(n)}$,
\begin{align*}
\norm*{\int V^{(n)}(z,t) \chi_{[a,b]}(t)\,dt} &\leq \frac{1}{\epsilon} \norm*{\int \nu(1,t)\chi_{[a,b]}(t)\,dt} \\
&= \frac{1}{\epsilon} \norm*{ \int D\tilde{V}^{(1)}(0,t)[1] \chi_{[a,b]}(t)\,dt } \\
&\leq \frac{1}{\epsilon} \norm*{D\tilde{V}^{(1)}(0,\cdot)[1]}_{\mathcal{L}(L^1[0,T],\mathcal{A})} \norm{\chi_{[a,b]}}_{L^1[0,T]}.
\end{align*}
Since this holds for all $[a,b] \subseteq [0,T]$, Lemma \ref{lem:smallerror} yields the estimate \eqref{eq:estimate2a}.  Similarly, the relations \eqref{eq:estimate2b} and \eqref{eq:estimate2c} are proved by applying \eqref{eq:estimate2} and \eqref{eq:estimate3} to $\int \nu(\cdot,t) \chi_{[a,b]}(t)\,dt$ and then invoking Lemma \ref{lem:smallerror}.
\end{proof}

\subsection{Differentiation of Loewner Chains} \label{subsec:differentiation}

Throughout this subsection, $F^{(n)}_t(z) = F^{(n)}(z,t)$ will be a Lipschitz, normalized chordal Loewner chain, and $F_{s,t}$ will be the subordination map satisfying $F_t = F_s \circ F_{s,t}$.  Proposition \ref{prop:twooutofthree} implies that $F_{s,t}$ is a reciprocal Cauchy transform.  We denote by $\mu_t$ the law of which $F_t$ is the reciprocal Cauchy transform, and by $\mu_{s,t}$ the law of which $F_{s,t}$ is the reciprocal Cauchy transform.

For the rest of the section, we will write
\begin{align}
F_t(z) &= z + H_t(z) \nonumber \\
F_{s,t}(z) &= z + H_{s,t}(z).
\end{align}
We also define
\begin{equation}
\tilde{H}_t(z) = H_t(z^{-1}) \qquad \tilde{H}_{s,t}(z) = H_{s,t}(z^{-1}).
\end{equation}
Note by Proposition \ref{prop:Ftransformbijection} that there exists a generalized law $\sigma_{s,t}$ such that
\begin{equation}
H_{s,t} = -G_{\sigma_{s,t}}
\end{equation}
In particular, Lemmas \ref{lem:estimate1} and \ref{lem:estimates} apply to $H_{s,t}$.

We let $C$ be the Lipschitz norm of $t \mapsto \mu_t(X^2)$.  We observe as a consequence of $F_t = F_s \circ F_{s,t}$ that
\begin{equation}
\sigma_{s,t}(1) = \sigma_t(1) - \sigma_s(1) = \mu_t(X^2) - \mu_s(X^2),
\end{equation}
and hence
\begin{equation} \label{eq:sigmaSTLipschitzbound}
\norm{\sigma_{s,t}(1)} \leq C|s - t|.
\end{equation}
We deduce the following local Lipschitz estimates.

\begin{lemma} \label{lem:LoewnerlocallyLipschitz}
Let $z, z' \in \h_\epsilon^{(n)}(\mathcal{A})$ and let $0 \leq s \leq t \leq u \leq T$.  Then
\begin{align}
\norm{F_{s,t}^{(n)}(z) - F_{s,t}^{(n)}(z')} &\leq \left(1 + \frac{C|s - t|}{\epsilon^2}\right) \norm{z - z'}, \label{eq:LoewnerlocallyLipschitz1} \\
\norm{F_{s,u}^{(n)}(z) - F_{s,t}^{(n)}(z)} &\leq \left(1 + \frac{CT}{\epsilon^2}\right) \frac{C}{\epsilon} |t -u|. \label{eq:LoewnerlocallyLipschitz2}.
\end{align}
In particular, each $F^{(n)}(z,t)$ is a locally Lipschitz family (Definition \ref{def:Lipschitzfamily}), and $F_{s,t}^{(n)}(z)$ is a locally Lipschitz family with respect to $t$.
\end{lemma}

\begin{proof}
To prove \eqref{eq:LoewnerlocallyLipschitz1}, note that by the definition of $\sigma_{s,t}$ and \eqref{eq:estimate2}, we have
\begin{align*}
\norm*{F_{s,t}^{(n)}(z) - F_{s,t}^{(n)}(z')} &= \norm{z + H_{s,t}^{(n)}(z) - z - H_{s,t}^{(n)}(z')} \\
&= \norm*{G_{\sigma_{s,t}}^{(n)}(z) - G_{\sigma_{s,t}}^{(n)}(z')} \\
&\leq \frac{\norm{\sigma_{s,t}(1)}}{\epsilon^2} \norm{z - z'},
\end{align*}
and as remarked earlier, $\norm{\sigma_{s,t}(1)} \leq C|s - t|$.

To prove \eqref{eq:LoewnerlocallyLipschitz2}, note by \eqref{eq:estimate1}, we have
\[
\norm{H_{t,u}^{(n)}(z)} \leq \frac{\norm{\sigma_{t,u}(1)}}{\epsilon} \leq \frac{C}{\epsilon} |t - u|.
\]
Therefore, we have
\begin{align*}
\norm*{F_{s,u}^{(n)}(z) - F_{s,t}^{(n)}(z)} &= \norm*{F_{s,t}^{(n)}(z + H_{t,u}^{(n)}(z)) - F_{s,t}^{(n)}(z)} \\
&\leq \left(1 + \frac{C|s - t|}{\epsilon^2} \right) \norm{H_{t,u}^{(n)}(z)} \\
&\leq \left(1 + \frac{CT}{\epsilon^2} \right) \frac{C}{\epsilon} |t - u|.
\end{align*}
\end{proof}

The following is the $\mathcal{A}$-valued analogue of \cite[Theorem 5.3]{Bauer2005}.

\begin{theorem} \label{thm:Loewnerdifferentiation}
Given the setup at the beginning of \S \ref{subsec:differentiation}, there exists a distributional Herglotz vector field $V(z,t)$ such that $F$ and $V$ satisfy the Loewner equation \eqref{eq:Loewner}.  Moreover, if $\rad(\sigma_{s,t}) \leq M$ when $t - s$ is sufficiently small, then $\rad(V) \leq M$.
\end{theorem}

\begin{proof} ~

{\bf Step 1:} Because $F^{(n)}(z,t)$ is a locally Lipschitz family (Definition \ref{def:Lipschitzfamily}, $\partial_t F(z,\cdot)$ is defined in the distributional sense.  By Proposition \ref{prop:biholomorphic}, $DF^{(n)}(z,t)$ is invertible and $DF^{(n)}(z,t)^{-1}$ is uniformly bounded for $t \in [0,T]$ and $\im z \geq \epsilon$.  Thus, because $DF^{(n)}(z,t)$ is Lipschitz in $t$ locally uniformly in $z$, we know $DF^{(n)}(z,t)^{-1}$ is also locally Lipschitz (in fact, it is uniformly Lipschitz in $t$ for $\im z \geq \epsilon$ and for all $n \in \N$).  Hence, we can define
\begin{equation}
V^{(n)}(z,t) := DF^{(n)}(z,t)^{-1} \partial_t F^{(n)}(z,t),
\end{equation}
so that the Loewner equation \eqref{eq:Loewner} holds.  Note also by Proposition \ref{prop:biholomorphic}, $V^{(n)}(z,t)$ is uniformly bounded when $\im z \geq \epsilon$.

{\bf Step 2:} It remains to show that $V(z,t)$ is a distributional Herglotz vector field.  Toward this end, we define a sequence of approximations to $V^{(n)}(z,t)$ given by
\begin{equation} \label{eq:Vmdef}
V_m^{(n)}(z,t) = \sum_{j=1}^m \chi_{[t_{j-1},t_j)}(t) \frac{m}{T} H_{t_{j-1},t_j}^{(n)}(z), \text{ where } t_j = Tj/m.
\end{equation}
By \eqref{eq:estimate1} and \eqref{eq:sigmaSTLipschitzbound},
\begin{equation}
\im z \geq \epsilon \implies \norm{H_{t_{j-1},t_j}^{(n)}(z)} \leq \frac{\norm{\sigma_{t_{j-1},t_j}(1)}}{\epsilon} \leq \frac{C}{m \epsilon},
\end{equation}
and hence $V_m^{(n)}(z,\cdot)$ is uniformly bounded in $\mathcal{L}(L^1[0,T], M_n(\mathcal{A}))$ for $z \in \h_\epsilon^{(n)}(\mathcal{A})$.  We claim that for each $\phi \in L^1[0,T]$ and $\epsilon > 0$, we have
\begin{equation} \label{eq:Vconvergence}
\int V_m^{(n)}(z,t) \phi(t)\,dt \to \int V^{(n)}(z,t) \phi(t)\,dt \text{ uniformly for } z \in \h_\epsilon^{(n)}(\mathcal{A}).
\end{equation}
Because $\norm{V_m^{(n)}(z,\cdot)}_{\mathcal{L}(L^1[0,T],M_n(\mathcal{A})}$ is uniformly bounded for $z \in \h_\epsilon^{(n)}$ by \eqref{eq:estimate2a}, it suffices to prove the claim for $\phi$ in a dense subspace of $L^1[0,T]$.

{\bf Step 3:} To prove \eqref{eq:Vconvergence}, let $\epsilon$ be fixed.  Assume that $\phi$ is continuous.  Let
\begin{equation}
\phi_m(t) = \sum_{j=1}^m \chi_{[t_{j-1},t_j)}(t) \phi(t_{j-1}), \text{ where } t_j = Tj / m.
\end{equation}
Then we have for $\im z \geq \epsilon$,
\begin{align}
\int V_m^{(n)}(z,t) \phi_m(t)\,dt &= \sum_{j=1}^m \int_{t_{j-1}}^{t_j} V_m^{(n)}(z,t) \phi_m(t)\,dt \\
&= \sum_{j=1}^m \frac{m}{T} \int_{t_{j-1}}^{t_j} H_{t_{j-1},t_j}^{(n)}(z) \phi(t_{j-1})\,dt \\
&= \sum_{j=1}^m \phi(t_{j-1}) H_{t_{j-1},t_j}^{(n)}(z).
\end{align}
Meanwhile,
\begin{equation}
\int V^{(n)}(z,t) \phi_m(t)\,dt = \sum_{j=1}^m \phi(t_{j-1}) \int_{t_{j-1}}^{t_j} DF^{(n)}(z,t)^{-1} \partial_t F^{(n)}(z,t)\,dt.
\end{equation}
Because $DF^{(n)}(z,t)$ is uniformly Lipschitz in $t$ for $z \in \h_\epsilon^{(n)}(\mathcal{A})$, we have
\begin{equation}
\int_{t_{j-1}}^{t_j} DF^{(n)}(z,t)^{-1} \partial_t F^{(n)}(z,t)\,dt = \int_{t_{j-1}}^{t_j} DF^{(n)}(z,t_j)^{-1} \partial_t F^{(n)}(z,t)\,dt + O(1/m^2).
\end{equation}
By the chain rule (Lemma \ref{lem:chainrule}),
\begin{equation}
\partial_t F^{(n)}(z,t) = \partial_t [F_{t_{j-1}}^{(n)} \circ F_{t_{j-1},t}^{(n)}(z)] = DF_{t_{j-1}}^{(n)}(z) \partial_t [F_{t_{j-1},t}^{(n)}(z)].
\end{equation}
Hence,
\begin{align}
\int_{t_{j-1}}^{t_j} DF^{(n)}(z,t)^{-1} \partial_t F^{(n)}(z,t)\,dt &= F_{t_{j-1},t_j}^{(n)}(z) - F_{t_{j-1},t_{j-1}}^{(n)}(z) + O(1/m^2) \nonumber \\
&= H_{t_{j-1},t_j}^{(n)}(z) + O(1/m^2).
\end{align}
Altogether, we have for $z \in \h_\epsilon^{(n)}(\mathcal{A})$,
\begin{align}
\int_0^T V^{(n)}(z,t) \phi_m(t)\,dt &= \sum_{j=1}^m \phi(t_{j-1}) H_{t_{j-1},t_j}^{(n)}(z) + O(1/m) \nonumber \\
&= \int_0^T V_m^{(n)}(z,t) \phi_m(t)\,dt + O(1/m).
\end{align}
Replacing $\phi_m$ by $\phi$ produces an error $O(\norm{\phi_m - \phi}_{L^1})$ which goes to zero as $m \to \infty$, and therefore \eqref{eq:Vconvergence} holds when $\phi$ is continuous.  By approximation, \eqref{eq:Vconvergence} holds for every $\phi \in L^1[0,T]$.  This guarantees that $\int V(\cdot,t) \phi(t)\,dt$ is fully matricial for each $\phi$.

{\bf Step 4:} We claim that for $\phi \geq 0$, the function $\int V(\cdot,t) \phi(t)\,dt$ maps $\h(\mathcal{A})$ into $\overline{\h}(\mathcal{A})$.  If we assume that $z \in \h_\epsilon^{(n)}(\mathcal{A})$, then
\begin{equation}
\im \int V_m^{(n)}(z,t)\phi(t)\,dt = \sum_{j=1}^m \int_{t_{j-1}}^{t_j} \phi(t)\,dt \, \Bigl( \im H_{t_{j-1},t_j}(z) \Bigr) \geq 0.
\end{equation}
Therefore, taking $m \to \infty$, we obtain $\im \int V^{(n)}(z,t)\phi(t)\,dt \geq 0$ as desired.

{\bf Step 5:} It remains to check condition (4) of Definition \ref{def:Herglotz}, that is, we must show that for each $\phi \in L^1[0,T]$, the function $\int \tilde{V}(z,t)\phi(t)\,dt = \int V(z^{-1},t) \phi(t)\,dt$ extends to be fully matricial in a neighborhood of $0$ (more precisely, fully matricial on $B_{M_n(\mathcal{A})}(0,R)$ for some $R$ independent of $n$).  To this end, we will show that $\tilde{V}_m(z,t) = V_m(z^{-1},t)$ extends to be fully matricial in a neighborhood of $0$, and then apply analytic continuation to show that $\tilde{V}_m^{(n)}(z,t)$ converges in a neighborhood of $0$ as $m \to \infty$.

Let $M$ be such that $\rad(\sigma_{s,t}) \leq M$ when $t - s$ sufficiently small.  Recall that $\norm{\sigma_{s,t}(1)} \leq C|s - t|$.  Hence, by Lemma \ref{lem:Cauchytransformseries}, $\tilde{H}_{s,t} = -\tilde{G}_{\sigma_{s,t}}$ extends to be fully matricial for $\norm{z} < 1 / M$ and satisfies
\begin{equation}
\norm*{\tilde{H}_{s,t}(z)} \leq \frac{C|s - t| \norm{z}}{1 - M\norm{z}}.
\end{equation}
Substituting this in \eqref{eq:Vmdef} yields that for sufficiently large $m$ and for $\norm{z} < 1 / M$,
\begin{equation}
\norm*{\tilde{V}_m^{(n)}(z,\cdot)}_{L_{\Boch}^\infty([0,T],M_n(\mathcal{A})} \leq \frac{C\norm{z}}{1 - M \norm{z}}.
\end{equation}
In particular, for $\phi \in L^1[0,T]$, we have
\begin{equation} \label{eq:tildeVmestimate}
\norm*{ \int \tilde{V}_m^{(n)}(z,t)\phi(t)\,dt } \leq \frac{C \norm{z}}{1 - M \norm{z}} \norm{\phi}_{L^1[0,T]}.
\end{equation}
We already know by \eqref{eq:Vconvergence} that
\begin{equation}
\lim_{m \to \infty} \int \tilde{V}_m^{(n)}(z,t) \phi(t)\,dt = \int \tilde{V}^{(n)}(z,t) \phi(t)\,dt \text{ for } z \in B_{M_n(\mathcal{A})}(0,1/M) \cap -\h_\epsilon^{(n)}(\mathcal{A}),
\end{equation}
where the convergence is locally uniform.  Therefore, by Lemma \ref{lem:analyticcontinuationofconvergence}, the sequence $\int \tilde{V}_m^{(n)}(z,t) \phi(t)\,dt$ converges locally uniformly on all of $B_{M_n(\mathcal{A})}(0,1/M)$.  Hence, $\int \tilde{V}^{(n)}(z,t) \phi(t)\,dt$ has an analytic extension to $B_{M_n(\mathcal{A})}(0,1/M)$, which also defines a fully matricial function, since the property of preserving direct sums and similarities is preserved when taking the limit of a sequence of functions.

The estimate \eqref{eq:tildeVmestimate} can be applied to $\tilde{V}$ by taking $m \to \infty$.  The relation $\tilde{V}_m^{(n)}(z^*,t) = V_m^{(n)}(z,t)^*$ implies in the limit that $\int \tilde{V}^{(n)}(z^*,t)\phi(t)\,dt = (\int \tilde{V}^{(n)}(z,t) \overline{\phi(t)}\,dt)^*$.  Thus, Definition \ref{def:Herglotz} (4) holds with $R = 1/M$.
\end{proof}

\subsection{Integration of the Loewner Equation} \label{subsec:integration}

The following is the $\mathcal{A}$-valued analogue of \cite[Theorem 5.5]{Bauer2005}.

\begin{theorem} \label{thm:Herglotzflow}
Let $V(z,t)$ be a distributional Herglotz vector field on $[0,T]$ with $\rad(V) \leq M$, and let $C = \norm{D\tilde{V}^{(1)}(0,\cdot)[1]}_{\mathcal{L}(L^1[0,T],\mathcal{A})}$.
\begin{enumerate}[(1)]
	\item There exists a unique fully matricial family
	\begin{equation*}
	W: \h(\mathcal{A}) \times [0,T] \to \h(\mathcal{A})
	\end{equation*}
	such that $W^{(n)}(z,t)$ is a locally Lipschitz family for each $n$, and $W$ satisfies
	\begin{equation}
	W^{(n)}(z,0) = z, \qquad \partial_t W^{(n)}(z,t) = V^{(n)}(W(z,t),t).
	\end{equation}
	
	\item $W(z,t)$ is the reciprocal Cauchy transform of a law with radius bounded by $M + \sqrt{2Ct}$.  Moreover, for $u > t$,
	\begin{equation} \label{eq:solutionniceestimate}
	\norm{z} < \frac{1}{M + \sqrt{2Cu}} \implies \norm{\tilde{W}(z,t)^{-1}} \leq \frac{1}{M + \sqrt{2C(u - t)}}.
	\end{equation}
	
	\item Letting $W(z,t) = z + H(z,t)$, we have $H(z,t) =- \sigma_t[(z - X)^{-1}]$, where $\sigma_t$ is a generalized law $\mathcal{A}\ip{X} \to \mathcal{A}$ satisfying
	\begin{equation}
	\rad(\sigma_t) \leq M + \sqrt{2Ct}
	\end{equation}
	and, setting $\tilde{H}(z,t) = H(z^{-1},t)$, we have
	\begin{equation} \label{eq:Herglotzderivativeformula}
	\sigma_t|_{\mathcal{A}} = -D\tilde{H}^{(1)}(0,t) = - \int_0^t \tilde{V}^{(1)}(0,s)\,ds
	\end{equation}
	and hence $\norm{\sigma_t(1)} \leq Ct$.
\end{enumerate}
\end{theorem}

\begin{remark}
The statement and proof here closely follow \cite[Theorem 5.5]{Bauer2005}.
\end{remark}

\begin{proof}
{\bf Step 1:} We define Picard iterates $W_m$ inductively by
\begin{align}
W_0^{(n)}(z,t) &= z  \\
W_{m+1}^{(n)}(z,t) &= z + \int_0^t V^{(n)}(W_m^{(n)}(z,s),s)\,ds. \label{eq:WMdef}
\end{align}
We prove the following claims by induction on $m$:
\begin{enumerate}[(a)]
	\item $W_m^{(n)}(z,t)$ is well-defined and it is a $W_m$ is a fully matricial function of $z$.
	\item $\im W_m^{(n)}(z,t) \geq \im z$.
	\item $W_m^{(n)}(z,\cdot)$ is $(C / \epsilon)$-Lipschitz for $\im z \geq \epsilon$.
\end{enumerate}
Each of the claims is trivial in the base case $m = 0$.

Now assume the claims hold for $m - 1$.  Since $W_{m-1}^{(n)}(z,t)$ is a locally Lipschitz family, we know that $V^{(n)}(W_{m-1}^{(n)}(z,t),t)$ is defined and analytic by Lemma \ref{lem:compositionanalytic}.  It follows that $W_m^{(n)}(z,t)$ is well-defined and analytic.  The fact that it preserves direct sums and similarities is clear.  Moreover, $W_m^{(n)}(z,t)$ is bounded for $\im z \geq \epsilon$, independently of $m$, because $\im W_{m-1}^{(n)}(z,t) \geq \im z \geq \epsilon$ and hence $\norm{V^{(n)}(W_{m-1}^{(n)}(z,t),\cdot)}_{\mathcal{L}(L^1[0,T],M_n(\mathcal{A}))} \leq C / \epsilon$ by \eqref{eq:estimate2a}.  Thus, (a) holds.

Because $\im V^{(n)}(z,t) \geq 0$ in the distributional sense by Definition \ref{def:Herglotz} (3), we conclude that $\im V^{(n)}(W_{m-1}^{(n)}(z,t),t) \geq 0$ also by a step-function approximation argument as in the construction of Lemma \ref{lem:compositionanalytic}.  Hence, (b) holds.

Next, because $\norm{V^{(n)}(z,\cdot)}_{\mathcal{L}(L^1[0,T],M_n(\mathcal{A}))} \leq C / \epsilon$ for $\im z \geq \epsilon$, and we have $\im W_{m-1}(z,t) \geq \im z$, we know that $V^{(n)}(W_{m-1}^{(n)}(z,t),\cdot)$ is bounded by $C / \epsilon$ in $\mathcal{L}(L^1[0,T], M_n(\mathcal{A}))$.  By \eqref{eq:stupidestimate2}, this implies (c).

{\bf Step 2:} To prove convergence of $W_m^{(n)}$ as $m \to \infty$, we will show that for $m > 0$ and $\im z \geq \epsilon$, we have
\begin{equation} \label{eq:PLconvergenceestimate}
\norm{W_m^{(n)}(z,t) - W_{m-1}^{(n)}(z,t)} \leq \frac{C^{m+1}t^m}{m! \epsilon^{2m+1}}
\end{equation}
In the case $m = 1$, this holds because $\norm{V^{(n)}(z,\cdot)}_{\mathcal{L}(L^1[0,T],M_n(\mathcal{A}))} \leq C/\epsilon$ by \eqref{eq:estimate2a}.  For the induction step, we use \eqref{eq:WMdef} together with the Lipschitz bound Lemma \eqref{eq:estimate2b} to argue that
\[
\norm*{W_{m+1}^{(n)}(z,t) - W_m^{(n)}(z,t)} = \norm*{ \int_0^t V^{(n)}(W_m^{(n)}(z,s),s)\,ds - \int_0^t V^{(n)}(W_{m-1}^{(n)}(z,s),s)\,ds }.
\]
Now by \eqref{eq:stupidestimate2}, we have
\begin{multline*}
\norm*{ \int_0^t V^{(n)}(W_m^{(n)}(z,s),s)\,ds - \int_0^t V^{(n)}(W_{m-1}^{(n)}(z,s),s)\,ds } \\ \leq \int_0^t \norm*{V^{(n)}(W_m^{(n)}(z,s),\cdot) - V^{(n)}(W_{m-1}^{(n)}(z,s),\cdot)}_{\mathcal{L}(L^1[0,T],M_n(\mathcal{A}))}\,ds.
\end{multline*}
Because $\im W_m^{(n)}(z,s) \geq \epsilon$ and $\im W_{m-1}^{(n)}(z,s) \geq \epsilon$, we can apply \eqref{eq:estimate2b} to conclude that
\begin{multline*}
\int_0^t \norm*{V^{(n)}(W_m^{(n)}(z,s),\cdot) - V^{(n)}(W_{m-1}^{(n)}(z,s),\cdot)}_{\mathcal{L}(L^1[0,T],M_n(\mathcal{A}))}\,ds \\
\leq \int_0^t \frac{C}{\epsilon^2} \norm*{W_m^{(n)}(z,s) - W_{m-1}^{(n)}(z,s)} \,ds.
\end{multline*}
Now by the inductive hypothesis,
\begin{align*}
\int_0^t \frac{C}{\epsilon^2} \norm*{W_m^{(n)}(z,s) - W_{m-1}^{(n)}(z,s)} \,ds &\leq \int_0^t \frac{C}{\epsilon^2} \cdot \frac{C^{m+1}s^m}{m! \epsilon^{2m+1}} \,ds \\
&= \frac{C^{m+2} t^{m+1}}{(m+1)! \epsilon^{2m+3}},
\end{align*}
which finishes the proof of \eqref{eq:PLconvergenceestimate}.

{\bf Step 3:} Step 2 implies that as $m$ goes to infinity, $W_m^{(n)}(z,t)$ converges uniformly for $n \in \N$, $t \in [0,T]$, and $\im z \geq \epsilon$ to a fully matricial function $W(z,t)$.  This implies that $V^{(n)}(W_m^{(n)}(z,t),t)$ converges to $V^{(n)}(W^{(n)}(z,t),t)$ in $\mathcal{L}(L^1[0,T], M_n(\mathcal{A}))$ by Observation \ref{obs:compositionerror}.  Therefore,
\begin{equation}
W^{(n)}(z,t) = z + \int_0^t V^{(n)}(W^{(n)}(z,s),s)\,ds,
\end{equation}
and hence $\partial_t W^{(n)}(z,t) = V^{(n)}(W^{(n)}(z,t),t)$ in the distributional sense, and $W^{(n)}(z,0) = z$.  Uniqueness of the solution follows from the standard Picard-Lindel\"of argument.  This completes the proof of (1) in the theorem statement.

{\bf Step 4:}  We will prove by induction that given $0 \leq t \leq T$ and $u > t$, the function $\tilde{W}_m^{(n)}(z,t)^{-1} := W_m^{(n)}(z^{-1},t)^{-1}$ extends to be fully matricial on $\norm{z} < 1/(M + \sqrt{2Cu})$ and that it satisfies
\begin{equation}
\norm*{\tilde{W}_m^{(n)}(z^{-1},t)^{-1}} \leq \left(M + \sqrt{2C(u - t)}\right)^{-1}.
\end{equation}
The base case $m = 0$ is trivial.  For the induction step, recall that $\tilde{V}^{(n)}(z^{-1},t) := V^{(n)}(z^{-1},t)$ extends to be analytic for $\norm{z} < 1/M$ and satisfies
\begin{equation}
\norm*{\tilde{V}^{(n)}(z,\cdot)}_{\mathcal{L}(L^1[0,T],M_n(\mathcal{A}))} \leq \frac{C}{\norm{z}^{-1} - M}.
\end{equation}
If $\norm{z} < 1/(M + \sqrt{2Cu})$ and $0 \leq s \leq t$, then the induction hypothesis implies that $\norm{\tilde{W}_m^{(n)}(z,s)^{-1}} \leq (M + \sqrt{2C(u-s)})^{-1}$, and therefore $\tilde{W}_m^{(n)}(z,s)^{-1}$ is in the domain where $\tilde{V}^{(n)}$ is analytic, and we have  
\begin{equation}
\norm{V^{(n)}(W_m^{(n)}(z^{-1},s),\cdot)}_{\mathcal{L}(L^1,M_n(\mathcal{A}))} = \norm{\tilde{V}^{(n)}(\tilde{W}_m^{(n)}(z,s)^{-1},\cdot)}_{\mathcal{L}(L^1,M_n(\mathcal{A}))} \leq \frac{C}{\sqrt{2 C(u - s)}}.
\end{equation}
Therefore, setting $H_m^{(n)}(z,t) := W_m^{(n)}(z,t) - z$, and using Lemma \ref{lem:diagonal},
\begin{align}
\norm*{\tilde{H}_m^{(n)}(z,t)} &= \norm*{\int_0^t V^{(n)}(W_m^{(n)}(z^{-1},s),s)\,ds} \nonumber \\
&\leq \int_0^t \norm*{V^{(n)}(W_m(z^{-1},s),\cdot)}_{\mathcal{L}(L^1[0,T],M_n(\mathcal{A}))}\,ds \nonumber \\
&\leq \int_0^t \frac{C}{\sqrt{2C(u - s)}}\,ds
= \sqrt{2Cu} - \sqrt{2C(u - t)}.  \label{eq:inductionnearinfinity}
\end{align}
Then
\begin{equation} \label{eq:WMpowerseries}
\tilde{W}_{m+1}^{(n)}(z,t)^{-1} = \left(z^{-1} + \tilde{H}_m^{(n)}(z,t)\right)^{-1} = \sum_{j=0}^\infty (-1)^j z[\tilde{H}_m^{(n)}(z,t)z]^j,
\end{equation}
from which we see that
\begin{align}
\norm*{\tilde{W}_{m+1}^{(n)}(z,t)^{-1}} &\leq \frac{1}{\norm{z}^{-1} - \norm{\tilde{H}_m^{(n)}(z,t)}} \nonumber \\
&\leq \frac{1}{(M + \sqrt{2Cu}) - (\sqrt{2Cu} - \sqrt{2C(u - t)})} \nonumber \\
&= \frac{1}{M + \sqrt{2C(u - t)}}.
\end{align}
This completes the induction proof.

{\bf Step 5:} Fix $t \in [0,T)$ and we will prove (2).  Let $u > t$ as in Step 4.  Because $\tilde{W}_m(z,t)^{-1}$ is analytic and uniformly bounded for $\norm{z} < (M + \sqrt{2Cu})^{-1}$ and because $\tilde{W}_m^{(n)}(z,t)^{-1}$ converges locally uniformly on
\[
\{z: \norm{z} < (M+\sqrt{2Cu})^{-1} \text{ and } \im z^{-1} \geq \epsilon\},
\]
Lemma \ref{lem:analyticcontinuationofconvergence} implies that $\tilde{W}_m^{(n)}(z,t)^{-1}$ converges locally uniformly on $\{z: \norm{z} < (1 + \sqrt{2Cu})^{-1} \}$ as $m \to \infty$.  This implies that $\tilde{W}(z,t)^{-1}$ has a fully matricial extension to $\norm{z} < (M + \sqrt{2Cu})^{-1}$, which is bounded by $(M + \sqrt{2C(u - t)})^{-1}$.  Moreover, using the power series expansion \eqref{eq:WMpowerseries}, we see that $\lim_{z \to 0} z^{-1} \tilde{W}^{(n)}(z,t)^{-1} = 1$, where the limit is taken over invertible $z \in M_n(\mathcal{A})$.  Therefore, Theorem \ref{thm:Cauchytransform} implies that $W(z,t)^{-1}$ is the Cauchy transform of a law with radius bounded by $(M + \sqrt{2Cu})$.  Letting $u \searrow t$ proves (2).

{\bf Step 6:} Fix $t$ and we will prove (3).  By Proposition \ref{prop:Ftransformbijection}, we know that $H(z,t) = G_{\sigma_t}(z)$ for some generalized law $\sigma$.  In order to bound $\rad(\sigma_t)$, note that
\[
H^{(n)}(z,t) = \int_0^t V^{(n)}(W^{(n)}(z,s),s)\,ds.
\]
Using the same reasoning as in \eqref{eq:inductionnearinfinity}, we see that if $u > t$, then $\tilde{H}^{(n)}(z,t)$ is defined for $\norm{z} < 1 / (M + \sqrt{2Cu})$ and bounded by $\sqrt{2Cu} - \sqrt{2C(u-t)}$.  Since for each $u > t$, there is a bounded independent of $n$, we obtain $\rad(\sigma_t) \leq M + \sqrt{2Ct}$.  This proves the first claim of (3), and the proof of \eqref{eq:Herglotzderivativeformula} is a direct computation.
\end{proof}

The following is the $\mathcal{A}$-valued analogue of \cite[Theorem 5.6]{Bauer2005}.

\begin{theorem} \label{thm:Loewnerintegration}
Let $V(z,t)$ be a distributional Herglotz vector field with $\rad(V) \leq M$ and $C = \norm{D\tilde{V}^{(1)}(0,\cdot)[1]}_{\mathcal{L}(L^1[0,T],\mathcal{A})}$.
\begin{enumerate}[(1)]
	\item There exists a Lipschitz Loewner chain $F(z,t)$ satisfying the Loewner equation \eqref{eq:Loewner}.
	\item We have $F_{s,t}(z) = z - \sigma_{s,t}[(z - X)^{-1}]$ where $\sigma_{s,t}: \mathcal{A}\ip{X} \to \mathcal{A}$ is a generalized law with $\rad(\sigma_{s,t}) \leq M + \sqrt{2C(t - s)}$ and $\norm{\sigma_{s,t}(1)} \leq C(t - s)$.
	\item Suppose that $\Psi(z,t)$ is a fully matricial family $\h(\mathcal{A}) \times [0,T] \to M(\mathcal{A})$ such that $\Psi^{(n)}(z,t)$ is a locally Lipschitz family for each $n$.  If $\Psi$ satisfies
	\begin{equation}
	\partial_t \Psi^{(n)}(z,t) = D\Psi^{(n)}(z,t)[V^{(n)}(z,t)],
	\end{equation}
	and $F$ is the Loewner chain from (1), then $\Psi_t = \Psi_0 \circ F_t$.  In particular, if $\Psi_0 = \id$, then $\Psi_t = F_t$, hence the solution in (1) is unique.
\end{enumerate}
\end{theorem}

\begin{proof}
{\bf Step 1:} Fix $s \leq t$.  Note that $V^{(n)}(z,t-u)$, viewed as a formal function of $(z,u)$ on $\h(\mathcal{A}) \times [0,t]$, is a distributional Herglotz vector field.  Thus, by Theorem \ref{thm:Herglotzflow}, there exists $W$ satisfying
\begin{equation}
\partial_u W^{(n)}(z,u) = V^{(n)}(W^{(n)}(z,u), t - u), \qquad W^{(n)}(z,0) = \id.
\end{equation}
Define $F_{s,t}^{(n)}(z) = W^{(n)}(z,t-s)$, so that we have for $s \in [0,t]$,
\begin{equation} \label{eq:FSTdiffeq}
-\partial_s F_{s,t}^{(n)}(z) = V^{(n)}(F_{s,t}^{(n)}(z), s), \qquad F_{t,t}^{(n)} = \id.
\end{equation}
By Theorem \ref{thm:Herglotzflow}, we have
\begin{equation}
F_{s,t}^{(n)}(z) = z - G_{\sigma_{s,t}}^{(n)}(z),
\end{equation}
where $\sigma_{s,t}$ is a generalized law $\mathcal{A}\ip{X} \to \mathcal{A}$ with $\rad(\sigma_{s,t}) \leq M + \sqrt{2C(t - s)}$.

{\bf Step 2:} We claim that if $s \leq t \leq u$, then $F_{s,t} \circ F_{t,u} = F_{s,u}$.  Fix $t \leq u$.  Then for $s \leq t$, we have
\begin{equation}
-\partial_s [F_{s,t}^{(n)}(z)] = V^{(n)}(F_{s,t}^{(n)}(z),s), \qquad F_{t,t}^{(n)}(z) = z,
\end{equation}
so that
\begin{equation}
-\partial_s [F_{s,t}^{(n)} \circ F_{t,u}^{(n)}(z)] = V^{(n)}(F_{s,t}^{(n)} \circ F_{t,u}^{(n)}(z),s), \qquad F_{s,t}^{(n)} \circ F_{t,u}^{(n)}(z)|_{s = t} = F_{t,u}^{(n)}(z).
\end{equation}
Therefore, $F_{s,t}^{(n)} \circ F_{t,u}^{(n)}$ solves the same initial value problem as $F_{s,u}^{(n)}$ with respect to the variable $s$ which runs backwards from $t$ to $0$.  Therefore, $F_{s,t} \circ F_{t,u} = F_{s,u}$ by the uniqueness claim of Theorem \ref{thm:Herglotzflow} (1).

{\bf Step 3:}  Let $F_t = F_{0,t}$.  By Theorem \ref{thm:Herglotzflow} (b), $F_t$ is the reciprocal Cauchy transform of a law $\mu_t$ with $\rad(\mu_t) \leq M + \sqrt{2Ct}$ and $\mu_t(X) = 0$.  And we just showed $F_t = F_s \circ F_{s,t}$ for $s < t$.  Finally, $\norm{\sigma_{s,t}(1)} \leq C|s - t|$ by Theorem \ref{thm:Herglotzflow} (c).  Hence, $F_t$ is a Loewner chain.  Moreover, by Lemma \ref{lem:LoewnerlocallyLipschitz}, $F_t^{(n)}$ is a locally Lipschitz family for each $n$.

{\bf Step 4:}  We now verify that $F_t$ satisfies the Loewner equation.   Assume that $[a,b] \subseteq [0,T]$. Then for $s \in [a,b]$, we have
\begin{equation}
F_b^{(n)}(z) = F_s^{(n)} \circ F_{s,b}^{(n)}(z).
\end{equation}
Upon differentiating with respect to $s$ and invoking the chain rule (Lemma \ref{lem:chainrule}),
\begin{align}
0 &= \partial_s F_s^{(n)}(F_{s,b}^{(n)}(z)) + DF_s^{(n)}(F_{s,b}^{(n)}(z))[\partial_s F_{s,b}^{(n)}(z)] \nonumber \\
&= \partial_s F_s^{(n)}(F_{s,b}^{(n)}(z)) - DF_s^{(n)}(F_{s,b}^{(n)}(z))[V^{(n)}(F_{s,b}^{(n)}(z),s)].
\end{align}
For $z \in \h_\epsilon^{(n)}(\mathcal{A})$ and $s \in [a,b]$, we have
\begin{equation}
F_{s,b}^{(n)}(z) = z + O(|b - a|).
\end{equation}
Because $F_{s,b}^{(n)}$ maps $\h_\epsilon^{(n)}(\mathcal{A})$ into itself, because $\partial_s F_s^{(n)}(z)$ and $V^{(n)}(z,s)$ are bounded for $z \in \h_\epsilon^{(n)}(\mathcal{A})$, and because $DF_s$ Lipschitz in $s$ in this region, we have
\begin{align}
\partial_s F_s^{(n)}(F_{s,b}^{(n)}(z), s) &= \partial_s F_s(z) + O(|b - a|) \\
DF_s^{(n)}(F_{s,b}^{(n)}(z))[V^{(n)}(F_{s,b}^{(n)}(z),s)] &= DF_s^{(n)}(z)[V^{(n)}(z,s)] + O(|b - a|),
\end{align}
where the equation holds in the space $\mathcal{L}(L^1[a,b],M_n(\mathcal{A}))$ with respect to the variable $s$.  Therefore, pairing with $\chi_{[a,b]}$ yields
\begin{equation}
\int_a^b \partial_s F_s^{(n)}(z)\,ds = \int_a^b DF_s^{(n)}(z)[V^{(n)}(z,s)] \,ds + O(|b - a|^2).
\end{equation}
Hence, by Lemma \ref{lem:smallerror}, we have $\partial_s F_s^{(n)}(z) = DF_s^{(n)}(z)[V^{(n)}(z,s)]$ as desired.  Therefore, (1) is proved.  Moreover, (2) follows from Theorem \ref{thm:Herglotzflow}, so it only remains to prove (3).

{\bf Step 5:} Suppose that $\Psi$ satisfies the hypotheses of (3) and let $t > 0$.  Then for $s \in [0,t]$, we apply the chain rule and the fact that $\Psi$ satisfies the Loewner equation to conclude that
\begin{align}
\partial_s [\Psi_s^{(n)} \circ F_{s,t}^{(n)}(z)] &= \partial_s \Psi_s^{(n)}(F_{s,t}^{(n)}(z)) + D\Psi_s^{(n)}(F_{s,t}^{(n)}(z))[\partial_s F_{s,t}^{(n)}(z)] \nonumber \\
&=  D\Psi_s^{(n)}(F_{s,t}^{(n)}(z))[V^{(n)}(F_{s,t}^{(n)}(z),s)] - D\Psi_s^{(n)}(F_{s,t}^{(n)}(z))[V^{(n)}(F_{s,t}^{(n)}(z),s)] \nonumber \\
&= 0.
\end{align}
Thus, after formally integrating with respect to $s$ from $0$ to $t$, we obtain
\begin{equation}
\Psi_t^{(n)} = \Psi_t^{(n)} \circ F_{t,t}^{(n)} = \Psi_0^{(n)} \circ F_{0,t}^{(n)} = \Psi_0^{(n)} \circ F_t^{(n)}
\end{equation}
which proves (3).
\end{proof}

\subsection{Monotone and Free Convolution Semigroups} \label{subsec:semigroups}

The special cases of monotone and free convolution semigroups have received a lot of attention in the literature, which has found analogues of the L{\'e}vy-Hin{\v c}in formula for various types of independence.  We briefly describe how previous results for semigroups relate to the theory developed in this paper.  The results of this section are not intended to be exhaustive.

A \emph{monotone convolution semigroup} is a family of $\mathcal{A}$-valued laws $\mu_t$ such that $\mu_{s+t} = \mu_s \rhd \mu_t$, where the monotone convolution occurs over $\mathcal{A}$.  By Theorem \ref{thm:processLoewnerchain}, a monotone convolution semigroup is equivalent to a composition semigroup $(F_t)$ of $\mathcal{A}$-valued $F$-transforms.  Such semigroups were studied in the scalar case by \cite{Hasebe2010-2} and in the operator-valued case by \cite{HS2014} and \cite{AW2016}.  We now give an alternative proof of the following results from \cite{AW2016}.

\begin{proposition} \label{prop:monotonesemigroup}
Let $\mu_t$ be an $\mathcal{A}$-valued monotone convolution semigroup with mean zero and let $F_t = F_{\mu_t}$ be the corresponding composition semigroup of $F$-transforms.
\begin{enumerate}
	\item There exists a generalized law $\nu$ such that
	\begin{equation} \label{eq:semigroupLoewnerequation}
	\partial_t F^{(n)}(z,t) = DF^{(n)}(z,t)[-G_\nu^{(n)}(z)],
	\end{equation}
	and
	\begin{equation} \label{eq:semigroupLoewnerequation2}
	\partial_t F^{(n)}(z,t) = -G_\nu^{(n)}(F^{(n)}(z,t)),
	\end{equation}
	where the time-derivatives exist pointwise with respect to the operator norm.
	\item Moreover, $\partial_t^k F(z,t)$ exists pointwise with respect to the operator norm for all $k$.
	\item $F(z,t)$ satisfies equality of mixed partials for derivatives of all order in $z$ and $t$.
	\item Conversely, given a generalized law $\nu$, there exists a monotone convolution semigroup $\mu_t$ such that $F_t$ satisfies \eqref{eq:semigroupLoewnerequation}.
\end{enumerate}
\end{proposition}

\begin{proof}
Note that $F_t$ is a normalized Loewner chain satisfying $F_{s,t} = F_{t-s}$.  It follows that $\mu_{s+t}(X^2) = \mu_s(X^2) + \mu_t(X^2)$.  If $\phi$ is a state on $\mathcal{A}$, then $t \mapsto \phi \circ \mu_t(X^2)$ is an additive function $[0,+\infty) \to [0,+\infty)$ which is also an increasing function.  This implies it is linear, so that $\phi \circ \mu_t(X^2) = t \phi \circ \mu_1(X^2)$ for all $t > 0$.  Since this holds for every state $\phi$, we have $\mu_t(X^2) = t \mu_1(X^2)$ and hence $\mu_t(X^2)$ is Lipschitz in $t$.  It follows that $F_t$ is a Lipschitz normalized Loewner chain.

By Theorem \ref{thm:Loewnerdifferentiation}, there exists a Herglotz vector field $V(z,t)$ such that $F$ and $V$ satisfy the Loewner equation.  We claim that
\begin{equation} \label{eq:constantderivative}
\int V^{(n)}(z,t) \phi(t)\,dt = \widehat{V}^{(n)}(z) \int \phi(t)\,dt, \text{ where } \widehat{V}^{(n)}(z) = \frac{1}{T} \int V^{(n)}(z,t)\,dt.
\end{equation}
Recall from the proof of Theorem \ref{thm:Loewnerdifferentiation} that $V(z,t)$ is the limit of approximations $V_m$ defined by \eqref{eq:Vmdef} as
\begin{equation}
V_m^{(n)}(z,t) = \sum_{j=1}^m \chi_{[t_{j-1},t_j)}(t) \frac{m}{T} H_{t_{j-1},t_j}^{(n)}(z).
\end{equation}
But $H_{t_{j-1},t_j}$ is independent of $j$ because $F_{t_{j-1},t_j}$ is independent of $j$.  Thus, $V_m(z,t)$ is given by a $L_{\Boch}^\infty$ function independent of $t$.  It follows that $V_m$ satisfies \eqref{eq:constantderivative}, and thus so does $V$.  Since $V(z,t)$ is a Herglotz vector field, we see that $\widehat{V}(z)$ is minus the Cauchy transform of some generalized law $\nu$.

Altogether, we have shown that \eqref{eq:semigroupLoewnerequation} holds in the distributional sense.  Hence, for $0 \leq s \leq t \leq T$,
\[
F(z,t) - F(z,s) = \int_s^t DF(z,u)[-G_\nu(z)]\,du.
\]
We know that $DF(z,u)$ is a Lipschitz (in particular, continuous) function of $u$ on $\im z \geq \epsilon$.  Therefore, the proof of the fundamental theorem of calculus shows that $\partial_t F(z,t)$ exists pointwise, and the convergence is uniform on $\im z \geq \epsilon$.  The second equation \eqref{eq:semigroupLoewnerequation2} follows from \eqref{eq:FSTdiffeq} because $F_{s,t} = F_{t-s}$.

By applying a priori estimates on the derivatives of an analytic function (Theorem \ref{thm:analytic}), we see that $\delta^k F(z,t;h)$ is also differentiable with respect to $t$, with uniform convergence for $\im z \geq \epsilon$, and equality of mixed partials holds.  To differentiate with respect to $t$ again, note that
\[
\partial_t F^{(n)}(z,t) = DF^{(n)}(z,t)[-G_\nu^{(n)}(z)] = \delta F^{(n)}(z,t;-G_\nu(z)),
\]
so that
\begin{align*}
\partial_t^2 F^{(n)}(z,t) &= \partial_t \delta F^{(n)}(z,t;-G_\nu^{(n)}(z)) \\
&= \delta \partial_t F^{(n)}(z,t;-G_\nu^{(n)}(z)) \\
&= \delta^2 F(z,t;-G_\nu^{(n)}(z)).
\end{align*}
This argument can be repeated inductively to $t$-derivatives of all orders.

Conversely, given a vector field $V(z)$, the existence of a solution follows from Theorem \ref{thm:Loewnerintegration} and Theorem \ref{thm:processLoewnerchain} (or alternatively from Theorem \ref{thm:Fockspace} below).  See also \cite[p.\ 13]{AW2016}.
\end{proof}

In a similar way, a \emph{free convolution semigroup} is a family of laws $\mu_t$ such that $\mu_{s+t} = \mu_s \boxplus \mu_t$, where $\boxplus$ denotes free convolution over $\mathcal{A}$.  In the free case, the following result was essentially proved in earlier work.

\begin{proposition} \label{prop:freesemigroup}
Let $\mu_t$ be an $\mathcal{A}$-valued free convolution semigroup with mean zero and let $F_t = F_{\mu_t}$.
\begin{enumerate}
	\item There exists a generalized law $\nu$ such that
	\begin{equation} \label{eq:freesemigroupLoewnerequation}
	\partial_t F^{(n)}(z,t) = DF^{(n)}(z,t)[-G_\nu^{(n)}(F^{(n}(z,t))],
	\end{equation}
	where the time-derivative exists pointwise with respect to the operator norm.
	\item Moreover, $\partial_t^k F(z,t)$ exists pointwise with respect to the operator norm for all $k$.
	\item $F(z,t)$ satisfies equality of mixed partials for derivatives of all order in $z$ and $t$.
	\item Conversely, given a generalized law $\nu$, there exists a monotone convolution semigroup $\mu_t$ such that $F_t$ satisfies \eqref{eq:freesemigroupLoewnerequation}.
\end{enumerate}
\end{proposition}

We now explain the results that underlie this proposition, and how it relates to the Loewner equation.  Let $\mu_t$ be a mean-zero free convolution semigroup.  By Theorem \ref{thm:freeLoewnerchain}, $F_{\mu_t}$ must be an $\mathcal{A}$-valued Loewner chain, and once again $\mu_t(X^2) = t\mu(X^2)$, so that $\mu_t(X^2)$ is automatically Lipschitz in $t$.  It follows that $F_t = F_{\mu_t}$ is a Lipschitz normalized Loewner chain and hence satisfies the Loewner equation for some Herglotz vector field $V(z,t)$.


The evolution equation for the $F$-transforms of free convolution semigroups was studied in \cite[Theorem 4.3]{Voiculescu1986}, and the relationship with Loewner chains was explained in  \cite[\S 3.5]{Schleissinger2017}.  See \cite[\S 4.5-4.7]{Speicher1998}, \cite[\S 3]{PV2013}, \cite[\S 8.1]{ABFN2013} for the operator-valued case.  In particular, the following facts have been proved.  Let $\Phi_\mu$ be the \emph{Voiculescu transform} of a law $\mu$ defined by
\begin{equation} \label{eq:Voiculescutransform}
z + \Phi_\mu(z) = F_\mu^{-1}(z),
\end{equation}
where the equation holds for $\norm{z^{-1}}$ sufficiently small.  If $\mu_t$ is a free convolution semigroup, then we have for $\norm{z^{-1}}$ sufficiently small that
\begin{equation} \label{eq:freesemigroup1}
\Phi_{\mu_t}(z) = t\Phi(z),
\end{equation}
where $\Phi = \Phi_{\mu_1}$.  In this case, since $F^{(n)}(z + t \Phi(z), t) = z$, we have
\begin{equation} \label{eq:freesemigroup2}
\partial_t F^{(n)}(z + t \Phi(z),t) = DF^{(n)}(z + t \Phi^{(n)}(z),t)[-\Phi^{(n)}(z)]
\end{equation}
In particular, by substituting $F^{(n)}(z,t)$ for $z$, we see that the Herglotz vector field in the Loewner equation satisfies $V^{(n)}(z,t) = -\Phi^{(n)}(F^{(n)}(z,t))$ for $\norm{z^{-1}}$ sufficiently small.

It was shown in \cite[Theorem 5.10]{PV2013} that $\Phi$ extends to be fully matricial on $\h(\mathcal{A})$ and in fact $\Phi = -G_\nu$ for some generalized law $\nu$ \cite[Remark 5.7]{PV2013}.  This can alternatively be deduced from the Loewner equation as follows.  We have for $\norm{z^{-1}}$ sufficiently small that $V^{(n)}(z,t) = \Phi^{(n)}(F^{(n)}(z,t))$.  Since $F^{(n)}(z,t) = z + O(t^2)$, we know that for such values of $z$,
\[
\Phi^{(n)}(z) = \lim_{t \searrow 0} \frac{1}{t} \int_0^t V^{(n)}(z,s)\,ds = \lim_{t \searrow 0} \frac{1}{t} H_t^{(n)}(z),
\]
where $H_t^{(n)}(z) = F_t^{(n)}(z) - z$.  But $H_t = -G_{\sigma_t}$ for some generalized law $\sigma_t$ with $\rad(\sigma_t)$ and $\sigma_t(1) / t$ uniformly bounded.  Hence, by Lemma \ref{lem:analyticcontinuationofconvergence}, $\lim_{t \searrow 0} t^{-1} H_t^{(n)}(z)$ exists for all $z \in \h_\epsilon^+(\mathcal{A})$, and the limit must be a function of the form $-G_\nu(z)$ for a generalized law $\nu$ by Theorem \ref{thm:Cauchytransform}.  Thus, $-G_\nu$ furnishes a fully matricial extension of $\Phi$ to the entire upper half-plane.

Once we know that $\Phi = -G_\nu$ is fully matricial on the entire upper half-plane, analytic continuation implies that \eqref{eq:freesemigroup1} and \eqref{eq:freesemigroup2} hold for all $z$ in the upper half-plane.  Moreover, for all $z$ and $t$, we have
\[
V^{(n)}(z,t) = -G_\nu^{(n)}(F^{(n)}(z,t)).
\]
Hence, \eqref{eq:freesemigroupLoewnerequation} holds, and from this equation one can prove smoothness in $t$ and equality of mixed partials similarly to Proposition \ref{prop:monotonesemigroup}.

The converse direction (4) is proved in \cite[\S 4.7]{Speicher1998} by explicitly constructing operators with the given law $\mu_t$ (here the result is stated in terms of moments rather than analytic functions).  An analytic proof is given in \cite[Theorem 4.1]{Williams2017} using the Earle-Hamilton theorem.

We remark that \cite[\S 8.1]{ABFN2013} proved an analogue of \eqref{eq:freesemigroup2} when the scalar parameter $t$ is replaced by a completely positive map $\eta: \mathcal{A} \to \mathcal{A}$.  Future research should consider such a generalization for monotone convolution semigroups.


\section{Combinatorics and Fock Space Model} \label{sec:combinatorics}

\subsection{Preliminaries} \label{subsec:combinatoricspreliminary}

Given a Herglotz vector field $V$ corresponding to the distributional family of generalized laws $\nu$, let $\mu_{s,t}$ be the law associated to the subordination function $F_{s,t}$ for the Loewner chain.  Our next goal is to describe the moments of $\mu_{s,t}$ combinatorially in terms of the moments of $\nu$.  Theorem \ref{thm:combinatorics} will express the moments of $\mu_{s,t}$ as a sum indexed by non-crossing partitions (certain combinatorial objects often used in non-commutative probability).  These terms will be defined by iterating the operations of multiplication and of the maps $C([0,T],\mathcal{A})^{k+1} \to \mathcal{A}$ given by
\[
(f_0,\dots,f_k) \mapsto \int \nu(f_0(t)Xf_1(t) \dots X f_k(t),t)\,dt.
\]
Therefore, we being by explaining the meaning of expressions of the form
\[
\int_0^T \nu(f(X,t),t)\,dt,
\]
in light of the constructions developed in \S \ref{subsec:distderiv}.

\begin{notation}
Let $\mathcal{C}_{\mathcal{A}} = C([0,T],\mathcal{A})$ and $\mathcal{L} = L_{\Boch}^\infty([0,T],\mathcal{A})$.  Note that $\mathcal{C}$ is a closed subalgebra of $\mathcal{L}_{\mathcal{A}}$.  Then $\mathcal{L}_{\mathcal{A}}\ip{X}$ is the linear span of terms of the form
\[
a_0(t) X a_1(t) \dots X a_k(t),
\]
where $a_j \in \mathcal{L}_{\mathcal{A}}$.  We will denote elements of $\mathcal{L}_{\mathcal{A}}\ip{X}$ as functions of $X$ and $t$, such as $p(X,t)$.
\end{notation}

We will define $\int_0^T \nu(f(X,t),t)\,dt$ when $f \in \mathcal{L}_{\mathcal{A}}\ip{X}$.  Suppose that $f(X,t) = f_0(t) X f_1(t) \dots X f_k(t)$ where each $f_j(t)$ is a simple function in $\mathcal{L}_{\mathcal{A}}$.  Then for each $t \in [0,T]$, we have
\[
\nu(f_0(t) X f_1(t) \dots X f_k(t), \cdot) \in \mathcal{L}(L^1[0,T],\mathcal{A}),
\]
and moreover
\[
t \mapsto \nu(f_0(t) X f_1(t) \dots X f_k(t), \cdot) \in \mathcal{L}(L^1[0,T],\mathcal{A})
\]
is a simple function in $L_{\Boch}^\infty([0,T], \mathcal{L}(L^1[0,T],\mathcal{A}))$.  Morever, if $g_0$, \dots, $g_k$ is another tuple of simple functions in $\mathcal{L}_{\mathcal{A}}$, and if $M = \rad(\nu)$, then
\begin{align*}
&\norm{\nu(f_0(t) X f_1(t) \dots X f_k(t), \cdot) - \nu(g_0(t) X g_1(t) \dots X g_k(t), \cdot)}_{\mathcal{L}(L^1[0,T],\mathcal{A})} \\
&\leq \sum_{j=1}^k \norm{\nu(f_0(t) X \dots f_{j-1}(t)X(f_j(t) - g_j(t))X g_{j+1}(t) \dots X g_k(t), \cdot)}_{\mathcal{L}(L^1[0,T],\mathcal{A})} \\
&\leq \sum_{j=1}^k M^k \norm{f_0} \dots \norm{f_{j-1}} \norm{f_j - g_j} \norm{g_{j+1}} \dots \norm{\widehat{a}_k}.
\end{align*}
This implies that the map
\[
(f_0,\dots,f_k) \mapsto (t \mapsto \nu(f_0(t) X f_1(t) \dots X f_k(t), \cdot))
\]
extends to a bounded multilinear map $\mathcal{L}_{\mathcal{A}}^{k+1} \to L_{\Boch}^\infty([0,T],\mathcal{L}(L^1[0,T],\mathcal{A}))$.  So by Lemma \ref{lem:diagonal}, the formal function $t \mapsto \nu(f_0(t) X f_1(t) \dots X f_k(t),t)$ is a well-defined element of $\mathcal{L}(L^1[0,T],\mathcal{A})$.  By linearity, the same holds when $f_0(t) X f_1(t) \dots X f_k(t)$ is replaced by an arbitrary element $f \in \mathcal{L}_{\mathcal{A}}\ip{X}$.

In particular, $\int_0^T \nu(f(X,t),t)\,dt$ is defined.  Moreover, for $0 \leq s \leq t \leq T$, the integral
\[
\int_s^t \nu(f(X,u),u)\,du
\]
is defined, and as in Observation \ref{obs:weakderivative} it is a Lipschitz function in the variable $s$ or in the variable $t$.  So for instance, $\int_t^T \nu(f(X,u),u)\,du$ can be viewed as an element of $C([0,T],\mathcal{A}) = \mathcal{C}_{\mathcal{A}}$.

We also have that for $F_j \in M_n(\mathcal{L}_\mathcal{A})$,
\begin{equation} \label{eq:improvedradiusbound2}
\int_0^T \nu^{(n)}(F_0(t) X F_1(t) \dots X F_k(t))\,dt \leq \norm{\nu(1,\cdot)}_{\mathcal{L}(L^1[0,T],\mathcal{A})} \rad(\nu)^k \int_0^T \norm{F_0(t)} \dots \norm{F_k(t)}\,dt
\end{equation}
using a simple function approximation argument and \eqref{eq:improvedradiusbound}.

\subsection{Combinatorial Formula} \label{subsec:combinatorics}

In this subsection we will prove the combinatorial moment formula Theorem \ref{thm:combinatorics} and in the next we will estimate the terms in this formula (Theorem \ref{thm:combinatoricsestimate}).  As explained in Remark \ref{rem:semigroupcumulants} below, these theorems generalize results of \cite{Muraki2000}, \cite{HS2011-2}, \cite{HS2014},  \cite{AW2016} on monotone convolution semigroups and the monotone cumulants.  The important special case of the operator-valued arcsine law was studied in \cite[Theorem 2.5]{BPV2013} (see \S \ref{subsec:arcsine}, Corollary \ref{cor:arcsinemoments} below).  We use the following terminology for non-crossing partitions.

\begin{definition}[Partitions]
Let $[k] = \{1,\dots,k\}$.  A \emph{partition of $[k]$} is a collection of disjoint nonempty subsets of $[k]$ (called \emph{blocks}) whose union is $[k]$.  The cardinality $|\pi|$ is the number of blocks of $\pi$.  We write $i \sim_\pi j$ to mean that two indices $i$ and $j$ are in the same block of $\pi$.
\end{definition}

\begin{definition}[Non-crossing partitions]
A \emph{crossing} in a partition $\pi$ is a set of indices $i_1 < j_1 < i_2 < j_2$ such that $i_1 \sim_\pi i_2 \not \sim_\pi j_1 \sim_\pi j_2$.  A partition is called \emph{non-crossing} if it has no crossings.  We denote by $NC(k)$ the set of non-crossing partitions of $[n]$ and define $NC = \bigsqcup_{k=1}^\infty NC(k)$.  Similarly, we denote by the set of partitions of $[k]$ with no singleton blocks by $NC_{\geq 2}(k)$ and define $NC_{\geq 2} = \bigsqcup_{k=1}^\infty NC_{\geq 2}(k)$.  We also set $NC(0) = NC_{\geq 2}(0) = \{\varnothing\}$, where $\varnothing$ is the partition of the set $\varnothing$ into zero blocks.
\end{definition}

\begin{definition}
If $B$ and $B'$ are blocks of $\pi$, we say that $B'$ \emph{surrounds} $B$, or $B' \prec B$, if there exist $i$, $j \in B'$ such that $B \subseteq \{i+1,\dots,j-1\}$.  Note that this is a strict partial order on the blocks of $\pi$.  Also, for a non-crossing partition $\pi$, if one element of $B$ is surrounded by $B'$, then $B' \prec B$.
\end{definition}

\begin{definition}[Concatenation]
If $\pi_1 \in NC(m)$ and $\pi_2 \in NC(n)$, we define the \emph{concatenation} $\pi_1 \pi_2 \in NC(m+n)$ by $\{B: B \in \pi_1\} \cup \{B'+m: B' \in \pi_2\}$, where $B' + m$ denotes the right translate by $m$ of the set $B'$; note that the concatention operation is associative.
\end{definition}

\begin{definition}[Nesting] \label{def:nesting}
Given $\pi_1$, \dots, $\pi_m \in NC$, we define $\Theta_m(\pi_1, \dots, \pi_m)$ as the partition obtained by taking a block $B$ of size $m+1$ and inserting $\pi_j$ between the $j$th and $(j+1)$st elements of $B$.  Explicitly, if $\pi_j \in NC(k_j - 1)$ and $K_j = 1 + k_1 + \dots + k_j$, then
\[
\Theta_m(\pi_1, \dots \pi_m) = \{\{K_0, \dots, K_m\}\} \cup \bigcup_{j=1}^m \{B' + K_{j-1}: B' \in \pi_j\} \in NC(N_m).
\]
\end{definition}

\begin{lemma} \label{lem:defineQ}
For $\pi \in NC_{\geq 2}(k)$, $k \geq 2$, and $0 \leq s \leq t \leq T$, there are unique multilinear maps $Q_{\pi;s,t}: \mathcal{A}^{k-1} \to \mathcal{A}\ip{X}$ satisfying the following.
\begin{enumerate}
	\item If $\pi = \pi_1 \pi_2$ with $\pi_1 \in NC_{\geq 2}(k_1)$ and $\pi_2 \in NC_{\geq 2}(k_2)$ with $k_1, k_2 \geq 1$, then
	\[
	Q_{\pi;s,t}(a_1,\dots,a_{k_1+k_2-1}) = Q_{\pi_1;s,t}(a_1,\dots,a_{k_1-1}) a_{k_1} Q_{\pi_2;s,t}(a_{k_1+1},\dots,a_{k_1+k_2-1}).
	\]
	\item If $\pi = \Theta_m(\pi_1,\dots,\pi_m)$ with $\pi_j \in NC_{\geq 2}(k_j-1)$ with $k_j \geq 1$ and if $K_j = 1 + k_1 + \dots + k_j$, then
	\begin{multline*}
	Q_{\pi;s,t}(a_1,\dots,a_{K_m-1}) \\
	= \int_s^t \nu\Bigl(a_1 Q_{\pi_1;s,t}(a_2,\dots,a_{K_1-2}) a_{K_1-1} X  \dots \\ \dots X a_{K_{m-1}} Q_{\pi_m;s,t}(a_{K_{m-1}+1}, \dots, a_{K_m-2}) a_{K_m-1},\, u\Bigr)\,du,
	\end{multline*}
\end{enumerate}
where in (2), the convention for the case $k_j - 1 = 0$ is that
\[
a_{K_{j-1}} Q_{\pi_j}(a_{K_{j-1}+1}, \dots, a_{K_j-2}) a_{K_j-1} = a_{K_{j-1}} = a_{K_j-1}.
\]
\end{lemma}

\begin{remark}
For example, if $\pi$ is the partition $\{\{1,4,5\},\{2,3\}, \{6,7\} \}$, then
\[
Q_{\pi;s,t}(a_1,\dots,a_6) = \left[ \int_s^t \nu \left(a_1 \left[ \int_u^t \nu(a_2,v)\,dv \right] a_3 X a_4, u \right)\,du \right] a_5 \left[ \int_s^t \nu(a_6,u)\,du \right].
\]
\end{remark}

\begin{proof}[Proof of Lemma \ref{lem:defineQ}]
We show that $Q_\pi$ is well-defined by induction.  Let $\pi$ be a partition in $NC_{\geq 2}(k)$.  Let $B_1$, \dots, $B_n$ be the ``outermost'' blocks of $\pi$ (those blocks which are minimal with respect to $\prec$), listed in order of $\min B_j$.  An element of $B_i$ cannot come in between two elements of $B_j$ for $i \neq j$, but on the other hand, every block $B$ must be surrounded by some $B_j$.  This implies that $[k]$ is the disjoint union of the intervals $\{\min B_j, \dots, \max B_j\}$, and $\pi$ restricts to define a partition $\pi_j$ on each subinterval.  Thus, $\pi = \pi_1 \dots \pi_n$, where each $\pi_j$ is irreducible with respect to concatenation.  If $m > 1$, then we define $Q_{\pi;s,t}$ by multiplying $Q_{\pi_1;s,t}$, \dots, $Q_{\pi_n;s,t}$ as in (1).

On the other hand, if $m = 1$, there is only one outermost block $B = B_1$.  Let us write $B = \{K_1,\dots,K_m\}$.  then because $\pi$ is non-crossing, every block $B' \neq B$ must be contained in $\{K_{j-1}+1,\dots, K_j-1\}$ for some $j$.  If $\pi_j$ is the restriction of $\pi$ to $\{K_{j-1}+1,\dots,K_j-1\}$, then $\pi = \Theta_m(\pi_1, \dots, \pi_m)$ (note that $\pi_j$ will be the empty partition in the case that $K_{j-1}+1 = K_j$).  We then define $Q_{\pi;s,t}$ by (2).  This shows that $Q_{\pi;s,t}$ is uniquely determined by conditions (1) and (2).  Multilinearity of $Q_{\pi;s,t}$ is verified by induction.
\end{proof}

\begin{theorem} \label{thm:combinatorics}
As in \S \ref{sec:Loewnerequation}, let $\nu$ be a distributional family of generalized laws on $[0,T]$, let $V$ be the corresponding distributional Herglotz vector field, let $F_{\mu_t}$ be the corresponding solution to the Loewner equation, and for $0 \leq s \leq t \leq T$, let $F_{\mu_{s,t}}$ be the subordination function satisfying $F_{\mu_s} \circ F_{\mu_{s,t}} = F_{\mu_t}$, and let $\sigma_{s,t}$ be the generalized law with $F_{\mu_{s,t}}(z) = z - G_{\sigma_{s,t}}(z)$.

Let $Q_{\pi;s,t}$ be defined by Lemma \ref{lem:defineQ}.  Then we have
\begin{equation} \label{eq:muSTcombinatorialformula}
\mu_{s,t}(a_0 X a_1 \dots X a_k) = \sum_{\pi \in NC_{\geq 2}(k)} a_0 Q_{\pi;s,t}(a_1,\dots,a_{k-1}) a_k,
\end{equation}
where right hand side is understood to be $a_0$ in the case where $k = 0$.  We also have
\begin{equation} \label{eq:sigmaSTcombinatorialformula}
\sigma_{s,t}(a_0 X a_1 \dots X a_k) = \sum_{\substack{\pi \in NC_{\geq 2}(k+2) \\ 1 \sim_\pi k+2}} Q_{\pi;s,t}(a_0,\dots,a_k).
\end{equation}
\end{theorem}

\begin{notation}
If $\Lambda: \mathcal{A}^k \to \mathcal{A}$ is a multilinear map, then we define $\Lambda^{(n)}: M_n(\mathcal{A})^k \to M_n(\mathcal{A})$ by
\[
[\Lambda^{(n)}(A_1,\dots,A_k)]_{i,j} = \sum_{i_1, \dots, i_{k-1}} \Lambda((A_1)_{i,i_1}, (A_2)_{i_1,i_2}, \dots, (A_{k-1})_{i_{k-2},i_{k-1}}, (A_k)_{i_{k-1},j}).
\]
\end{notation}

\begin{proof}[Proof of Theorem \ref{thm:combinatorics}]
As we remarked after Theorem \ref{thm:Cauchytransform}, the moments $\sigma_{s,t}(a_0 X a_1 \dots X a_k)$ can be evaluated from $\tilde{G}_{\sigma_{s,t}}^{(k+2)}(z X z \dots X z)$ where $z$ is a certain upper triangular matrix; see \eqref{eq:Cauchytransformmoment}.  Thus, it suffices to show that for each $n$ and $k$,
\begin{equation} \label{eq:FSTtruncatedformula}
\tilde{G}_{\sigma_{s,t}}(z) = \sum_{j=2}^{k+1} \sum_{\substack{\pi \in NC_{\geq 2}(j) \\ 1 \sim_\pi j}} Q_{\pi;s,t}^{(n)}(z,\dots,z) + O(\norm{z}^{k+1}),
\end{equation}
and
\begin{equation} \label{eq:GSTtruncatedformula}
\tilde{G}_{\mu_{s,t}}^{(n)}(z) = \sum_{j=0}^k \sum_{\pi \in NC_{\geq 2}(j)} z Q_{\pi;s,t}^{(n)}(z,\dots,z)z + O(\norm{z}^{k+2}).
\end{equation}
In light of \eqref{eq:Cauchytransformmoment}, we do not need the error estimates here to be independent of $n$, although it turns out that they will be.  In the following, we abbreviate $Q_{\pi;s,t}^{(n)}(z,\dots,z)$ to $Q_{\pi;s,t}^{(n)}(z)$.  We continue to use the notation $\tilde{H}(z) = H(z^{-1})$ for various functions $H$.  We also remark, preliminary to the proof, that $Q_{\pi;s,t}(z) = O(\norm{z}^{k-1})$ by a straightforward induction argument (and an explicit estimate will be proved in Theorem \ref{thm:combinatoricsestimate} below).

We verify by induction on $k$ that \eqref{eq:GSTtruncatedformula} and \eqref{eq:FSTtruncatedformula} hold with error estimates that are uniform for $s, t \in [0,T]$ with $s \leq t$.  In the base case $k = 0$, the index set for the sum in \eqref{eq:FSTtruncatedformula} is empty, so the equation reduces to $\tilde{G}_{\sigma_{s,t}}(z) = O(\norm{z})$.  Meanwhile, \eqref{eq:GSTtruncatedformula} reduces to $\tilde{G}_{\mu_{s,t}}(z) = z + O(\norm{z}^2)$ which holds because $\mu_{s,t}|_{\mathcal{A}} = \id$.

For the induction step, suppose $k > 0$.  Let us first show \eqref{eq:FSTtruncatedformula}.  From the construction of solutions to the Loewner equation in \eqref{eq:FSTdiffeq}, we have that
\[
F_{\mu_{s,t}}^{(n)}(z) = z + \int_s^t V^{(n)}(F_{u,t}^{(n)}(z), u)\,du = z + \int_s^t \tilde{V}^{(n)}(\tilde{G}_{\mu_{u,t}}^{(n)}(z),u)\,du.
\]
which implies that for $\norm{z}$ small enough, we have
\begin{align*}
\tilde{G}_{\sigma_{s,t}}(z) &= -\int_s^t \tilde{V}^{(n)}(\tilde{G}_{\mu_{u,t}}^{(n)}(z),u)\,du. \\
&= \int_s^t \nu^{(n)} \left( (\tilde{G}_{\mu_{u,t}}(z)^{-1} - X)^{-1},u \right)\,du \\
&= \sum_{m=1}^k \int_s^t \nu^{(n)}\left( G_{\mu_{u,t}}^{(n)}(z)X)^{m-1} G_{\mu_{u,t}}^{(n)}(z), \, u \right)\,du  + O(\norm{z}^{k+1}).
\end{align*}
By the induction hypothesis,
\[
\tilde{G}_{\mu_{u,t}}^{(n)}(z) = \sum_{j=0}^{k-1} \sum_{\pi \in NC_{\geq 2}(j)} z Q_{\pi;u,t}^{(n)}(z)z + O(\norm{z}^{k+1}).
\]
Therefore,
\begin{align*}
&\tilde{G}_{\sigma_{s,t}}^{(n)}(z) \\
=& \sum_{m=1}^k \sum_{j_1,\dots,j_m=0}^{k-1} \sum_{\pi_i \in NC_{\geq 2}(j_i)}
\int_s^t \nu^{(n)} \left( zQ_{\pi_1;u,t}^{(n)}(z)z X^{(n)} \dots X^{(n)} zQ_{\pi_m;u,t}^{(n)}(z), u \right)\,du + O(\norm{z}^{k+1}) \\
=&  \sum_{m=1}^k \sum_{j_1,\dots,j_m=0}^{k-1} \sum_{\pi_i \in NC_{\geq 2}(j_i)} Q_{\Theta_m(\pi_1,\dots,\pi_m);s,t}^{(n)}(z) + O(\norm{z}^{k+1}).
\end{align*}
For $j \leq k + 1$, every $\pi \in NC_{\geq 2}(j)$ where the first and last elements are in the same block can be uniquely written as $\Theta_m(\pi_1,\dots,\pi_m)$, where $\pi_i \in NC_{\geq 2}(j_i)$ and $j_i \leq (k + 1) - 2 = k - 1$ by the same reasoning as in Lemma \ref{lem:defineQ}.  Therefore, relabelling the terms and discarding the terms of size $O(\norm{z}^{k+1})$, we obtain \eqref{eq:FSTtruncatedformula}.

To check \eqref{eq:GSTtruncatedformula}, observe that
\begin{align*}
\tilde{G}_{\mu_{s,t}}^{(n)}(z) &= (z^{-1} - \tilde{G}_{\sigma_{s,t}}^{(n)}(z))^{-1} \\
&= z + \sum_{m=1}^k (z \tilde{G}_{\sigma_{s,t}}^{(n)}(z))^m z + O(\norm{z}^{k+2}).
\end{align*}
When we substitute \eqref{eq:FSTtruncatedformula} for $\tilde{G}_{\sigma_{s,t}}$, we obtain
\begin{align*}
\tilde{G}_{\mu_{s,t}}^{(n)}(z) &= z + \sum_{m=1}^k \sum_{j_1,\dots,j_m=0}^k \sum_{\pi_i \in NC_{\geq 2}(j_i)} z Q_{\pi_1;s,t}^{(n)}(z) z \dots z Q_{\pi_m;s,t}^{(n)}(z) z + O(\norm{z}^{k+2}) \\
&= z + \sum_{m=1}^k \sum_{j_1,\dots,j_m=0}^k \sum_{\pi_i \in NC_{\geq 2}(j_i)} z Q_{\pi_1 \dots \pi_m;s,t}^{(n)}(z) + O(\norm{z}^{k+2}) 
\end{align*}
But every $\pi \in NC_{\geq 2}(j)$ for $j = 1, \dots, k$ can be written as $\pi_1 \dots \pi_m$ for some $\pi_i \in NC_{\geq 2}(j_i)$ with $j_i \leq k$.  Thus, after regrouping the terms and ignoring terms of size $O(\norm{z}^{k+2})$, we obtain \eqref{eq:GSTtruncatedformula}.
\end{proof}

\subsection{Combinatorial Estimates} \label{subsec:combinatorics2}

In the next theorem, we prove a relatively sharp estimate for $\norm{Q_{\pi;s,t}^{(n)}(A_1,\dots,A_{k-1})}$.  This will allow us to turn \eqref{eq:FSTtruncatedformula} and \eqref{eq:GSTtruncatedformula} into convergent series expansions for $\tilde{G}_{\mu_{s,t}}^{(n)}$ and $\tilde{G}_{\sigma_{s,t}}^{(n)}$, where the convergence is independent of whether we group the terms $Q_{\pi;s,t}^{(n)}$ by their degree in $z$.  The idea behind the proof of the next theorem is to compare the coefficients $Q_{\pi;s,t}^{(n)}$ of our $\mathcal{A}$-valued Loewner chain to those of a scalar-valued Loewner chain for which the sums can be computed exactly.  The coefficients for this scalar-valued Loewner chain are computed in terms of universal constants $\alpha_\pi$ defined as follows.

\begin{definition}
Let $\pi \in NC(k)$ for $k \geq 1$.  We say that a function $\tau: \pi \to \R$ (that is, a function defined on the set of blocks of $\pi$) is \emph{compatible with $\pi$} if $B \prec B'$ implies $\tau(B) < \tau(B')$.  Let $[0,1]^\pi$ be the set of functions $\pi \to [0,1]$, and define
\[
\alpha_\pi = |\{\tau \in [0,1]^\pi: B \prec B' \implies \tau_B < \tau_{B'}\}|,
\]
where $|\cdot|$ denotes Lebesgue measure, viewing $[0,1]^\pi$ as a subset of Euclidean space.  We also set $\alpha_{\varnothing} = 1$.
\end{definition}

\begin{theorem} \label{thm:combinatoricsestimate}
Continuing with setup of Theorem \ref{thm:combinatorics}, let $M = \rad(\nu)$ and $C = \norm{\nu(1,\cdot)}_{\mathcal{L}(L^1[0,T], \mathcal{A})}$.  For $0 \leq s \leq t \leq T$ and $\pi \in NC_{\geq 2}(k)$ and $A_1$, \dots, $A_{k-1} \in M_n(\mathcal{A})$, we have
\begin{equation} \label{eq:Qpiestimate}
\norm*{ Q_{\pi;s,t}^{(n)}(A_1,\dots,A_{k-1}) } \leq \alpha_\pi (C(t - s))^{|\pi|} M^{k-2|\pi|}\norm{A_1} \dots \norm{A_{k-1}}.
\end{equation}
Moreover, for $\zeta \in \C$ with $|\zeta| < 1 / (M + \sqrt{2C(t - s)})$, we have
\begin{align}
\sum_{k \geq 0} \sum_{\pi \in NC_{\geq 2}(k)} \alpha_\pi (C(t-s))^{|\pi|} M^{k-2|\pi|} \zeta^k &= \left( \sqrt{(\zeta - M)^2 - 2C(t - s)} + M \right)^{-1} \label{eq:combinatorialscalarformula1} \\
\sum_{k \geq 2} \sum_{\substack{\pi \in NC_{\geq 2}(k) \\ 1 \sim_\pi k}} \alpha_\pi (C(t - s))^{|\pi|} M^{k-2|\pi|}\zeta^k &= \zeta - M - \sqrt{(\zeta - M)^2 - 2C(t - s)} \label{eq:combinatorialscalarformula2}
\end{align}
In particular, the series expansions
\begin{align}
\tilde{G}_{\mu_{s,t}}^{(n)}(z) &= \sum_{k=0}^\infty \sum_{\pi \in NC_{\geq 2}(k)} z Q_\pi^{(n)}(z,\dots,z)z \label{eq:GSTcombinatorialformula} \\
\tilde{G}_{\sigma_{s,t}}^{(n)} &= \sum_{k=2}^\infty \sum_{\substack{\pi \in NC_{\geq 2}(k) \\ 1 \sim_\pi k}} Q_{\pi;s,t}^{(n)}(z,\dots,z),  \label{eq:FSTcombinatorialformula},
\end{align}
are absolutely convergent (regardless of the grouping of the terms) for $\norm{z} < 1 / (M + \sqrt{2C(t - s)})$.
\end{theorem}

\begin{proof}
To simplify notation, let us only write the proof in the case where $A_1 = \dots = A_{k-1} = z$.  In fact, the general case can be deduced from this one by rescaling the $A_j$'s to have norm $1$ and then letting $z$ be an upper triangular matrix of size $nk$ as in \eqref{eq:uppertriangular}.  We must now estimate $Q_{\pi;s,t}^{(n)}(z,\dots,z)$ for $z \in M_n(\mathcal{A})$, and as in the proof of Theorem \ref{thm:combinatorics}, we abbreviate this to $Q_{\pi;s,t}^{(n)}(z)$.

We proceed by induction on $k$ with the base case $k = 0$ (the empty partition) being trivial.  If $\pi \in NC_{\geq 2}(k)$ with $k > 0$, there are two subcases.  First, if $\pi = \pi_1 \pi_2$, then
\begin{align*}
\norm{Q_{\pi;s,t}^{(n)}(z)} &\leq \norm{Q_{\pi_1;s,t}^{(n)}(z)} \norm{z} \norm{Q_{\pi_2;s,t}^{(n)}(z)} \\
&\leq \Bigl(  \alpha_{\pi_1}  (C(t-s))^{|\pi_1|} M^{k_1-2|\pi_1|} \norm{z}^{k_1-1} \Bigr) \norm{z} \Bigl( \alpha_{\pi_2} (C(t-s))^{|\pi_2|} M^{k_2-2|\pi_2|} \norm{z}^{k_2-1} \Bigr) \\
&= \alpha_{\pi_1} \alpha_{\pi_2}(C(t-s))^{|\pi|} M^{k_1+k_2-2|\pi|} \norm{z}^{k_1+k_2-1}.
\end{align*}
To finish the induction step in this case, observe that $\alpha_{\pi_1 \pi_2} = \alpha_{\pi_1} \alpha_{\pi_2}$ because a function $\tau: \pi_1 \pi_2 \to [0,1]$ compatible with $\pi_1 \pi_2$ is equivalent to a pair of functions $\tau_j: \pi_j \to [0,1]$ compatible with $\pi_j$ for $j = 1,2$.

Second, if $\pi = \Theta_m(\pi_1,\dots,\pi_m)$ with $\pi_j \in NC_{\geq 2}(k_j)$, then using \eqref{eq:improvedradiusbound2}
\begin{align}
\norm{Q_{\pi;s,t}^{(n)}(z)} &= \norm*{ \int_s^t \nu^{(n)}(z Q_{\pi_1;u,t}^{(n)}(z)zX \dots Xz Q_{\pi_m;u,t}^{(n)}(z)z,u)\,du } \nonumber \\
&\leq \norm{\nu(1,\cdot)}_{\mathcal{L}(L^1[0,T],\mathcal{A})} \rad(\nu)^{m-1} \int_s^t \norm{z}^{2m} \prod_{j=1}^m \norm{Q_{\pi_j;u,t}^{(n)}(z)} \,du \nonumber \\
&\leq CM^{m-1} \int_s^t \norm{z}^{2m} \prod_{j=1}^m \left( \alpha_{\pi_j} (C(t-u))^{|\pi_j|} M^{k_j - 2|\pi_j|} \norm{z}^{k_j-1} \right)\,du \nonumber \\
&= \left( \int_s^t \prod_{j=1}^m \alpha_{\pi_j}(t - u)^{|\pi_j|}\,du \right) C^{|\pi|} M^{k - 2|\pi|} \norm{z}^{k-1}, \label{eq:nestinginduction}
\end{align}
where the last equality follows from the fact that $\sum_{j=1}^m (k_j - 1) + 2m = \sum_{j=1}^m k_j + m = k - 1$ and $|\pi| = 1 + \sum_{j=1}^m |\pi_j|$.  To complete the argument in this case, we will show that $\int_s^t \prod_{j=1}^m \alpha_{\pi_j}(t - u)^{|\pi_j|}\,du$ is equal to $\alpha_\pi|t - s|^{|\pi|}$.  By translation and rescaling of Lebesgue measure
\[
\alpha_\pi(t - s)^{|\pi|} = |\{\tau \in [s,t]^\pi \text{ compatible with } \pi\}|.
\]
Let $B$ be the block of size $m + 1$ which surrounds the $\pi_j$'s.  We can write the vector $\tau = (u, \tau')$, where $u = \tau(B)$ and $\tau'$ is the vector of the remaining coordinates in $[s,t]^{\pi \setminus B}$.  By the Fubini-Tonelli theorem,
\[
\alpha_\pi(t - s)^{|\pi|} = \int_s^t |\{\tau' \in [s,t]^{\pi \setminus B}: (u,\tau') \text{ compatible with } \pi\}|\,du.
\]
But $(u,\tau')$ being compatible with $\pi$ is equivalent to $\tau'$ taking values in $(u,t]$ and $\tau'|_{\pi_j}$ being compatible with $\pi_j$ for each $j$.  Thus,
\begin{align*}
\alpha_\pi(t - s)^{|\pi|} &= \int_s^t \prod_{j=1}^m |\{\tau_j \in (u,t]^{\pi_j} \text{ compatible with } \pi_j\}|\,du \\
&= \int_s^t \prod_{j=1}^m \alpha_{\pi_j}(t - u)^{|\pi_j|}\,du.
\end{align*}
Substituting this into \eqref{eq:nestinginduction} completes the inductive proof of our estimate for $Q_{\pi;s,t}^{(n)}(z)$ and hence shows  \eqref{eq:Qpiestimate}.

To show \eqref{eq:combinatorialscalarformula2} and \eqref{eq:FSTcombinatorialformula}, consider the scalar-valued Loewner chain $F_{\widehat{\mu}_t}$ with $\widehat{\nu}(\cdot,t) = C \delta_M$, where $\delta_M$ is the Dirac measure at $M \in \R$.  The corresponding coefficients are
\[
\widehat{Q}_{\pi;s,t}(\zeta) = \alpha_\pi (C(t-s))^{|\pi|} M^{k-2|\pi|}
\]
by an inductive argument similar to the foregoing, except with equalities rather than inequalities.  Let $\widehat{\mu}_{s,t}$ and $\widehat{\sigma}_{s,t}$ be the scalar-valued laws (that is, compactly supported finite measures on $\R$) which result from solving the Loewner equation in this case.  A direct computation of the solution yields
\[
F_{\widehat{\mu}_{s,t}}(\zeta) = \sqrt{(\zeta - M)^2 - 2C(t - s)} + M,
\]
and $F_{\widehat{\mu}_{s,t}}(\zeta) - \zeta$ is analytic for $|\zeta| < 1 / (M + \sqrt{2C(t - s)})$.  For such values of $\zeta$, we have as a consequence of \eqref{eq:FSTtruncatedformula} that
\[
\sum_{k \geq 2} \left( \sum_{\substack{\pi \in NC_{\geq 2}(k) \\ 1 \sim_\pi k}} \alpha_\pi (C(t - s))^{|\pi|} M^{k-2|\pi|}\zeta^k \right) = \tilde{G}_{\widehat{\sigma}_{s,t}}(\zeta) = \zeta - M - \sqrt{(\zeta - M)^2 - 2C(t - s)}.
\]
If $\zeta > 0$, then all the terms on the left hand side are nonnegative, and hence we have absolute convergence independent of the grouping of the terms.  This shows \eqref{eq:combinatorialscalarformula2}.  But each term in this sum is an upper bound for $\norm{Q_{\pi;s,t}^{(n)}(z)}$ and therefore we have absolute convergence of the series in \eqref{eq:FSTcombinatorialformula}, and the sum evaluates to $\tilde{G}_{\sigma_{s,t}}^{(n)}(z)$ by \eqref{eq:FSTtruncatedformula}.  Thus, we have shown \eqref{eq:FSTcombinatorialformula}.  The argument for \eqref{eq:combinatorialscalarformula1} and \eqref{eq:GSTcombinatorialformula} is similar.
\end{proof}

\begin{remark}
Another way to express coefficients $\alpha_\pi$ is as follows.  For a partial order $\prec'$ on $\pi$, define $P_{\prec'} = \{\tau \in [0,1]^\pi: B \prec' B' \implies \tau(B) < \tau(B')\}$.  Up to sets of measure zero, $[0,1]^\pi$ can be expressed as the disjoint union of the sets $P_{\prec'} = \{\tau \in [0,1]^\pi: B \prec' B \implies \tau(B) < \tau(B')\}$ as $\prec'$ ranges over all possible \emph{total} orders of $\pi$, and the number of such total orders is $|\pi|!$.  Moreover, $P_{\prec'}$ is contained in $P_\prec$ if and only if $\prec'$ extends the partial order $\prec$.  Therefore,
\[
\alpha_\pi = |P_{\prec}| = \frac{1}{|\pi|!} \#\{\text{total orders on } \pi \text{ extending } \prec\}.
\]
Compare \cite[Definition 4.1 - Definition 4.4]{AW2016} and \cite[Theorem 5.3]{HS2011-2}.
\end{remark}

\begin{remark} \label{rem:semigroupcumulants}
In the case of monotone convolution semigroups, the results of this section and the last section boil down to results of \cite{AW2016}.  Suppose that $\nu: \mathcal{A}\ip{X} \to \mathcal{A}$ is a generalized law independent of $t$.  Let us define $\Lambda_\pi: \mathcal{A}^{k-1} \to \mathcal{A}$ inductively by
\[
\Lambda_{\pi_1\pi_2}(a_1,\dots,a_{k_1+k_2}-1) = \Lambda_{\pi_1}(a_1,\dots,a_{k_1-1}) a_{k_1} \Lambda_{\pi_2}(a_{k_1+1}, \dots, a_{k_1+k_2-1})
\]
and for $\pi_1 \in NC(k_1-1)$, \dots, $\pi_m \in NC(k_m-1)$,
\begin{multline*}
\Lambda_{\Theta_m(\pi_1,\dots,\pi_m)}(a_1,\dots,a_{K_m-1}) =  \nu\Bigl(a_1 \Lambda_{\pi_1}(a_2,\dots,a_{K_1-2}) a_{K_1-1} X  \dots
\\ \dots X a_{K_{m-1}} \Lambda_{\pi_m}(a_{K_{m-1}+1}, \dots, a_{K_m-2}) a_{K_m-1} \Bigr),
\end{multline*}
using the notation from Definition \ref{def:nesting} and the convention that $a_0 \Lambda_\varnothing(a_1,\dots,a_{-1}) a_0 = a_0$.  An induction argument similar to the proof of Theorem \ref{thm:combinatoricsestimate} shows that
\[
Q_{\pi;s,t}(a_1,\dots,a_{k-1}) = (t - s)^{|\pi|} \alpha_\pi \Lambda_\pi(a_1,\dots,a_{k-1}),
\]
so that
\begin{equation} \label{eq:semigroupcumulantformula}
\mu_{s,t} (Xa_1X \dots a_{k-1}X) = \sum_{\pi \in NC_{\geq 2}(k)} (t - s)^{|\pi|} \alpha_\pi \Lambda_\pi(a_1,\dots,a_{k-1}).
\end{equation}
This means precisely that if $Y_{s,t}$ is an operator with the law $\mu_{s,t}$ and if $K_\pi$ denotes the operator-valued monotone cumulant in \cite[Theorem 3.4]{HS2014}, then for $\pi \in NC_{\geq 2}$, we have
\[
K_\pi(Y_{s,t},a_1Y_{s,t},\dots,a_{k-1}Y_{s,t}) = (t - s)^{|\pi|} \Lambda_\pi(a_1,\dots,a_{k-1}).
\]
Here we are using the fact that because $Y_{s,t}$ has mean zero, the cumulants $K_\pi$ will vanish if $\pi$ has any singleton blocks, and thus the moment-cumulant formula of \cite[Theorem 3.4]{HS2014} reduces to a sum over $NC_{\geq 2}(k)$ rather than $NC(k)$.  The preceding argument shows that
\[
K_n(Y_{s,t}, a_1 Y_{s,t}, \dots, a_{n-1} Y_{s,t}) = (t - s) \nu(a_1Xa_2 \dots Xa_{n-1}).
\]
In other words, the \emph{cumulants} of the law $\mu_{s,t}$ are the $(t - s)$ times the \emph{moments} of the generalized law $\nu$.  With this explanation in mind, compare Theorem \ref{thm:combinatorics} and its proof with \cite[\S 5]{Muraki2000} \cite[Definition 4.4 and Proposition 4.8]{AW2016}, \cite[Corollary 5.2, Theorem 5.3, Remark 6.4]{HS2011-2}, \cite[concluding paragraph]{HS2014}, \cite[Theorem 2.5]{BPV2013}.
\end{remark}

\subsection{Construction of a Fock Space} \label{subsec:Fockspace1}

The combinatorics of Theorem \ref{thm:combinatorics} can be modeled by using a Fock space construction similar to that of \cite[\S 4.6 - 4.7]{Speicher1998} in the free case and \cite{Lu1997} and \cite{Muraki1997} in the scalar-valued monotone case.  For a Loewner chain $F_t(z)$ with the associated map $F_{s,t}(z)$ and law $\mu_{s,t}$, we will explicitly construct an operator $Y_{s.t}$ on the Fock space which has law $\mu_{s,t}$, and in fact, $Y_{0,t}$ will be a process with monotone independent increments (Theorem \ref{thm:Fockspace}).

Let $\nu$ be the distributional family of generalized laws on $[0,T]$ corresponding to our Loewner chain.  In this subsection, we will construct the Fock space $\mathcal{H}_\nu$ (Definition \ref{def:Fockspace}) as a direct sum of tensor powers of $\mathcal{L}_{\mathcal{A}}\ip{X}$ with respect to a certain $\mathcal{A}$-valued inner product (see Lemma \ref{lem:tensorpower}).  We will define creation, annihilation, and multiplication operators on $\mathcal{H}_\nu$ (Definitions \ref{def:multiplicationoperator} and \ref{def:creationannihilationoperator}).  Finally, we will equip $B(\mathcal{H}_\nu)$ with the structure of an $\mathcal{A}$-valued probability space (Observation \ref{obs:Fockprobabilityspace}).  In the next subsection, we will use this setup to define the operators $Y_{s,t}$ realizing the law $\mu_{s,t}$.

\begin{lemma}
The map $I_\nu: \mathcal{L}_{\mathcal{A}}\ip{X} \to \mathcal{C}_{\mathcal{A}}$ given by
\[
[I_\nu(f)](t) = \int_t^T \nu(f(X,s),s)\,ds
\]
is completely positive.
\end{lemma}

\begin{proof}
We must show that for $P(X,\cdot) \in M_n(\mathcal{L}_{\mathcal{A}}\ip{X})$ and for $t \in [0,T]$, we have
\[
\int_t^T \nu^{(n)}(P(X,s)^*P(X,s),s)\,ds \geq 0.
\]
As remarked above, $\nu(f_1(t) X f_2(t) \dots X f_k(t),\cdot)$ depends continuously on $a_1$, \dots, $a_k$ in $\mathcal{L}_{\mathcal{A}}$.  Thus, it suffices to consider the case where each entry of $P(X,\cdot)$ is a sum of monomials in $\mathcal{L}_{\mathcal{A}}\ip{X}$ where each coefficient is a simple function.  Such a matrix $P(X,\cdot)$ can be expressed as
\[
\sum_{j=1}^\infty \chi_{E_j}(t) P_j(X)
\]
where the $E_j$'s are disjoint and measurable and $P_j(X) \in M_n(\mathcal{A}\ip{X})$.  We then have
\[
\int_t^T \nu^{(n)}(P(X,s)^*P(X,s),s)\,ds = \sum_{j=1}^\infty \int_t^T \nu^{(n)}(P_j(X)^* P_j(X),s) \chi_{E_j}(s)\,ds.
\]
Because $\int_t^T \nu(\cdot,s) \chi_{E_j}(s)\,ds$ is a generalized law (see Definition \ref{def:distributionallaw}, it is in particular completely positive, and hence $\int_t^T \nu^{(n)}(P_j(X)^*P_j(X),s) \chi_{E_j}(s)\,ds \geq 0$.  This implies that $\int_t^T \nu^{(n)}(P(X,s)^* P(X,s),s)\,ds \geq 0$ as desired.
\end{proof}

\begin{lemma} \label{lem:tensorpower}
Let us denote $I_{\nu,0}(f) = I_\nu(f)|_{t=0}$ for $f \in \mathcal{L}_{\mathcal{A}}\ip{X}$.  The sequilinear form on $\mathcal{L}_{\mathcal{A}}\ip{X}^{\otimes_{\alg} k} \otimes_{\alg} \mathcal{A}$ given by
\[
\ip{f_k \otimes \dots \otimes f_1 \otimes a, g_k \otimes \dots \otimes g_1 \otimes a'} = a^* I_{\nu,0}[f_1^*I_\nu[f_2^* \dots I_\nu[f_n^*g_n] \dots g_2]g_1] a'
\]
is a $\mathcal{A}$-valued pre-inner product, and therefore Lemma \ref{lem:rightHilbertmodule} (4) the completed quotient with respect to this inner product is a right Hilbert $\mathcal{A}$-module.
\end{lemma}

\begin{notation}
We denote the right Hilbert $\mathcal{A}$-module constructed in the lemma by
\[
\mathcal{H}_{\nu,k} = \underbrace{\mathcal{L}_{\mathcal{A}}\ip{X} \otimes_{I_\nu} \dots \otimes_{I_\nu} \mathcal{L}_{\mathcal{A}}\ip{X}}_k \otimes_{I_{\nu,0}} \mathcal{A}.
\]
\end{notation}

\begin{proof}[Proof of Lemma \ref{lem:tensorpower}]
It is clear that the inner product is right $\mathcal{A}$-linear and symmetric.  To prove positivity, we proceed as in the proof of Proposition \ref{prop:operatorvaluedGNS}.  Consider a sum of simple tensors
\[
h = \sum_{j=1}^n f_{j,k} \otimes \dots \otimes f_{j,1} \otimes a_j
\]
with $f_{j,k}$, \dots, $f_{j,1} \in \mathcal{L}_{\mathcal{A}}\ip{X}$ and $f_{j,0} \in \mathcal{C}_{\mathcal{A}}$.  For $i = 1$, \dots, $k - 1$, let $D_i$ be the $n \times n$ diagonal matrix $\diag(f_{1,i}, \dots, f_{n,i})$.  Let $D_0$ be the diagonal matrix $\diag(a_1,\dots,a_n)$.  Let $P$ be the matrix with $P_{i,j} = f_{i,k}^* f_{j,k}$.  Let $\Tr: M_n(\mathcal{C}_\mathcal{A}) \to \mathcal{C}_\mathcal{A}$ be the sum of the diagonal entries.  Then
\begin{align*}
\ip{h,h} &= \sum_{i,j=1}^n a_i^* I_{\nu,0}[f_{i,1}^* I_\nu[ f_{i,2}^* \dots I_\nu[f_{i,k-1}^* I_\nu[f_{i,k}^*f_{j,k}] f_{j,k-1}] \dots f_{j,2}] f_{j,1}] a_j \\
&= \Tr\bigl( D_0^* I_{\nu,0}^{(n)}[D_1^* I_\nu^{(n)}[D_2^* \dots D_{k-1}^*I_\nu^{(n)}[P]D_{k-1} \dots D_2]D_1]D_0 \bigr).
\end{align*}
Now $P$ is positive in $M_n(\mathcal{L}_\mathcal{A}\ip{X})$ by construction.  Hence, $I_\nu^{(n)}[P]$ is positive in $M_n(\mathcal{C}_{\mathcal{A}})$ by complete positivity of $I_\nu$.  It follows that $D_{k-1}^* I_\nu^{(n)}(P) D_{k-1}$ is positive in $M_n(\mathcal{L}_{\mathcal{A}}\ip{X})$ and hence $I_\nu^{(n)}[D_{k-1}^* I_\nu^{(n)}(P) D_{k-1}]$ is positive in $M_n(\mathcal{C}_{\mathcal{A}})$.  Continuing inductively, we obtain
\[
D_0^* I_{\nu,0}^{(n)}[D_1^* I_\nu^{(n)}[D_2^* \dots I_\nu^{(n)}[D_{k-1}^*I_\nu^{(n)}[P]D_{k-1}] \dots D_2]D_1]D_0 \geq 0 \text{ in } M_n(\mathcal{A}).
\]
Thus, taking the trace, we have $\ip{h,h} \geq 0$.
\end{proof}

\begin{definition} \label{def:Fockspace}
The \emph{monotone Fock space} $\mathcal{H}_\nu$ is the right Hilbert $\mathcal{A}$-module defined by
\[
\mathcal{H}_\nu := \mathcal{A} \oplus \bigoplus_{k=1}^\infty \mathcal{H}_{\nu,k},
\]
where $\mathcal{H}_{\nu,k}$ is given by Lemma \ref{lem:tensorpower}.  We denote by $\xi$ the vector $1$ in the first summand $\mathcal{A}$, and henceforth we denote this subspace by $\mathcal{A} \xi$ or $\mathcal{H}_{\nu,0}$ rather than $\mathcal{A}$.
\end{definition}

Next, we define certain operators on $B(\mathcal{H}_\nu)$ called creation, annihilation, and multiplication operators, which will be used to build the operators $Y_{s,t}$ in Theorem \ref{thm:Fockspace} below.  We first consider the multiplication operators.

\begin{lemma} \label{lem:multiplicationoperator}
Let $k \geq 1$ and $f \in \mathcal{L}_{\mathcal{A}}\ip{X}$.  Then there is a unique bounded operator $\mathfrak{m}_k(f): \mathcal{H}_{\nu,k} \to \mathcal{H}_{\nu,k}$ given by
\[
\mathfrak{m}_k(f)[f_k \otimes \dots \otimes f_1 \otimes a] = (f \cdot f_k) \otimes f_{k-1} \otimes \dots \otimes f_1 \otimes a.
\]
\end{lemma}

\begin{proof}
Note that $\mathcal{L}_{\mathcal{A}}$ is a $C^*$-algebra and $I_\nu: \mathcal{L}_{\mathcal{A}}\ip{X} \to \mathcal{C}_{\mathcal{A}} \subseteq \mathcal{L}_{\mathcal{A}}$ is a completely positive map satisfying $\norm{I_\nu(f_0 X f_1 \dots X f_k)} \leq M^k \norm{f_0} \dots \norm{f_k}$.  In other words, $I_\nu$ is an $\mathcal{L}_{\mathcal{A}}$-valued generalized law.  Hence, by Proposition \ref{prop:CPmap}, multiplication by $f \in \mathcal{L}_{\mathcal{A}}\ip{X}$ defines a bounded operator $\pi(f(X))$ on the right Hilbert $\mathcal{L}_{\mathcal{A}}$-module $\mathcal{M} := \mathcal{L}_{\mathcal{A}}\ip{X} \otimes_{I_\nu} \mathcal{L}_{\mathcal{A}}$.

This implies that $\norm{\pi(f)}^2 - \pi(f^*f) \geq 0$ in the $C^*$-subalgebra of $B(\mathcal{M})$ generated by $\pi(\mathcal{L}_\mathcal{A}\ip{X})$ and hence can be written as $g^*g$ for some $g$ in the $C^*$-algebra.  There exist functions $g_n \in \mathcal{L}_{\mathcal{A}}\ip{X}$ with $\pi(g_n) \to g$ and hence $\pi(g_n^*g_n) \to \norm{\pi(f)}^2 - \pi(f^*f)$.  This implies that for $p, q \in \mathcal{L}_{\mathcal{A}}\ip{X}$,
\[
I_\nu(p^*g_n^* g_nq) \to I_\nu(p^*(\norm{\pi(f)}^2 - f^*f)q).
\]
It follows that for $h \in \mathcal{L}_{\mathcal{A}}^{\otimes_{\alg} k} \otimes_{\alg} \mathcal{A}$, we have
\[
\ip{g_nh,g_nh} = \ip{h,g_n^*g_nh} \to \ip{h, (\norm{\pi(f)}^2 - f^*f) h}.
\]
Thus, $\ip{h, (\norm{\pi(f)}^2 - f^*f) h} \geq 0$ which means that
\[
\ip{fh,fh} \leq \norm{\pi(f)}^2 \ip{h,h}.
\]
This is sufficient to show that multiplication by $f$ defines a bounded operator on the completed quotient $\mathcal{H}_{\nu,k}$, as in Proposition \ref{prop:operatorvaluedGNS}.
\end{proof}

\begin{definition} \label{def:multiplicationoperator}
For $f \in \mathcal{L}_{\mathcal{A}}\ip{X}$, we define the \emph{multiplication operator} $\mathfrak{m}(f): \mathcal{H}_\nu \to \mathcal{H}_\nu$ as the direct sum of the operators $\mathfrak{m}_k(f): \mathcal{H}_{\nu,k} \to \mathfrak{H}_{\nu,k}$ for $k \geq 1$ and of the zero operator on $\mathcal{A} \xi$ for $k = 0$.
\end{definition}

\begin{lemma} \label{lem:creationoperator}
For $f \in \mathcal{L}_{\mathcal{A}}\ip{X}$, there exists a unique operator $\ell(f) \in B(\mathcal{H}_\nu)$ such that
\[
\ell(f)[f_k \otimes \dots \otimes f_1 \otimes a] = f \otimes f_k \otimes \dots \otimes f_1 \otimes a.
\]
The adjoint of $\ell(f)$ satisfies
\[
\ell(f)^*[f_k \otimes \dots \otimes f_1 \otimes a] =
\begin{cases}
I_\nu(f^*f_k) f_{k-1} \otimes \dots \otimes f_1 \otimes a, & k \geq 2 \\
I_{\nu,0}(f^*f_1)a, & k = 1 \\
0, & k = 0,
\end{cases}
\]
and we also have
\[
\norm{\ell(f)} = \norm{I_\nu(f^*f)}^{1/2} = \norm{f \otimes 1}_{\mathcal{H}_{\nu,1}}^{1/2}.
\]
\end{lemma}

\begin{definition} \label{def:creationannihilationoperator}
We call $\ell(f)$ the \emph{creation operator} and $\ell(f)^*$ the \emph{annihilation operator} associated to $f$.
\end{definition}

\begin{proof}[Proof of Lemma \ref{lem:creationoperator}]
A direct computation checks that for $k \geq 1$, for $h \in \mathcal{L}_{\mathcal{A}}\ip{X}^{\otimes_{\alg} k} \otimes_{\alg} \mathcal{A}$, we have
\[
\ip{\ell(f) h, \ell(f) h} = \ip{h, I_\nu(f^*f)h},
\]
Note $I_\nu(f^*f)$ is an element of $\mathcal{L}_{\mathcal{A}}$ and we already checked that the multiplication action by $\mathcal{L}_{\mathcal{A}}\ip{X}$ is bounded.  In fact, since $\mathcal{L}_{\mathcal{A}}$ is a $C^*$-algebra, we have
\[
\norm{\mathfrak{m}(I_\nu(f^*f))} \leq \norm{I_\nu(f^*f)}.
\]
Hence, $\norm{\ell(f)h} \leq \norm{I_\nu(f^*f)}^{1/2} \norm{h}$ and thus $\ell(f)$ defines a bounded operator $\mathcal{H}_{\nu,k} \to \mathcal{H}_{\nu,k+1}$.  In the case $k = 0$, we have $\ip{\ell(f)a, \ell(f)a} = a^* I_{\nu,0}(f^*f) a$ and thus $\ell(f)$ also defines an operator $\mathcal{H}_{\nu,0} \to \mathcal{H}_{\nu,1}$ with norm bounded by $\norm{I_\nu(f^*f)}^{1/2}$.

Altogether, $\ell(f)$ defines a bounded operator $\mathcal{H}_\nu \to \mathcal{H}_\nu$ with norm less than or equal to $\norm{I_\nu(f^*f)}^{1/2}$.  Furthermore, observe that $I_\nu(f^*f)$ is a nonnegative decreasing function $[0,T] \to \mathcal{A}$ and hence
\[
\norm{I_\nu(f^*f)} = \norm{I_\nu(f^*f)(0)} = \norm{I_{\nu,0}(f^*f)} = \norm{f \otimes 1}_{\mathcal{H}_{\nu,1}}^2.
\]
Since $\ell(f)\xi = f \otimes 1$, we have
\[
\norm{f \otimes 1}_{\mathcal{H}_{\nu,1}} \leq \norm{\ell(f)} \leq \norm{I_\nu(f^*f)}^{1/2} = \norm{f \otimes 1}_{\mathcal{H}_{\nu,1}},
\]
and thus all the inequalities are equalities.  The formula for $\ell(f)^*$ is a computation we leave to the reader.
\end{proof}

We gather some elementary properties of the creation, annihilation, and multiplication operators for future reference.

\begin{observation} \label{obs:Fockrelations1}
Let $f, g \in \mathcal{L}_{\mathcal{A}}$.  Let $P_{\mathcal{A}\xi} \in B(\mathcal{H}_\nu)$ be the projection onto $\mathcal{A} \xi$.  Then
\begin{align}
\mathfrak{m}(f) \mathfrak{m}(g) &= \mathfrak{m}(fg) \\
\mathfrak{m}(f) \ell(g) &= \ell(fg) \\
\ell(f)^* \mathfrak{m}(g) &= \ell(g^*f)^* \\
\ell(f)^* \ell(g) &= I_{\nu,0}(f^*g) P_{\mathcal{A}\xi} + \mathfrak{m}(I_\nu(f^*g)).
\end{align}
\end{observation}

In light of the last identity, we establish the notation
\begin{equation} \label{eq:definem0}
\mathfrak{m}_0(f) = f(0) P_{\mathcal{A}\xi} + \mathfrak{m}(f),
\end{equation}
for $f \in C([0,T],\mathcal{A})$.  Note that $\mathfrak{m}_0$ defines a unital $*$-homomorphism from $\mathcal{C}_{\mathcal{A}} \to B(\mathcal{H}_\nu)$.  In particular, $\mathfrak{m}_0$ restricts to a $*$-homomorphism $\mathcal{A} \to B(\mathcal{H}_\nu)$.  Moreover, this map is injective since $\ip{\xi, \mathfrak{m}_0(a) \xi} = a$.  The embedding $\mathfrak{m}_0|_{\mathcal{A}}$ of $\mathcal{A}$ into $B(\mathcal{H}_\nu)$ allows us to endow $B(\mathcal{H}_\nu)$ with the structure of an $\mathcal{A}$-valued probability space.

\begin{observation} \label{obs:Fockprobabilityspace}
Let us view $\mathcal{A}$ as a subalgebra of $B(\mathcal{H}_\nu)$ by identifying $\mathcal{A}$ with the its image $\mathfrak{m}_0(\mathcal{A})$ in $B(\mathcal{H}_\nu)$.  Then the vector $\xi$ is an $\mathcal{A}$-central unit vector in $\mathcal{H}_\nu$.  Thus, the map $E_\nu: B(\mathcal{H}_\nu) \to \mathcal{A}$ given by $b \mapsto \ip{\xi,b\xi}$ is an $\mathcal{A}$-valued expectation.  Therefore, $(B(\mathcal{H}_\nu),E_\nu)$ is an $\mathcal{A}$-valued probability space.
\end{observation}

\subsection{Realization of $\mu_{t_1,t_2}$ on the Fock Space}

\begin{theorem} \label{thm:Fockspace}
Let $\nu$ be a distributional family of generalized laws on $[0,T]$, and let $V$ be the corresponding distributional Herglotz vector field.  Let $(F_t(z))_{t \in [0,T]}$ be the Loewner chain generated by $V$.  For $t_1 \leq t_2$, let $F_{t_1,t_2}(z)$ be the subordination map.  Let $\mu_{t_1,t_2}$ be the $\mathcal{A}$-valued law such that $F_{t_1,t_2}$ is the $F$-transform of $\mu_{t_1,t_2}$.  Let $(\mathcal{H}_\nu,E_\nu)$ be the Fock space defined in the previous section and define the operator $Y_{t_1,t_2} \in B(\mathcal{H}_\nu)$ by
\begin{equation}
Y_{t_1,t_2} = \ell(\chi_{(t_1,t_2)}) + \ell(\chi_{(t_1,t_2)})^* + \mathfrak{m}(\chi_{(t_1,t_2)} X)
\end{equation}
Then
\begin{enumerate}[(1)]
	\item If $t_1 \leq t_2 \leq t_3$, then $Y_{t_1,t_2} + Y_{t_2,t_3} = Y_{t_1,t_3}$.
	\item The law of $Y_{t_1,t_2}$ with respect to $E_\nu$ is $\mu_{t_1,t_2}$.
	\item $\norm{Y_{t_1,t_2}} \leq \sqrt{2C(t_2 - t_1)} + \rad(\nu)$.
	\item If $t_0 < \dots < t_N$, then $Y_{t_0,t_1}$, \dots, $Y_{t_{N-1},t_N}$ are monotone independent in $(B(\mathcal{H}_\nu),E_\nu)$.
\end{enumerate}
\end{theorem}

\begin{proof}[Proof of Theorem \ref{thm:Fockspace}] (1) is immediate because $\ell(\chi_{(t_1,t_2)}) + \ell(\chi_{(t_2,t_3})) = \ell(\chi_{(t_1,t_3})$ and $\mathfrak{m}(\chi_{(t_1,t_2)}X) + \mathfrak{m}(\chi_{(t_2,t_3)} X) = \mathfrak{m}(\chi_{(t_1,t_3)} X)$.

(2) By Theorem \ref{thm:combinatorics}, it suffices to show that
\[
E_\nu[Y_{t_1,t_2} a_1 Y_{t_1,t_2} \dots a_{k-1} Y_{t_1,t_2}] = \sum_{\pi \in NC_{\geq 2}(k)} Q_{\pi;t_1,t_2}(a_1,\dots,a_{k-1}).
\]
We substitute $Y_{t_1,t_2} = \ell(\chi_{(t_1,t_2)}) + \ell(\chi_{(t_1,t_2)})^* + \mathfrak{m}(\chi_{(t_1,t_2)} X)$ and expand by linearity.  This results in the sum of $E_\nu[b_1a_1 b_2 \dots a_{k-1}b_k]$ over all possible values of $b_j \in \{\ell(\chi_{(t_1,t_2)}), \ell(\chi_{(t_1,t_2)})^*, \mathfrak{m}(\chi_{(t_1,t_2)} X)\}$.

{\bf Step 1:} Let us say that a choice of $b_1$, \dots, $b_k$ as above is \emph{compatible} with the partition $\pi \in NC_{\geq 2}(k)$ if for every block $B$ of $\pi$, we have
\[
b_j =
\begin{cases}
\ell(\chi_{(t_1,t_2)})^*, & j = \min B \\
\mathfrak{m}(\chi_{(t_1,t_2)} X), & j \in B \setminus \{\min B, \max B\} \\
\ell(\chi_{(t_1,t_2)}), & j = \max B.
\end{cases}
\]
We claim that $\ip{\xi, b_1a_1 b_2 \dots a_{k-1}b_k \xi}$ is zero unless there is a partition $\pi$ compatible with $b_1$, \dots, $b_k$.  We describe an algorithm that will construct a compatible partition for $b_1$, \dots, $b_k$ if one exists and otherwise will prove that ${\xi, b_1a_1 b_2 \dots a_{k-1}b_k} = 0$.

We will start with the vector $\xi$ and then apply the operators $b_k$ one at a time (starting from the right).  We will define $S_j \subseteq \{j,j+1,j+2,\dots\}$ inductively for $j = k + 1, k, k - 1, \dots$ and show that $b_j a_{j+1} \dots a_k b_k \xi$ is a simple tensor in $\mathcal{H}_{\nu,|S_j|}$.  We start by setting $S_{k+1} = \varnothing$ and note $\xi \in \mathcal{H}_{\nu,0}$.  For the inductive step, we divide into cases:
\begin{enumerate}[(A)]
	\item Suppose $b_j = \ell(\chi_{(t_1,t_2)})$.  Then set $S_j = \{j\} \cup S_{j+1}$.  Then $|S_j| = |S_{j+1}|$.  Since $b_{j+1} a_{j+1}\dots a_{k-1} b_k \xi$ is a simple tensor in $\mathcal{H}_{\nu,|S_{j+1}|}$, we see that $b_j a_j \dots a_{k-1} b_k$ is a simple tensor in $\mathcal{H}_{\nu,|S_j|}$.
	\item Suppose $b_j = \mathfrak{m}(\chi_{(t_1,t_2)} X)$.  If $S_{j+1} = \varnothing$, then $b_{j+1} a_{j+1} \dots a_{k-1} b_k \in \mathcal{A} \xi$.  Thus, applying $b_j$ will produce the zero vector.  In this case, we terminate the algorithm.  Otherwise, we set $S_j = S_{j+1}$.  Note that $b_ja_j \dots a_{k-1} b_k$ is still a simple tensor in $\mathcal{H}_{\nu;|S_j|}$.
	\item Suppose $b_j = \ell(\chi_{(t_1,t_2)})^*$.  If $S_{j+1} = \varnothing$, then $b_{j+1} a_{j+1} \dots a_{k-1} b_k \in \mathcal{A}\xi$.  Hence, applying $b_j$ will produce the zero vector.  In this case, we terminate the algorithm.  Otherwise, we define $S_j = S_{j+1} \setminus \{\min S_{j+1}\}$, and note that $b_j a_j \dots a_{k-1} b_k$ is a simple tensor in $\mathcal{H}_{\nu,|S_j|}$.
\end{enumerate}
Suppose that the algorithm completes the step $j = 1$ without terminating.  If $S_1 \neq \varnothing$, then $b_1a_1 \dots a_{k-1} b_k \xi$ is in $\mathcal{H}_{\nu,|S_1|}$ and hence is orthogonal to $\xi$.  Therefore, $\ip{\xi, b_1a_1 \dots a_{k-1} b_k \xi} = 0$.

If $S_1 = \varnothing$, we define a partition $\pi$ as follows.  Let $\mathcal{J}$ be the set of indices for which $b_j$ is a creation operator.  For $j \in \mathcal{J}$, define $B_j = \{i: \min (S_i \cup S_{i+1}) = j\}$, where
\[
\min(S_i \cup S_{i+1}) = \begin{cases}
i = \min S_i, & b_i = \ell(\chi_{(t_1,t_2)}) \\
\min S_i, & b_i = \mathfrak{m}(\chi_{(t_1,t_2)}X) \\
\min S_{i+1}, & b_i = \ell(\chi_{(t_1,t_2)})^*
\end{cases}
\]
One can check that $\min(S_i \cup S_{i+1}) < +\infty$ in all cases; for instance, in the case of a multiplication operator, we must have $S_i \neq \varnothing$ because otherwise we would have terminated the algorithm in step (B) above.  To show that $\pi$ is non-crossing, consider two blocks $B_j$ and $B_{j'}$ with $j < j'$.  If $i_1$ and $i_2$ are two indices of $B_j$ with $i_1 < i_2$, then for $i$ between $i_1$ and $i_2$, the minimum of the list $S_i$ must be $\leq j$, and hence there cannot be any elements of $B_j$ between $i_1$ and $i_2$.  This in particular rules out the possibility of a crossing.

{\bf Step 2:} It remains to prove that $E_\nu[a_0 b_1 a_1 \dots a_{k-1} b_k a_k] = a_0 Q_{\pi;t_1,t_2}(a_1,\dots,a_{k-1}) a_k$ when $b_1$, \dots, $b_k$ is the string corresponding to a partition $\pi \in NC_{\geq 2}(k)$.  It suffices to show that for every such partition $\pi$,
\begin{equation} \label{eq:partitionidentity}
a_0 b_1 a_1 \dots a_{k-1} b_k a_k = \mathfrak{m}_0(f_{\pi,a_0,\dots,a_k}),
\end{equation}
where $\mathfrak{m}_0$ is given by \eqref{eq:definem0} and where
\[
f_{\pi;a_1,\dots,a_k}(t) =
\begin{cases}
a_0 Q_{\pi;t_1, t_2}(a_1,\dots,a_{k-1})a_k, & t \in [0,t_1] \\
a_0 Q_{\pi;t,t_2}(a_1,\dots,a_{k-1}) a_k, & t \in [t_1,t_2] \\
0, & t \in [t_2,T]
\end{cases}
\]
We verify this by induction on $k$, using the inductive framework and conventions regarding the empty partition from Lemma \ref{lem:defineQ}.  The base case $k = 0$ is trivial.  In the case where $\pi = \pi_1 \pi_2$, the definition of $Q_{\pi;t_1,t_2}$ implies that
\[
f_{\pi_1\pi_2;a_0,\dots,a_k} = f_{\pi_1;a_0,\dots,a_{k_1-1},1} a_{k_1} f_{\pi_2;1,a_{k_1+1},\dots,a_{k_1+k_2}},
\]
and hence \eqref{eq:partitionidentity} holds for $\pi_1 \pi_2$ if it holds for $\pi_1$ and $\pi_2$.  Now suppose that $\pi = \Theta_m(\pi_1, \dots, \pi_m)$ with $\pi_j \in NC_{\geq 2}(k_j-1)$ and that $K_j = 1 + k_1 + \dots + k_j$ as in Definition \ref{def:nesting}.  By the inductive hypothesis,
\begin{multline*}
a_0 b_1 a_1 \dots a_{k-1} b_k a_k \\
= a_0 \ell(\chi_{(t_1,t_2)})^* f_{\pi_1;a_1,\dots,a_{K_1-1}} \mathfrak{m}(\chi_{(t_1,t_2)} X) \dots \mathfrak{m}(\chi_{(t_1,t_2)}X) f_{\pi_m;a_{K_{m-1}+1},\dots,a_{k-1}} \ell(\chi_{(t_1,t_2)}) a_k.
\end{multline*}
In light of Observation \ref{obs:Fockrelations1}, this evaluates to $I_{\nu,0}(g) P_{\mathcal{A}\xi} + \mathfrak{m}(I_\nu(g))$, where
\[
g = \chi_{(t_1,t_2)} f_{\pi_1;a_1,\dots,a_{K_1-1}} X \dots X f_{\pi_m;a_{K_{m-1}+1},\dots,a_{k-1}}
\]
But from the definition of $Q_{\pi;t_1,t_2}$, one checks that $I_\nu(g) = f_{\pi;a_0,\dots,a_k}$, and hence \eqref{eq:partitionidentity} holds in this case as well.
\end{proof}

\begin{remark}
The argument in Step 1 formalizes the physical intuition that $\ell(\chi_{(t_1,t_2)})$ creates a particle and $\ell(\chi_{(t_1,t_2)})^*$ annihilates a particle.  The set $S_j$ represents the list of particles that exist at the time indexed by $j$ (and unfortunately time is indexed backwards).  Each creation operator produces a new particle, each multiplication operator acts on the last particle created that still exists, and the each annihilation operator destroys the last particle created that still exists.  The operators that create, annihilate, or transform the same particle are put into the same block of the partition.  Versions of this construction are standard in non-commutative probability, see e.g.\ \cite[\S 2.1.3]{AGZ2009}.
\end{remark}

We prove Theorem \ref{thm:Fockspace} (3) in \S \ref{subsec:Fockspacenorm} and (4) in \S \ref{subsec:Fockspaceindependence}

\subsection{Norm of Operators on the Fock Space} \label{subsec:Fockspacenorm}

Now we turn to the proof of (3) of Theorem \ref{thm:Fockspace}.  To set the stage, we remark that by Lemma \ref{lem:creationoperator}, we would have
\[
\norm{\ell(\chi_{(t_1,t_2)})} + \norm{\ell(\chi_{(t_1,t_2)})} + \norm{\mathfrak{m}(\chi_{(t_1,t_2)}X)}
\leq \sqrt{C(t_2 - t_1)} + \sqrt{C(t_2 - t_1)} + M,
\]
which gives the bound with $2 \sqrt{C(t_2 - t_1)} + M$ instead of $\sqrt{2C(t_2 - t_1)} + M$.  However, it is reasonable to hope for the sharper bound $\sqrt{2C(t_2 - t_1)}$ because we showed in Theorem \ref{thm:Loewnerintegration} that $\rad(\mu_{t_1,t_2}) \leq \sqrt{2C(t_2 - t_1)} + M$.

At this point, it is tempting to assume that $\norm{Y_{t_1,t_2}} = \rad(\mu_{t_1,t_2})$, but we have not proved this.  The most that we can say a priori is that the norm of $Y_{t_1,t_2}$ restricted to the subspace $\mathcal{K} = \overline{\mathcal{A}\ip{Y_{t_1,t_2}} \xi}$ equals $\rad(\mu_{t_1,t_2})$.  We do not know that the restriction map $B(\mathcal{H}_\nu) \to B(\mathcal{K})$ is injective on the $C^*$-algebra generated by $Y_{t_1,t_2}$, so we cannot conclude $\norm{Y_{t_1,t_2}} = \rad(\mu_{t_1,t_2})$.

One feasible approach to proving (3) is to evaluate $\ip{h, Y_{t_1,t_2}^m h}$ combinatorially when $h$ is a simple tensor similar to Theorem \ref{thm:Fockspace} (2) and then estimate each term similarly to Theorem \ref{thm:combinatoricsestimate}.  However, we will instead take an algebraic approach based on the following observation.

\begin{lemma} \label{lem:algebraicformula}
With the setup of Theorem \ref{thm:Fockspace}, fix $0 \leq t_1 \leq t_2 \leq T$ and fix $n$.  For $z \in \h^{(n)}(\mathcal{A})$, define
\[
\mathcal{F}^{(n)}(z) \in C([0,T],M_n(\mathcal{A}))
\]
by
\[
\mathcal{F}^{(n)}(z)(t) =
\begin{cases}
F_{\mu_{t_1,t_2}}^{(n)}(z), & t \in [0,t_1], \\
F_{\mu_{t,t_2}}^{(n)}(z), & t \in [t_1,t_2], \\
z, & t \in [t_2,T].
\end{cases}
\]
Let us abbreviate
\begin{align*}
\mathcal{F}_\xx(z) &= \mathfrak{m}_0^{(n)}[\mathcal{F}(z)] \in M_n(B(\mathcal{H}_\nu)) \\
\ell_{\xx} &= \ell(\chi_{(t_1,t_2)})^{(n)} \in M_n(B(\mathcal{H}_\nu)) \\
\ell_{\xx}^* &= [\ell(\chi_{(t_1,t_2)})^*]^{(n)} \in M_n(B(\mathcal{H}_\nu)) \\
\mathfrak{m}_{\xx} &= \mathfrak{m}(\chi_{(t_1,t_2)} X)^{(n)} \in M_n(B(\mathcal{H}_\nu))
\end{align*}
Then the following relation holds in $M_n(B(\mathcal{H}_\nu))$ for $z \in \h^{(n)}(\mathcal{A})$
\begin{multline}
[z - (\mathfrak{m}_\xx + \ell_\xx + \ell_\xx^*)]^{-1} \\
= \left[ 1 - (\mathcal{F}_\xx(z) - \mathfrak{m}_\xx)^{-1} \ell_\xx \right]^{-1} (\mathcal{F}_\xx(z) - \mathfrak{m}_\xx)^{-1} \left[ 1 - \ell_\xx^* (\mathcal{F}_\xx(z) - \mathfrak{m}_\xx)^{-1} \right]^{-1} \label{eq:algebraicformula1}
\end{multline}
Moreover, for invertible $z$ in a neighborhood of zero,
\begin{multline}
[z^{-1} - (\mathfrak{m}_\xx + \ell_\xx + \ell_\xx^*)]^{-1} \\
= \sum_{j,k=0}^\infty \left[ (\mathcal{F}_\xx(z^{-1}) - \mathfrak{m}_\xx)^{-1} \ell_\xx \right]^j (\mathcal{F}_\xx(z^{-1}) - \mathfrak{m}_\xx)^{-1}  \left[ \ell_\xx^* (\mathcal{F}_\xx(z^{-1}) - \mathfrak{m}_\xx)^{-1} \right]^k, \label{eq:algebraicformula2}
\end{multline}
and this function extends to be fully matricial for $z \in B_{M_n(\mathcal{A})}(0,R)$ for some $R > 0$ independent of $n$.
\end{lemma}

\begin{proof}
Using Observation \ref{obs:Fockrelations1}, one has
\[
\ell_\xx^* (\mathcal{F}_\xx(z) - \mathfrak{m}_\xx)^{-1} \ell_\xx
= \mathfrak{m}_0[I_\nu^{(n)}(\chi_{(t_1,t_2)} (\mathcal{F}^{(n)}(z) - X^{(n)})^{-1})],
\]
where
\begin{align*}
[I_\nu^{(n)}(\chi_{(t_1,t_2)} (\mathcal{F}(z) - X^{(n)})^{-1})](t) &= \int_t^T \nu^{(n)}((\mathcal{F}^{(n)}(z) - X)^{-1},s) \chi_{(t_1,t_2)}(s)\,ds \\
&= -\int_t^T V^{(n)}(\mathcal{F}(z), s) \chi_{(t_1,t_2)}(s) \,ds, \\
&= z - \mathcal{F}(z),
\end{align*}
where the last equality follows from \eqref{eq:FSTdiffeq} (one checks this by considering the cases $t \in [0,t_1]$, $t \in [t_1,t_2]$, and $t \in [t_2,t_3]$).  Thus,
\[
\ell_\xx^* (\mathcal{F}_\xx(z) - \mathfrak{m}_\xx)^{-1} \ell_\xx = z - \mathfrak{m}_0(\mathcal{F}(z)) = z - \mathcal{F}_\xx(z).
\]
This implies that
\begin{align*}
z - \mathfrak{m}_\xx - \ell_\xx - \ell_\xx^* &= \mathcal{F}_\xx(z) - \mathfrak{m}_\xx - \ell_\xx - \ell_\xx^* + \ell_\xx^* (\mathcal{F}_\xx(z) - \mathfrak{m}_\xx)^{-1} \ell_\xx \\
&= [1 - \ell_\xx^* (\mathcal{F}_\xx(z) - \mathfrak{m}_\xx)^{-1}] (\mathcal{F}_\xx(z) - \mathfrak{m}_\xx)^{-1} [1 - (\mathcal{F}_\xx(z) - \mathfrak{m}_\xx)^{-1} \ell_\xx].
\end{align*}
Upon taking inverses, we obtain \eqref{eq:algebraicformula1}.  To prove \eqref{eq:algebraicformula2}, note that $\mathcal{F}_\xx(z^{-1})^{-1}$ extends to be fully matricial and uniformly bounded on $B_{M_n(\mathcal{A})}(0,R_1)$ for some $R_1 > 0$ independent of $n$, and it vanishes at zero; indeed, this follows from the fact that $\mathcal{F}(z^{-1})^{-1}(t)$ is the Cauchy transform of $\mu_{s,t_2}$ where $s = \min(\max(t,t_1),t_2)$.  Thus,
\[
(\mathcal{F}_\xx(z^{-1}) - \mathfrak{m}_\xx)^{-1} = \sum_{k=0}^\infty [\mathcal{F}_\xx(z^{-1})^{-1} \mathfrak{m}_\xx]^k \mathcal{F}_\xx(z^{-1})^{-1}
\]
extends to be fully matricial and uniformly bounded on $B_{M_n(\mathcal{A})}(0,R_2)$ for some $R_2$ independent of $n$, and it vanishes at zero.  We can therefore apply the geometric series expansions for $[1 - \ell_\xx^* (\mathcal{F}_\xx(z^{-1}) - \mathfrak{m}_\xx)^{-1}]^{-1}$ and $[1 - (\mathcal{F}_\xx(z^{-1}) - \mathfrak{m}_\xx)^{-1} \ell_\xx]^{-1}$ for $z \in B_{M_n(\mathcal{A})}(0,R)$ for some $R > 0$ independent of $n$.
\end{proof}

\begin{proof}[Proof of Theorem \ref{thm:Fockspace} (3)]
We want to show that $\norm{Y_{t_1,t_2}h} \leq (\sqrt{2C(t_2 - t_1)} + M) \norm{h}$ for $h \in \mathcal{H}_\nu$.  It suffices to consider $h$ in a dense subspace.  Moreover, if $\rho_h$ is the generalized law
\[
\rho_h( f(X) ) = \ip{h, f(Y_{t_1,t_2}) h},
\]
then we have $\ip{h,Y_{t_1,t_2}^2h} \leq \rad(\rho_h) \norm{\rho_h(1)} = \rad(\rho_h) \norm{h}^2$, and hence it suffices to show that $\rad(\rho_h) \leq \sqrt{2C(t_2 - t_1)} + M$ for $h$ in a dense subspace of $\mathcal{H}_\nu$.

Let $h$ be a vector in the algebraic direct sum of the $\mathcal{H}_{\nu;k}$'s, that is, $\sum_{j=1}^N \mathcal{H}_{\nu;j}$ for some $N \in \N$.  Note that $\mathcal{H}_{\nu^{(n)}}$ can naturally be identified with $M_n(\mathcal{H}_\nu)$, and we have
\[
G_{\rho_h}^{(n)}(z) = \ip{h^{(n)}, (z - Y_{t_1,t_2}^{(n)})^{-1} h^{(n)}}.
\]
In the notation of Lemma \ref{lem:algebraicformula}, this equals
\begin{multline}
 \ip{h^{(n)}, (z^{-1} - \mathfrak{m}_\xx - \ell_\xx - \ell_\xx^*)^{-1} h^{(n)}} \\
 = \sum_{j,k=0}^N \ip{h^{(n)}, \left[ (\mathcal{F}_\xx(z^{-1}) - \mathfrak{m}_\xx)^{-1} \ell_\xx \right]^j (\mathcal{F}_\xx(z^{-1}) - \mathfrak{m}_\xx)^{-1}  \left[ \ell_\xx^* (\mathcal{F}_\xx(z^{-1}) - \mathfrak{m}_\xx)^{-1} \right]^k h^{(n)}}, \label{eq:rhohexpansion}
\end{multline}
where the infinite sum in \eqref{eq:algebraicformula2} is truncated to degree $N$ because applying more than $N$ annihilation operators would kill $h^{(n)}$.  Recall that $\mathcal{F}(z)(t)$ is the $F$ transform of $\mu_{s,t_2}$ for some $s \in [t_1,t_2]$; the solution of the Loewner equation is given by solving the ODE in Theorem \ref{thm:Herglotzflow} and therefore by \eqref{eq:solutionniceestimate}, we have for $u > t_2 - t_1$ that
\[
\norm{z} < \frac{1}{M + \sqrt{2Cu}} \implies \norm{\mathcal{F}(z^{-1})^{-1}} \leq \frac{1}{M + \sqrt{2C(u - t_2 + t_1)}}.
\]
However, since $\norm{\mathfrak{m}_\xx} \leq M$, we have by the geometric series argument that
\[
\norm{\mathcal{F}(z^{-1})^{-1}} \leq \frac{1}{M + \sqrt{2C(u - t_2 + t_1}} \implies \norm{(\mathcal{F}(z^{-1})^{-1} - \mathfrak{m}_\xx)^{-1}} \leq \frac{1}{\sqrt{2C(u - t_2 + t_1)}}.
\]
In particular, \eqref{eq:rhohexpansion} is fully matricial and bounded for $\norm{z} < 1 / (M + \sqrt{2Cu})$.  Because this holds for every $u > t_2 - t_1$, we have $\rad(\rho_h) \leq M + \sqrt{2C(t_2 - t_1)}$ by Theorem \ref{thm:Cauchytransform}.
\end{proof}

\begin{remark}
Actually, Lemma \ref{lem:algebraicformula} provides a alternative proof that the law of $Y_{t_1,t_2}$ equals the law $\mu_{t_1,t_2}$ derived by solving the Loewner equation.  Indeed, if we take apply $\ip{\xi^{(n)}, \cdot \xi^{(n)}}$ to \eqref{eq:algebraicformula2}, all the terms vanish except when $j = k = 0$, so that
\[
\ip{\xi^{(n)}, (z^{-1} - Y_{t_1,t_2}^{(n)})^{-1} \xi^{(n)}} = \ip{\xi, (\mathcal{F}_\xx(z^{-1}) - \mathfrak{m}_\xx)^{-1} \xi}.
\]
Now $\mathfrak{m}_\xx$ restricted to $M_n(\mathcal{A}\xi)$ is zero and $\mathfrak{F}_\xx(z^{-1})$ restricted to $M_n(\mathcal{A} \xi)$ is $F_{\mu_{t_1,t_2}}^{(n)}(z^{-1})$.  Thus, we get $\ip{\xi^{(n)}, (z^{-1} - Y_{t_1,t_2}^{(n)})^{-1} \xi^{(n)}} = F_{\mu_{t_1,t_2}}(z^{-1})^{-1} = \tilde{G}_{\mu_{t_1,t_2}}(z)$ as desired.
\end{remark}

\subsection{Monotone Independence in the Fock Space} \label{subsec:Fockspaceindependence}

\begin{notation}
Let $0 \leq s < t \leq T$.  Let $\mathcal{L}_{\mathcal{A};s,t}$ be the subalgebra of $\mathcal{L}_{\mathcal{A}} = L_{\Boch}^\infty([0,T],\mathcal{A})$ consisting of functions supported in $[s,t]$.  Note that $\mathcal{L}_{\mathcal{A};s,t}\ip{X}$ can be regarded as a subalgebra of $\mathcal{L}_{\mathcal{A};s,t}$.  We denote by $\mathcal{B}_{s,t}$ be the (non-unital) $\mathcal{A}$-algebra generated by the operators $\ell(f)$, $\ell(f)^*$, and $\mathfrak{m}(f)$ for $f \in \mathcal{L}_{\mathcal{A};s,t}\ip{X}$.
\end{notation}

\begin{proposition} \label{prop:Fockindependence}
Let $0 = t_0 < t_1 < \dots < t_N = T$.  Then the algebras $\mathcal{B}_{t_0,t_1}$, \dots, $\mathcal{B}_{t_{N-1},t_N}$ are monotone independent over $\mathcal{A}$ in $(B(\mathcal{H}_\nu),E_\nu)$.
\end{proposition}

In particular, this implies that the operators $Y_{t_0,t_1}$, \dots, $Y_{t_{N-1},t_N}$ defined in Theorem \ref{thm:Fockspace} are monotone independent, which completes the proof of Theorem \ref{thm:Fockspace} (4).  We will establish Proposition \ref{prop:Fockindependence} using the following relations between the creation, annihilation, and multiplication operators.

\begin{lemma} \label{lem:Fockrelations2}
Let $0 < t < T$.  Suppose that $f \in \mathcal{L}_{\mathcal{A};0,t}\ip{X}$ and $g \in \mathcal{L}_{\mathcal{A};t,T}\ip{X}$.  Then
\begin{align}
\mathfrak{m}(f) \mathfrak{m}(g) &= \mathfrak{m}(g) \mathfrak{m}(f) = 0 \label{eq:Fockrelation2a} \\
\mathfrak{m}(f) \ell(g) &= \mathfrak{m}(g) \ell(f) = 0 \label{eq:Fockrelation2b} \\
\ell(f)^* \mathfrak{m}(g) &= \ell(g)^* \mathfrak{m}(f) = 0 \label{eq:Fockrelation2c} \\
\ell(f)^* \ell(g) &= \ell(g)^* \ell(f) = 0 \label{eq:Fockrelation2d} \\
\ell(f) \ell(g) &= \ell(g)^* \ell(f) = 0 \label{eq:Fockrelation2e} \\
\ell(f) \mathfrak{m}(g) &= \mathfrak{m}(g) \ell(f)^* = 0 \label{eq:Fockrelation2f} 
\end{align}
\end{lemma}

\begin{proof}
Relations \eqref{eq:Fockrelation2a} through \eqref{eq:Fockrelation2d} are immediate from Observation \ref{obs:Fockrelations1} because $fg = gf = 0$.  To prove \eqref{eq:Fockrelation2e}, observe that
\begin{align*}
[\ell(f) \ell(g)]^*[\ell(f)\ell(g)] &= \ell(g)^* \ell(f)^* \ell(f) \ell(g) \\
&= \ell(g)^*[\mathfrak{m}(I_\nu(f^*f)) + I_{\nu,0}(f^*f)P_{\mathcal{A}\xi})]\ell(g) \\
&= \ell(g)^* \mathfrak{m}(I_\nu(f^*f)) \ell(g) \\
&= \mathfrak{m}(I_\nu(g^* I_\nu(f^*f) g)) + I_{\nu,0}(g^* I_\nu(f^*f) g) P_{\mathcal{A} \xi}.
\end{align*}
But $I_\nu(f^*f)(s) = \int_s^T \nu(f^*f(X,u),u)\,du$ is supported in $[0,t]$ and hence $g^* I_\nu(f^*f) g = 0$.  This shows that $[\ell(f) \ell(g)]^*[\ell(f)\ell(g)] = 0$ and hence $\ell(f) \ell(g) = 0$ and $\ell(g)^* \ell(f)^* = 0$.

The proof of \eqref{eq:Fockrelation2f} is similar.  Note that
\begin{align*}
[\ell(f) \mathfrak{m}(g)]^* [\ell(f) \mathfrak{m}(g)] &= \mathfrak{m}(g^*)[ \mathfrak{m}(I_\nu(f^*f)) + I_{\nu,0}(f^*f) P_{\mathcal{A}\xi} ] \mathfrak{m}(g) \\
&= \mathfrak{m}(g^* I_\nu(f^*f) g) = 0.
\end{align*}
Hence, $\ell(f) \mathfrak{m}(g) = 0$.  Moreover, $\mathfrak{m}(g) \ell(f)^* = 0$ follows by taking adjoints.
\end{proof}

\begin{lemma}
Let $t \in [0,T]$.  Regarding $\mathcal{A}$ as a subalgebra of $B(\mathcal{H}_\nu)$ as above, we denote by $\mathcal{A} + \mathcal{B}_{t,T}$ the algebraic sum as subspaces of $B(\mathcal{H}_\nu)$, which is equal to the unital $\mathcal{A}$-algebra generated by $\mathcal{B}_{t,T}$.  Then $\mathcal{B}_{t,T} + \mathcal{A}$ is the linear span of operators of the form
\[
\ell(f_m) \dots \ell(f_1) [\mathfrak{m}(p) + a] \ell(g_1)^* \dots \ell(g_n)^*,
\]
where $m, n \geq 0$ and $a \in \mathcal{A}$ and $f_i, g_i, p \in \mathcal{L}_{\mathcal{A};s,t}\ip{X}$.
\end{lemma}

\begin{proof}
Let $\mathcal{W}$ be the span of the operators given above.  It is clear that $\mathcal{W} \subseteq \mathcal{A} + \mathcal{B}_{s,t}$.

Moreover, $\mathcal{W}$ contains the creation, annihilation, and multiplication operators $\ell(f)$, $\ell(f)^*$, and $\mathfrak{m}(f)$ by substituting appropriate values for $m$, $n$, $a$, $p$, etc.  Thus, to show that $\mathcal{W} \supseteq \mathcal{B}_{s,t}$, it suffices to show that $\mathcal{W}$ is closed under left multiplication by operators in $\mathcal{B}_{s,t}$.  Furthermore, it suffices to show that if $w$ is one of the generating vectors $\ell(f_m) \dots \ell(f_1) [\mathfrak{m}(p) + a(P_{\mathcal{A}\xi} + \mathfrak{m}(\chi_{(0,s)}))] \ell(g_1)^* \dots \ell(g_n)^*$ of $\mathcal{W}$ and if $f \in \mathcal{L}_{\mathcal{A};s,t}$, then $\ell(f)w$, $\mathfrak{m}(f)w$, and $\ell(f)^*w$ are in $\mathcal{W}$.

{\bf Case 1:} It is immediate that $\ell(f) w$ is in $\mathcal{W}$.

{\bf Case 2:} For $\mathfrak{m}(f)w$, there are two subcases.  If $m \geq 1$, then
\[
\mathfrak{m}(f)w = \ell(f \cdot f_m) \dots \ell(f_1) \ell[\mathfrak{m}(p) + a] \ell(g_1)^* \dots \ell(g_n)^* \in \mathcal{W},
\]
while if $m = 0$, then
\begin{align*}
\mathfrak{m}(f)w &= \mathfrak{m}(f)  [\mathfrak{m}(p) + a] \ell(g_1)^* \dots \ell(g_n)^* \\
&= \mathfrak{m}(fp + fa) \ell(g_1)^* \dots \ell(g_n)^* \in \mathcal{W}.
\end{align*}

{\bf Case 3:} For $\ell(f)^*w$, there are three subcases.  If $m \geq 2$, then observe that
\[
\ell(f)^* \ell(f_m) \ell(f_{m-1}) = [I_{\nu,0}(f^*f_m) P_{\mathcal{A}\xi} + \mathfrak{m}(I_\nu(f^*f_m))] \ell(f_{m-1}) = \ell(I_\nu(f^*f_m) f_{m-1}),
\]
and $I_\nu(f^*f_m)f_{m-1}$ is supported in $[t,T]$, so that
\[
\ell(f)^* w = \ell(I_\nu(f^*f_m)f_{m-1}) \ell(f_{m-2}) \dots \ell(f_1) [\mathfrak{m}(p) + a] \ell(g_1)^* \dots \ell(g_n)^* \in \mathcal{W}.
\]
For $m = 1$, observe that $I_\nu(f_m^*f)$ is constant on $[0,t]$.  Thus, using the fact that $P_{\mathcal{A}\xi} + \mathfrak{m}(\chi_{(0,t)}) + \mathfrak{m}(\chi_{(t,T)}) = \id$, we have
\begin{align*}
\ell(f)^* \ell(f_m) &= \mathfrak{m}(I_\nu(f_m^*f)) + I_{\nu,0}(f_m^*f) P_{\mathcal{A}\xi} \\
&= \mathfrak{m}(\chi_{(t,T)} I_\nu(f_m^*f) - \chi_{(t,T)} I_{\nu,0}(f)) + I_{\nu,0}(f_m^*f) \id,
\end{align*}
The function $q = \chi_{(t,T)} I_\nu(f_m^*f) - \chi_{(t,T)} I_{\nu,0}(f)$ is in $\mathcal{L}_{\mathcal{A};t,T}\ip{X}$.  Moreover, combining this with the ``middle'' term $\mathfrak{m}(p) + a$ in $w$ yields
\[
[\mathfrak{m}(q) + I_{\nu,0}(f_m^*f)][\mathfrak{m}(p) + a] = [\mathfrak{m}(qp + I_{\nu,0}(f_m^*f)p + qa) + I_{\nu,0}(f_m^*m) a],
\]
which is another term of the same form as $\mathfrak{m}(p) + a$, from which it follows that $\ell(f)^*w$ is in $\mathcal{W}$.  Finally, in the case $m = 0$, we observe that
\begin{align*}
\ell(f)^* w &= \ell(f)^* [\mathfrak{m}(p) + a] \ell(g_1)^* \dots \ell(g_n)^* \\
&= \ell((p^* + a^*)f)^*\ell(g_1)^* \dots \ell(g_n)^* \in \mathcal{W}.
\end{align*}
\end{proof}

\begin{lemma}
Let $0 < t < T$.  If $x, y \in \mathcal{B}_{0,t}$ and $z \in \mathcal{B}_{t,T} + \mathcal{A}$, then we have
\begin{align}
xzy &= x E_\nu(z) y \label{eq:Fockrelation3a} \\
xz \xi &= x E_\nu(z) \xi. \label{eq:Fockrelation3b}
\end{align}
\end{lemma}

\begin{proof}
By linearity, it suffices to consider the case when $x$ and $y$ are strings of creation, annihilation, and multiplication operators for functions in $\mathcal{L}_{\mathcal{A};0,t}\ip{X}$.  Moreover, if $x$ and $y$ are such strings and if $x'$ is the last creation, annihilation, or multiplication operator in the string $x$, and $y'$ is the first one in the string $y$, then it suffices to show that $x' z y' = x' E_\nu(z) y'$.  Hence, we can assume without loss of generality that $x$ and $y$ are either creation, annihilation, or multiplication operators.  Furthermore, by the previous lemma and linearity, it suffices to consider the case where
\[
z = \ell(f_m) \dots \ell(f_1) [\mathfrak{m}(p) + a] \ell(g_1)^* \dots \ell(g_n)^*,
\]
with $f_j$, $g_j$, and $p$ in $\mathcal{L}_{\mathcal{A};t,T}$.

{\bf Case 1:} Suppose that $m > 0$.  Then $E_\nu(z) = \ip{\xi, z \xi} = 0$. Then since $x$ is a creation, annhilation, or multiplication operator for some function in $\mathcal{L}_{\mathcal{A};0,t}\ip{X}$, we have $x \ell(f_m) = 0$ by \eqref{eq:Fockrelation2d}, \eqref{eq:Fockrelation2e}, and \eqref{eq:Fockrelation2b} and hence $xz = 0$.  This implies that $xzy = x E_\nu(z) y$ and $xz \xi = x E_\nu(z) \xi$.

{\bf Case 2:} Suppose that $m = 0$ and $n > 0$.  Then $z \xi = 0$ and hence $E_\nu(z) = 0$.  We also have $xz \xi = 0 = x E_\nu(z) \xi$ so that \eqref{eq:Fockrelation2b} holds.  To check \eqref{eq:Fockrelation2a} in this case, note that $\ell(g_n) y = 0$ by a symmetrical argument to Case 1, and hence $xzy = 0 = x E_\nu(z) y$.

{\bf Case 3:} Suppose that $m = 0$ and $n = 0$.  Then $z = \mathfrak{m}(p) + a$ and $E_\nu(z) = \ip{\xi, z \xi} = a$.  By \eqref{eq:Fockrelation2a}, \eqref{eq:Fockrelation2c}, \eqref{eq:Fockrelation2f}, we have $x \mathfrak{m}(p) = 0$.  Hence, $xz = xa = x E_\nu(z)$ and thus $xzy = x E_\nu(z) y$ and $xz \xi = x E_\nu(z) \xi$.
\end{proof}

\begin{proof}[Proof of Proposition \ref{prop:Fockindependence}]
As in Definition \ref{def:monotoneindependence}, let $m \geq 2$ and consider a string $b_1$, \dots, $b_m$ of operators with $b_i \in \mathcal{B}_{t_{k_i-1},t_{k_i}}$.  Suppose that for some $j$, the index $k_j$ is larger than the adjacent indices.  We must show that
\begin{equation} \label{eq:checkmonotoneindependence}
E_\nu[b_1 \dots b_m] = E_\nu[b_1 \dots b_{j-1} E_\nu(b_j) b_{j+1} \dots b_m].
\end{equation}

{\bf Case 1:} Suppose that $1 < j < m$.  In this case, $k_{j-1} \leq k_j -1$ and $k_{j+1} \leq k_j - 1$, hence
\begin{align*}
b_j &\in \mathcal{B}_{t_{k_j-1},t_{k_j}} \subseteq \mathcal{B}_{t_{k_j-t},T} + \mathcal{A} \\
b_{j-1} &\in \mathcal{B}_{t_{k_{j-1}-1},t_{k_{j-1}}} \subseteq \mathcal{B}_{0,t_{k_j-1}} \\
b_{j+1} &\in \mathcal{B}_{t_{k_{j+1}-1},t_{k_{j+1}}} \subseteq \mathcal{B}_{0,t_{k_j-1}}.
\end{align*}
Therefore, applying \eqref{eq:Fockrelation3a} with $t = t_{k_j-1}$, we have $b_{j-1} b_j b_{j+1} = b_{j-1} E_\nu(b_j) b_{j+1}$ and hence \eqref{eq:checkmonotoneindependence} holds.

{\bf Case 2:} Suppose that $j = m$.  Then using \eqref{eq:Fockrelation3a} with $t = t_{k_m - 1}$, we have $b_{m-1}b_m \xi = b_{m-1} E_\nu(b_m) \xi$ and hence \eqref{eq:checkmonotoneindependence} holds.

{\bf Case 3:} If $j = 1$, then we take adjoints and apply Case 2.
\end{proof}

\section{Central Limit Theorem for Loewner Chains} \label{sec:CLT}

\subsection{Motivation}

Muraki \cite{Muraki2000} \cite{Muraki2001} showed that the central limit object for monotone convolution is the \emph{arcsine law} given by the density $1/2 \pi \sqrt{t^2 - 2}$ on the interval $[-\sqrt{2}, \sqrt{2}]$.  Its reciprocal Cauchy transform is $F(z) = \sqrt{z^2 - 2}$, where the square root is chosen to be analytic on $\C^* \setminus [-\sqrt{2}, \sqrt{2}]$ and satisfy $F(z) = z - 1/z + O(1/z^2)$.  The rescaled versions $F_t(z) = \sqrt{z^2 - 2t}$ are maps from $\h$ onto $\h$ minus a vertical slit.  These functions form a composition semigroup and they solve the Loewner equation with $V(z,t) = -1/z$.

The operator-valued version of the monotone central limit theorem was proved combinatorially in \cite[Theorem 2.5]{BPV2013}.  The limiting distribution of $\frac{1}{\sqrt{k}} \sum_{j=1}^k X_j$ as $k \to \infty$ is called the \emph{operator-valued arcsine law}, and it depends only on the $\mathcal{A}$-valued variance $\eta(a) = E(X_j a X_j)$, which is a completely positive map $\mathcal{A} \to \mathcal{A}$.  As we will verify, an equivalent definition of the operator-valued arcsine law $\mathfrak{as}(\eta)$ is that $F_{\mathfrak{as}(\eta)}(z) = F(z,1)$, where $F(z,t)$ is the solution of the Loewner equation $\partial_t F_\eta(z,t) = DF_\eta(z,t)[-\eta(z^{-1})]$ for $t \in [0,1]$.

In terms of $F$-transforms, the central limit theorem states that
\begin{equation}
\frac{1}{\sqrt{N}} (F_\mu)^{\circ N}(\sqrt{N} z) \to F_{\mathfrak{as}(\eta)}(z) \text{ as } N \to \infty,
\end{equation}
where the superscript $\circ k$ denotes composition $k$ times.  The continuous-time analogue of this statement is that if $\mu_t$ is a monotone convolution semigroup such that $\mu_t$ has variance $t \eta$, and if $F_t = F_{\mu_t}$ is the corresponding Loewner chain, then
\begin{equation}
t^{-1/2} F_{\mu_t}(t^{1/2} z) \to F_{\mathfrak{as}(\eta)}(z) \text{ as } t \to \infty.
\end{equation}
In general, if $F_t$ is a Lipschitz normalized Loewner chain on $[0,T]$, we will show that
\begin{equation}
t^{-1/2} F_t(t^{1/2}z) - t^{1/2} F_{\mathfrak{as}(\eta)}(t^{1/2} z) = O(\rad(\nu) t^{-1/2}),
\end{equation}
where $\nu$ is the distributional family of generalized laws which generates the Loewner chain, $\eta = \nu|_{\mathcal{A} \times L^1[0,T]}$, and $\mathfrak{as}(\eta)$ is a generalized arcsine law defined in \S \ref{subsec:arcsine}.  Here the error estimate holds uniformly for $\im z \geq \epsilon$.

This result amounts, roughly speaking, to a CLT for a continuous-time family of random variables that are \emph{not} identically distributed and may not even have the same variance.  We give two versions, one using coupling (Theorem \ref{thm:CLTcoupling}) and one using the Loewner equation (Theorem \ref{thm:CLT}).

\begin{remark}
In the study of Schramm-Loewner evolution (see \cite{Lawler2005} for background), the law $\nu$ is given by a delta mass on $\R$ which is moved in time according to Brownian motion, and one has with high probability that $\rad(\nu|_{[0,t]}) = O(t^{1/2})$.  In this regime, the ``error'' estimate in our central limit theorem no longer goes to zero (it is $O(1)$).  Moreover, $t^{-1/2} F_t(t^{1/2} z)$ will not converge to the $F$-transform of the arcsine law because the distribution of SLE is invariant under this rescaling.  The results of this section are not motivated by SLE but rather by the situation where $\rad(\nu|_{[0,t]}) = O(1)$, such as composition semigroups.
\end{remark}

\subsection{Generalized Arcsine Laws} \label{subsec:arcsine}

We call $\eta: \mathcal{A} \times L^1[0,T] \to \mathcal{A}$ a \emph{distributional family of completely positive maps} if $\int \eta(\cdot,t)\phi(t)\,dt$ is a completely positive map $\mathcal{A} \to \mathcal{A}$ for each nonnegative $\phi \in L^1[0,T]$.

If $\eta: \mathcal{A} \times L^1[0,T] \to \mathcal{A}$ is a distributional family of completely positive maps, then we can define a distributional family of generalized laws by $\nu(f(X),\cdot) = \eta(f(0),\cdot)$ for $f \in \mathcal{A}\ip{X}$.  In particular, $V(z,\cdot) = -\eta(z^{-1},\cdot)$ is a distributional Herglotz vector field.  Thus, by Theorem \ref{thm:Loewnerintegration}, there exists a Loewner chain $F(z,t)$ satisfying
\begin{equation}
\partial_t F(z,t) = DF(z,t)[-\eta(z^{-1},t)].
\end{equation}
Then $F(z,T)$ is the reciprocal Cauchy transform of a law $\mathfrak{as}(\eta)$ with variance $\mathfrak{as}(\eta)(X z X) = \int_0^T \eta(z,t)\,dt$.  We call $\mathfrak{as}(\eta)$ the \emph{generalized arcsine law} corresponding to $\eta$.

\begin{remark}
We caution that $\mathfrak{as}(\eta)$ is not uniquely determined by the variance $\int_0^T \eta(\cdot,t)\,dt$, but it depends a priori on the behavior of $\eta$ on the entire interval $[0,T]$.  We do not yet know how uniquely $\eta$ is determined by $\mathfrak{as}(\eta)$.
\end{remark}

The generalized arcsine laws form a stable family under monotone convolution in the following sense:  If $\eta_1$ and $\eta_2$ are distributional families of completely positive maps on $[0,T_1]$ and $[0,T_2]$, then $\mathfrak{as}(\eta_1) \rhd \mathfrak{as}(\eta_2) = \mathfrak{as}(\eta)$, where $\eta$ is defined on $[0,T_1 + T_2]$ by concatenating $\eta_1$ and $\eta_2$.  This follows from the construction of solutions to the Loewner equation by solving the ODE (as in Step 2 of the proof of Theorem \ref{thm:Loewnerintegration}).

In the case of a generalized arcsine law, the combinatorial formulas of Theorem \ref{thm:combinatorics} simplify as follows.  Because in this case $\nu(f(X),\cdot) = \eta(f(0),\cdot)$, the coefficients $Q_{\pi;s,t}$ defined in Lemma \ref{lem:defineQ} will vanish if $\pi$ has any blocks of size $>2$.  Let $NC_2(k)$ be the set of non-crossing pair partitions of $[k]$ (partitions in which every block has exactly two elements).

\begin{corollary} \label{cor:arcsinemoments}
Let $\eta: \mathcal{A} \times L^1[0,T] \to \mathcal{A}$ be a distributional family of completely positive maps.  For $\pi \in NC_2(k)$ and $0 \leq s \leq t \leq T$, we define $Q_{\pi;s,t}$ by
\begin{enumerate}
	\item For $\pi_1 \in NC_2(k_1)$ and $\pi_2 \in NC_2(k_2)$,
	\[
	Q_{\pi_1\pi_2;s,t}(a_1,\dots,a_{k-1}) = Q_{\pi_1}(a_1,\dots,a_{k_1-1}) a_{k_1} Q_{\pi_2}(a_{k_1+1},\dots,a_{k_1+k_2-1});
	\]
	\item For $\pi \in NC_2(k)$, we have
	\[
	Q_{\Theta_1(\pi|);s,t}(a_1,\dots,a_{k+1}) = \int_s^t \eta(a_1 Q_{\pi;s,t}(a_2,\dots,a_k) a_{k+1}, u)\,du.
	\]
\end{enumerate}
Then we have
\[
\mathfrak{as}(\eta)(a_0 X a_1 \dots X a_k) = \sum_{\pi \in NC_2(k)} a_1 Q_{\pi;0,T}(a_1,\dots,a_{k-1}) a_k
\]
\end{corollary}

If we further restrict to the case where $T = 1$ and is independent of $t$ (that is $\int \eta(f(X),t)\phi(t)\,dt = \int \eta_0(f(X)) \phi(t)\,dt$ for some completely positive $\eta_0: \mathcal{A} \to \mathcal{A}$), then we obtain the formulas for the operator-valued arcsine law from \cite[Theorem 2.5]{BPV2013}.  This verifies that our definition of the operator-valued arcsine law coincides with theirs.

\subsection{Central Limit Theorem via Coupling} \label{subsec:CLTcoupling}

Let $\nu: \mathcal{A}\ip{X} \times L^1[0,T] \to \mathcal{A}$ be a distributional family of $\mathcal{A}$-valued generalized laws on $[0,T]$, and let $\eta$ be the distributional family of completely positive maps given by $\eta = \nu|_{\mathcal{A} \times L^1[0,T]}$.  Let $F_{\mu_t}$ be the solution to the Loewner equation on $[0,T]$ for the Herglotz vector field $V$ corresponding to $\nu$, and let $F_{\mu_{s,t}}$ be the corresponding family of subordination maps.  Let $\mathfrak{as}(\eta|_{[s,t]})$ be the generalized arcsine law given by $\eta|_{[s,t]}$ translated to the interval $[0,t-s]$.  Our goal is to estimate the difference between $\mu_{s,t}$ and $\mathfrak{as}(\eta|_{[s,t]})$.

Let $\mathcal{H}_\nu$ be the Fock space constructed in \S \ref{subsec:Fockspace1} with the corresponding expectation $E_\nu: B(\mathcal{H}_\nu) \to \mathcal{A}$.  Let
\begin{align}
Y_{s,t} &= \ell(\chi_{(s,t)}) + \ell(\chi_{(s,t)})^* + \mathfrak{m}(\chi_{(s,t)} X) \in B(\mathcal{H}_\nu) \\
Z_{s,t} &= \ell(\chi_{(s,t)}) + \ell(\chi_{(s,t)})^* \in B(\mathcal{H}_\nu).
\end{align}
By Theorem \ref{thm:Fockspace}, $Y_{s,t}$ has the law $\mu_{s,t}$.  On the other hand, upon inspecting the proof of Theorem \ref{thm:Fockspace} (2), we see that $Z_{s,t}$ has the law obtained by replacing $\mathfrak{m}(\chi_{(s,t)}X)$ by zero or by discarding all the terms indexed by partitions in $NC_{\geq 2} \setminus NC_2$.  But this is equivalent to replacing $\nu$ by $\eta$.  Hence, the law of $Z_{s,t}$ is $\mathfrak{as}(\eta|_{[s,t]})$.  We also have $\norm{Y_{s,t} - Z_{s,t}} = \norm{\mathfrak{m}(X)} \leq \rad(\nu)$, which leads to the following result.

\begin{theorem} \label{thm:CLTcoupling}
Let $\nu$ be a distributional family of generalized laws on $[0,T]$ and let $F_{\mu_{s,t}}$ be the laws associated to the Loewner chain generated by $\nu$.  Let $\eta = \nu|_{\mathcal{A} \times L^1[0,T]}$.  There exists an $\mathcal{A}$-valued probability space $(\mathcal{B},E)$ and self-adjoint random variables $Y_{s,t}$ and $Z_{s,t}$ such that
\begin{enumerate}[(1)]
	\item $Y_{t_1,t_2} + Y_{t_2,t_3} = Y_{t_1,t_3}$ and $Z_{t_1,t_2} + Z_{t_2,t_3} = Z_{t_1,t_3}$.
	\item $Y_{s,t} \sim \mu_{s,t}$ and $Z_{s,t} \sim \mathfrak{as}(\eta|_{[s,t]})$.
	\item $\norm{Y_{s,t}} \leq \sqrt{2C(t - s)} + \rad(\nu)$ and $\norm{Z_{s,t}} \leq \sqrt{2C(t - s)}$, where $C = \norm{\nu(1,\cdot)}_{\mathcal{L}(L^1[0,T],\mathcal{A})}$.
	\item $\norm*{Y_{s,t} - Z_{s,t}} \leq \rad(\nu)$.
	\item Given $0 = t_0 < t_1 < \dots < t_N = T$, the non-unital $\mathcal{A}$ algebras $\mathcal{A}\ip{Y_{t_{j-1},t_j}, Z_{t_{j-1},t_j}}$ are monotone independent.
\end{enumerate}
\end{theorem}

\begin{proof}
We have already proved most of the theorem using the operators constructed on the Fock space above.  The claim that $\norm{Z_{s,t}} \leq \sqrt{C(t - s)}$ does not follow immediately from prior results, but it can be deduced by the same proof as Theorem \ref{thm:Fockspace} (4) in \S \ref{subsec:Fockspacenorm}.  The claim (5) about monotone independence follows from Proposition \ref{prop:Fockindependence} since $Y_{s,t}$ and $Z_{s,t}$ are contained in $\mathcal{B}_{s,t}$.
\end{proof}

Of course, if we rescale by $(t - s)^{1/2}$, we have the central-limit-type estimate
\begin{equation}
\norm*{(t - s)^{-1/2}Y_{s,t} - (t - s)^{-1/2} Z_{s,t}} \leq (t - s)^{-1/2} \rad(\nu).
\end{equation}
In particular, in the case of a monotone convolution semigroup, we have the following.

\begin{corollary}
Let $\mu_t$ be a $\mathcal{A}$-valued monotone convolution semigroup with mean zero, and let $\eta(a) = \mu_1(XaX)$.  Then there exist random variables $Y_t$ and $Z_t$ such that $Y_t \sim \mu_t$ and $Z_t \sim \mathfrak{as}(\eta|_{[0,t]})$ and
\begin{equation}
\norm{t^{-1/2} Y_t - t^{-1/2} Z_t} \leq 2 t^{1/2} \inf_{s > 0} \rad(\mu_s).
\end{equation}
Note in this case that the law of $t^{-1/2} Z_t$ is independent of $t$ because of the scale-invariance of the arcsine law.
\end{corollary}

\begin{proof}
By Proposition \ref{prop:monotonesemigroup}, there exists a generalized law $\nu$ such that $\partial_tF_t(z) = DF_t(z)[-G_\nu(z)]$.  In light of Theorem \ref{thm:CLTcoupling}, it suffices to show that $\rad(\nu) \leq 2 \rad(\mu_1)$.  By Proposition \ref{prop:Ftransformbijection}, we have $F_{\mu_t}(z) = z - G_{\sigma_t}(z)$ for a generalized law $\sigma_t$ with $\rad(\sigma_t) \leq \rad(\sigma_s)$.  Now if $s \leq t$, then $F_{\mu_t} = F_{\mu_{t-s}} \circ F_{\mu_s}$, hence $\im F_{\mu_t} \geq \im F_{\mu_s}$.  It follows that $G_{\sigma_t} - G_{\sigma_s}$ is the Cauchy transform of some generalized law, and hence by Lemma \ref{lem:maxradius}, $\rad(\sigma_s) \leq \rad(\sigma_t)$.  Therefore, $s \leq t$ implies that $\rad(\sigma_s) \leq 2 \rad(\mu_t)$.  By Theorem \ref{thm:Loewnerdifferentiation}, we have
\[
\rad(\nu) \leq \limsup_{s \searrow 0} \rad(\sigma_s) \leq \inf_{s > 0} \rad(\mu_s).  \qedhere
\]
\end{proof}

\begin{corollary} \label{cor:CLTestimates1}
With the setup of Theorem \ref{thm:CLTcoupling}, let $\mu_t = \mu_{0,t}$.  Then we have for $z \in \h_\epsilon^{(n)}(\mathcal{A})$ and $0 < t \leq T$ that
\begin{equation}
\norm*{t^{1/2} G_{\mu_t}^{(n)}(t^{1/2} z) - t^{1/2} G_{\mathfrak{as}(\eta|_{[0,t]})}^{(n)}(t^{1/2}z)} \leq \frac{\rad(\nu)}{\epsilon^2 t^{1/2}}.
\end{equation}
The same holds when $\norm{z^{-1}} \leq (\sqrt{2C} + t^{-1/2} \rad(\nu) + \epsilon)^{-1}$.
\end{corollary}

\begin{proof}
We know that $\norm{Y_{0,t}} \leq M + 2\sqrt{Ct}$ and $\norm{Z_{0,t}} \leq \sqrt{2Ct}$ by Theorem \ref{thm:CLTcoupling}  Therefore, if $\im z \geq \epsilon$ or $\norm{z^{-1}} \geq (M + 2\sqrt{Ct} + \epsilon)^{-1}$, then we have
\begin{align}
\norm*{E_\nu^{(n)}[(z - Y_{0,t}^{(n)})^{-1} - (z - Z_{0,t}^{(n)})^{-1}]} &= \norm*{E_\nu^{(n)}[(z - Y_{0,t}^{(n)})^{-1}\mathfrak{m}(\chi_{(0,t)}X)^{(n)}(z - Z_{0,t}^{(n)})^{-1}]} \nonumber \\
&\leq \epsilon^{-2} \norm{\mathfrak{m}(\chi_{(0,t)}X)} \leq \epsilon^{-2} \rad(\nu).
\end{align}
The asserted estimate follows after renormalization of $z$ and $\epsilon$.
\end{proof}

\subsection{Central Limit Theorem via the Loewner Equation}

The estimates of Corollary \ref{cor:CLTestimates1} derived from coupling depend on $\rad(\nu)$.  Now we will derive another estimate which only depends on second moments of $\nu$ rather than the support radius.  For technical reasons, we have to assume that all our random variables are bounded because the theory of Cauchy transforms relies on that.  But the estimates below do not reference the operator norm and hence could be used for unbounded laws if the analytic theory of Cauchy transforms and Loewner chains were to be extended to that setting.

\begin{theorem} \label{thm:CLT}
Let $\nu$ be a distributional family of generalized laws.  Let $F_{\mu_t}(z)$ for $t \in [0,T]$ be the Loewner chain corresponding to the distributional Herglotz vector field $V$ given by $\nu$.  Let $\eta = \nu|_{\mathcal{A} \times L^1[0,T]}$, and let
\begin{align} \label{eq:CLTconstants}
C_1 &:= \norm{\eta(1,\cdot)}_{\mathcal{L}(L^1[0,T], \mathcal{A})}, \\
C_2 &:= \sup_{n \geq 1} \sup_{\substack{z \in M_n(\mathcal{A}) \\ \norm{z} \leq 1}} \norm*{\nu^{(n)}\left(z^*X^2z\right)}_{\mathcal{L}(L^1[0,T], M_n(\mathcal{A}))} \leq C_1 \rad(\nu)^2.
\end{align}
Then for $z \in \h_\epsilon^{(n)}(\mathcal{A})$, we have
\begin{equation} \label{eq:CLT1}
\norm*{t^{-1/2} F_{\mu_t}^{(n)}(t^{1/2} z) - t^{-1/2} F_{\mathfrak{as}(\eta|_{[0,t]})}^{(n)}(t^{1/2} z)} \leq t^{-1/2} \left( 1 + \frac{C_1}{2\epsilon^2} \right) \frac{(C_1 C_2)^{1/2}}{\epsilon^2}.
\end{equation}
Moreover,
\begin{equation} \label{eq:CLT2}
\norm*{t^{1/2} G_{\mu_t}^{(n)}(t^{1/2}z) - t^{1/2} G_{\mathfrak{as}(\eta|_{[0,t]})}^{(n)}(t^{1/2}z)} \leq t^{-1/2} \frac{(C_1 C_2)^{1/2}}{\epsilon^4}.
\end{equation}
\end{theorem}

\begin{remark}
Theorems \ref{thm:CLTcoupling} and \ref{thm:CLT} provide complementary information about $G_\nu(z,t) - G_\eta(z,t)$.  The estimates from Theorem \ref{thm:CLTcoupling} are better when $\epsilon$ is small, but the estimates from Theorem \ref{thm:CLT} are better when $\epsilon$ is large.
\end{remark}

\begin{lemma} \label{lem:CLTestimate}
With the setup of Theorem \ref{thm:CLT}, we have
\begin{equation}
z \in \h_\epsilon^{(n)}(\mathcal{A}) \implies \norm{V(z,\cdot) + \eta(z^{-1},\cdot)}_{\mathcal{L}(L^1[0,T], M_n(\mathcal{A}))} \leq \frac{(C_1C_2)^{1/2}}{\epsilon^2}.
\end{equation}
\end{lemma}

\begin{proof}
Fix $\phi \geq 0$ in $L^1[0,T]$.  By Proposition \ref{prop:CPmap}, the generalized law $\sigma = \int \nu(\cdot,t) \phi(t)\,dt$ can be expressed as $\widehat{\sigma} \circ \pi$, where $\pi$ is a $*$-homomorphism from $\mathcal{A}{X}$ into a $C^*$-algebra $\mathcal{B}$ and where $\widehat{\sigma}: \mathcal{B} \to \mathcal{A}$ is a completely positive map.  Then we have
\begin{align}
- \int V^{(n)}(z,t)\phi(t)\,dt - \int \eta^{(n)}(z^{-1},t)\phi(t)\,dt &= \widehat{\sigma}^{(n)}((z - \pi(X)^{(n)})^{-1} - z^{-1}) \\
&= \widehat{\sigma^{(n)}}((z - \pi(X)^{(n)})^{-1} \pi(X)^{(n)} z^{-1}).
\end{align}
Note $\widehat{\sigma}^{(n)}$ is completely positive, so as in Proposition \ref{prop:CPmap} we can define a right Hilbert $M_n(\mathcal{A})$-module $\mathcal{H} = M_n(\mathcal{B}) \otimes_{\widehat{\sigma}^{(n)}} M_n(\mathcal{A})$.  By applying the CBS inequality (Lemma \ref{lem:rightHilbertmodule},
\[
\norm{\widehat{\sigma^{(n)}}((z - \pi(X)^{(n)})^{-1} \pi(X)^{(n)} z^{-1})} \leq \norm{1 \otimes 1}_{\mathcal{H}} \norm{(z - \pi(X)^{(n)})^{-1}}_{B(\mathcal{H})} \norm{\pi(X)^{(n)} z^{-1} \otimes 1}_{\mathcal{H}}.
\]
But note that $\norm{(z - \pi(X))^{-1}} \leq 1 / \epsilon$ and
\[
\norm{1 \otimes 1}_{\mathcal{H}}^2 = \norm{\widehat{\sigma}^{(n)}(1)} \leq C_1 \norm{\phi}_{L^1[0,T]}.
\]
and
\[
\norm{\pi(X)^{(n)} z^{-1} \otimes 1}_{\mathcal{H}}^2 = \norm{\widehat{\sigma}^{(n)}((z^{-1})^* X^2 z^{-1})} \leq C_2 \norm{z^{-1}}^2 \leq \frac{C_2}{\epsilon}.
\]
Combining these inequalities shows that
\begin{equation}
\norm*{\int [V^{(n)}(z,t) + \eta(z^{-1},t)] \phi(t)\,dt} \leq \frac{(C_1 C_2)^{1/2}}{\epsilon^2} \norm{\phi}_{L^1[0,T]}.
\end{equation}
This holds for all $\phi \geq 0$, so by Lemma \ref{lem:smallerror}, it holds for all $\phi \in L^1[0,T]$.  This completes the proof.
\end{proof}

\begin{proof}[Proof of Theorem \ref{thm:CLT}]
Let $\alpha_t = \mathfrak{as}(\eta|_{[0,t]})$ and $\alpha_{s,t} = \mathfrak{as}(\eta|_{[s,t]})$.  Note that for $0 \leq s \leq t \leq T$, $F_{\alpha_{s,t}}$ is the subordination map for the Loewner chain generated by the Herglotz vector field $-\eta(z^{-1},t)$.

We will estimate $F_{\mu_t}^{(n)}(z) - F_{\alpha_t}^{(n)}(z)$ by observing that
\begin{equation} \label{eq:Lindebergthing}
F_{\mu_t}^{(n)}(z) - F_{\alpha_t}^{(n)}(z) = \int_0^t \partial_s \left( F_{\mu_s}^{(n)} \circ F_{\alpha_{s,t}}(z) \right)\,ds,
\end{equation}
By the chain rule (Lemma \ref{lem:chainrule}),
\[
\partial_s \left( F_{\mu_s}^{(n)} \circ F_{\alpha_{s,t}}^{(n)} \right) = \partial_s F_{\mu_s}^{(n)} \circ F_{\alpha_{s,t}}^{(n)} + (DF_{\mu_s}^{(n)} \circ F_{\alpha_{s,t}})^{(n)}[\partial_s F_{\alpha_{s,t}}^{(n)}].
\]
From the Loewner equation,
\[
\partial_s F_{\mu_s}^{(n)} \circ F_{\alpha_{s,t}}^{(n)} = (DF_{\mu_s}^{(n)} \circ F_{\alpha_{s,t}}^{(n)})[V^{(n)}(F_{\alpha_{s,t}}^{(n)},s)].
\]
Moreover, applying \eqref{eq:FSTdiffeq} with the Herglotz vector field $\tilde{\eta}(z,t) = \eta(z^{-1},t)$, we have
\[
\partial_s F_{\alpha_{s,t}}^{(n)} = \tilde{\eta}^{(n)}(F_{\alpha_{s,t}}^{(n)},s)
\]
Therefore,
\[
\partial_s \left( F_{\mu_s}^{(n)} \circ F_{\alpha_{s,t}}^{(n)} \right) = \left(DF_{\mu_s}^{(n)} \circ F_{\alpha_{s,t}}^{(n)}\right)[(V^{(n)} + \tilde{\eta}^{(n)})(F_{\alpha_{s,t}}^{(n)}, s)].
\]
Assuming that $z \in \h_\epsilon^{(n)}(\mathcal{A})$, we have from Lemma \ref{lem:CLTestimate} that
\[
\norm{(V^{(n)} + \tilde{\eta}^{(n)})(F_{\alpha_{s,t}}^{(n)},\cdot)}_{\mathcal{L}(L^1[0,T], M_n(\mathcal{A}))} \leq \frac{(C_1 C_2)^{1/2}}{\epsilon^2}.
\]
Moreover, by \eqref{eq:LoewnerlocallyLipschitz1}, we have
\begin{equation}
\norm{DF_{\mu_s}^{(n)} \circ F_{\alpha_{s,t}}^{(n)}} \leq 1 + \frac{C_1s}{\epsilon^2}.
\end{equation}
Then by \eqref{eq:stupidestimate}, we have
\begin{align}
\norm*{F_{\mu_t}^{(n)}(z) - F_{\alpha_t}^{(n)}(z)} &\leq \int_0^t \left( 1 + \frac{C_1s}{\epsilon^2} \right) \frac{(C_1 C_2)^{1/2}}{\epsilon^2} \nonumber \\
&= \left( 1 + \frac{C_1t}{2\epsilon^2} \right) \frac{(C_1 C_2)^{1/2} t}{\epsilon^2}.
\end{align}
The estimate \eqref{eq:CLT1} now follows upon renormalization of $z$ and $\epsilon$.

The estimate \eqref{eq:CLT2} follows by similar reasoning.  Indeed
\begin{equation}
G_{\mu_t}^{(n)}(z) - G_{\alpha_t}^{(n)}(z) = \int_0^t \partial_s \left( G_{\mu_s}^{(n)} \circ F_{\alpha_{s,t}}^{(n)}(z) \right)\,ds.
\end{equation}
Because $G_{\mu_t}^{(n)}(z)$ is an analytic function applied to $F_{\mu_t}(z)$, a chain rule computation shows that
\begin{equation}
\partial_t [G_{\mu_t}^{(n)}(z)] = DG_{\mu_t}^{(n)}(z)[V^{(n)}(z,t)].
\end{equation}
Therefore, as before,
\begin{equation}
\left( G_{\mu_s}^{(n)} \circ F_{\alpha_{s,t}}^{(n)}(z) \right) = \left(DG_{\mu_s}^{(n)} \circ F_{\alpha_{s,t}}^{(n)} \right)[(V^{(n)} + \tilde{\eta}^{(n)})(F_{\alpha_{s,t}}^{(n)},s)].
\end{equation}
We estimate $DG_{\mu_s}^{(n)} \circ F_{\alpha_{s,t}}^{(n)}$ by $1/\epsilon^2$, and we estimate $(V^{(n)} + \tilde{\tilde{\eta}}^{(n)})(F_{\alpha_{s,t}}^{(n)},\cdot)$ by $(C_1C_2)^{1/2} / \epsilon^2$.  This proves \eqref{eq:CLT2} after renormalization of $z$ and $\epsilon$.
\end{proof}

\appendix

\section{The Need for Distributional Differentiation} \label{sec:needdistribution}

For a better understanding of our results, we will now explain heuristically why pointwise differentiation in the Loewner equation is not possible in our setting.  Of course, if $F(z,t)$ were a $C^1$ function of $t$ for each $z$, there would be little difficulty.  But in order to get a general correspondence between Loewner chains and vector fields $V(z,t)$, we must get by with only assuming that $F(z,t)$ is Lipschitz in $t$.

There are known results about differentiating an absolutely continuous function from a time interval $[0,T]$ into a Banach space $\mathcal{X}$ (for instance, \cite[Appendix]{Komura1967}).  However, these theorems usually rely on separability or reflexivity of $\mathcal{X}$, which is something we cannot assume in an operator algebras setting.  Indeed, infinite-dimensional $C^*$-algebras are never reflexive, and furthermore, infinite-dimensional von Neumann algebras are never separable in the norm topology.

Pointwise differentiation will certainly not be possible in the norm topology.  If $\mathcal{A}$ is a von Neumann algebra acting on a separable Hilbert space, then differentiation in the strong operator topology (SOT) may be possible (thanks to the theory of differentiation of Hilbert-valued functions).  However, in order to use the chain rule for SOT differentiation, we would have to make the additional assumption that the Frechet derivatives of the maps we are composing are SOT-continuous, which means making additional SOT continuity assumptions about the laws $\mu_t$.

Furthermore, suppose that we \emph{can} for a fixed $z$, differentiate $F(z,t)$ almost everywhere with respect to $t$; then it would still be problematic to carry out such differentiation with the same exceptional null set of times for all values of $z$ ranging over an open set in a non-separable Banach space.  One might try to solve this problem by assuming that $\mathcal{A}$ is separable in SOT and that our analytic functions are continuous in SOT.  However, even this is not sufficient because we cannot enforce SOT equicontinuity of $(F(z,t+\delta) - F(z,t))/\delta$ as $\delta \to 0$.

A possible solution would be to assume that $\mathcal{A}$ is tracial von Neumann algebra and that for each function $F(z) = z - G_\sigma(z)$ that we are dealing with, the state $\tau \circ \sigma$ is tracial on $\mathcal{A}\ip{X}$.  The problems with the SOT approach sketched above would be solved by using explicit estimates in $L^2$ norm to guarantee SOT-equicontinuity of $(F(z,t+\delta) - F(z,t))/\delta$ for different values of $\delta$, as well as SOT equicontinuity of $a \mapsto DF_t(z)[a]$ for different values of $t$.

However, traciality of $\sigma$ seems like an artificial and restrictive condition.  If $F_\mu(z) = z - G_\sigma(z)$, it is unclear (at least to the author) whether traciality of $\sigma$ and traciality of $\mu$ are related.  It also seems doubtful that functions of the form $z - G_\sigma(z)$ for $\sigma$ tracial satisfy the two-out-of-three property.  Even so, the tracial setting would be the best place to start developing the theory for laws $\mu_t$ with unbounded support, although that was not the goal of this paper.

\section{Order $1$ Loewner Chains} \label{sec:order1}

In this section, we show that for a normalized Loewner chain $F_t = F_{\mu_t}$, the condition that $\mu_t(X^2)$ is continuous is equivalent to the Loewner chain being of ``order $1$'' on $\h(\mathcal{A})$.  As in \cite[Definition 1.2]{CDG2009}, we will say that a Loewner chain is \emph{of order $d$} if for each point $z \in \h(\mathcal{A})$, there exists a some nonnegative $\phi_z \in L^d[0,T]$ such that
\begin{equation}
\norm{F(z,t) - F(z,s)} \leq \int_s^t \phi_z \text{ for } s < t.
\end{equation}

\begin{proposition}
Let $F_t$ be an $\mathcal{A}$-valued Loewner chain such that $F_t$ is the $F$-transform of an $\mathcal{A}$-valued law $\mu_t$ with mean zero.  Then the following are equivalent:
\begin{enumerate}[(1)]
	\item $\mu_t(X^2)$ is an absolutely continuous map $[0,T] \to \mathcal{A}$.
	\item $F_t$ is of order $1$.
	\item There exists some $n$ and some $z \in \h_\epsilon^{(n)}(\mathcal{A})$ such that $t \mapsto F_t(z)$ is absolutely continuous.
\end{enumerate}
\end{proposition}

\begin{proof}
(1) $\implies$ (2).  For $0 \leq s \leq t \leq T$, let $F_{s,t}$ be the subordination map associated to the Loewner chain.  Recall that $F_{s,t}(z) = z - G_{\sigma_{s,t}}(z)$ for a generalized law $\sigma_{s,t}$.  Proceeding similarly to Lemma \ref{lem:LoewnerlocallyLipschitz}, we have that for $z \in \h_\epsilon^{(n)}(\mathcal{A})$ that
\begin{equation}
\norm{F_{s,t}^{(n)}(z) - z} \leq \frac{1}{\epsilon} \norm{\mu_t(X^2) - \mu_s(X^2)}
\end{equation}
Because $F_s^{(n)}(z) = z - G_{\sigma_s}^{(n)}(z)$ is $(1 + \norm{\sigma_s(1)}/\epsilon^2)$-Lipschitz on $\im z \geq \epsilon$, and because $F_{s,t}$ maps $\{\im z \geq \epsilon\}$ into itself, we have
\begin{align}
\norm{F_t^{(n)}(z) - F_s^{(n)}(z)} &= \norm{F_s^{(n)} \circ F_{s,t}^{(n)}(z) - F_s^{(n)}(z)} \nonumber \\
&\leq \left( 1 + \frac{\norm{\sigma_T(1)}}{\epsilon^2} \right) \frac{1}{\epsilon} \norm{\mu_t(X^2) - \mu_s(X^2)}
\end{align}
Therefore, absolute continuity of $t \mapsto \mu_t(X^2)$ implies absolute continuity of $t \mapsto F_t(z)$ for every $z$ and hence implies that $F_t(z)$ is of order $1$.

(2) $\implies$ (3) is trivial.  For (3) $\implies$ (1), observe that
\begin{equation}
\im F_{s,t}^{(n)}(z) = \im z + \sigma_{s,t}^{(n)}[(z^* - X^{(n)})^{-1}(\im z)(z - X^{(n)})^{-1}].
\end{equation}
Recall that $M := \sup \rad(\sigma_{s,t}) < +\infty$ by Lemma \ref{lem:radiusbound}.  Thus, if we realize the law $\sigma_{s,t}$ by an operator $x = \pi(X)$ on a Hilbert space as in Proposition \ref{prop:CPmap}, then $\norm{z + x} \leq \norm{z} + M$.  This implies that if $\im z \geq \epsilon$, then
\begin{equation}
(z^* - x^{(n)})^{-1}(\im z)(z - x^{(n)})^{-1} \geq (\norm{z} + M)^{-2} \epsilon.
\end{equation}
Therefore,
\begin{align}
\norm{\im F_{s,t}^{(n)}(z) - \im z} &\geq (\norm{z} + M)^{-2} \epsilon \norm{\sigma_{s,t}(1)} \nonumber \\
&= (\norm{z} + M)^{-2} \norm{\mu_t(X^2) - \mu_s(X^2)}.
\end{align}
If $\delta > 0$ is given by Proposition \ref{prop:biholomorphic}, then
\begin{align}
\norm*{F_t^{(n)}(z) - F_s^{(n)}(z)} &\geq \delta \norm*{F_{s,t}^{(n)}(z) - z} \nonumber \\
&\geq \delta \norm*{\im F_{s,t}^{(n)}(z) - \im z} \nonumber \\
&\geq \delta \epsilon (\norm{z} + M)^{-2} \norm{\mu_t(X^2) - \mu_s(X^2)}.
\end{align}
Therefore, absolute continuity of $t \mapsto F_t^{(n)}(z)$ for one values of $z \in \h_\epsilon^{(n)}(\mathcal{A})$ implies absolute continuity of $t \mapsto \mu_t(X^2)$.
\end{proof}

\section{Radii Estimates for Lipschitz Loewner Chains}

Here we give a precise statement which allows us to compare the radii of the various generalized laws associated to a Loewner chain $F_t$.

\begin{proposition} \label{cor:Loewnerchainradius}
Consider a Loewner chain $F_t$.  Let $\mu_{s,t}$, $\sigma_{s,t}$, and $V$ be as above.  Then the following are equivalent:
\begin{enumerate}[(1)]
	\item For each $s < t$, we have $\rad(\mu_{s,t}) \leq M + \sqrt{2C(t - s)}$.
	\item For each $s < t$, we have $\rad(\sigma_{s,t}) \leq M + \sqrt{2C(t - s)}$.
	\item We have $\rad(V) \leq M$.
\end{enumerate}
\end{proposition}

\begin{proof}
(2) $\implies$ (3) follows from Theorem \ref{thm:Loewnerdifferentiation} and the remark afterwards.  On the other hand, (3) $\implies$ (1) and (3) $\implies$ (2) were shown in Theorem \ref{thm:Loewnerintegration}.

Thus, it suffices to show that (1) $\implies$ (3).  Assume that (1) holds and fix $\epsilon > 0$.  Write $G_{s,t}(z) = F_{s,t}(z)^{-1}$.  We know that for $t - s$ sufficiently small, $G_{s,t}(z^{-1})$ defined on $\norm{z} < (M + \epsilon)^{-1}$.  From the power series expansion of the Cauchy transform, $z^{-1} G_{s,t}^{(n)}(z^{-1})$ is analytic on $\norm{z} < (M + \epsilon)^{-1}$.  Using the argument of Lemma \ref{lem:analyticcontinuationofconvergence}, we have $z^{-1}G_{s,t}^{(n)}(z^{-1}) \to 1$ uniformly on $\norm{z} < (M + 2\epsilon)^{-1}$ as $t - s \to 0$.  In particular, $z^{-1} G_{s,t}(z^{-1})$ is invertible and hence $F_{s,t}^{(n)}(z^{-1})z$ is fully matricial and uniformly bounded on $\norm{z} < (M + 2 \epsilon)^{-1}$.  Note that for $\norm{z}$ small,
\begin{equation}
F_{s,t}^{(n)}(z^{-1})z = 1 - G_{\sigma_{s,t}}^{(n)}(z^{-1})z = \sum_{k=0}^\infty \sigma_{s,t}^{(n)}(z(X^{(n)}z)^k)z.
\end{equation}
If we fix $z$ with $\norm{z} \leq 1$ and let $\zeta \in \C$ with $|\zeta| < (M + 2\epsilon)^{-1}$, then we have for $k \geq 1$,
\begin{equation}
\norm{\sigma_{s,t}^{(n)}(z(X^{(n)}z)^k)z} = \norm*{ \frac{1}{(k+2)!} \frac{d^{k+2}}{d\zeta^{k+2}}\Bigr|_{\zeta = 0} (1 - G_{\sigma_{s,t}}^{(n)}((\zeta z)^{-1})\zeta z) } \leq C(M + 2 \epsilon)^{k+2},
\end{equation}
where $C$ is an upper bound on $F_{s,t}^{(n)}(z^{-1})z$ on the ball of radius $(M + 2 \epsilon)^{-1}$.  If we suppose that $\norm{a_0} = \dots = \norm{a_k} = 1$ and take
\begin{equation}
z =
\begin{pmatrix}
0 & a_0 & \dots & 0 & 0 \\
\vdots & \vdots  & \ddots & \vdots & \vdots \\
0 & 0 & \dots & a_k & 0 \\
0 & 0 & \dots & 0 & 1 \\
0 & 0 & \dots & 0 & 0
\end{pmatrix}
\in M_{k+2}(\mathcal{A}),
\end{equation}
then the upper right entry of $\sigma_{s,t}(z(Xz)^k)z$ is equal to $\sigma_{s,t}(a_0X a_1 \dots X a_k)$ and therefore,
\begin{equation}
\norm{\sigma_{t_0,t}(a_0X a_1 \dots X a_k)} \leq C(M + 2 \epsilon)^{k+2}.
\end{equation}
This implies that $\rad(\sigma_{s,t}) \leq M + 2\epsilon$ when $t - s$ is sufficiently small.  But then as in Theorem \ref{thm:Loewnerdifferentiation}, this implies that $\rad(V) \leq M$.
\end{proof}

\section{Continuous-Time Lindeberg Exchange}

There is an instructive parallel between the proof of Theorem \ref{thm:CLT} and Lindeberg's exchange method for the classical CLT.  Suppose that $Y_1$, \dots, $Y_n$ are monotone independent random variables and that $Z_1$, \dots, $Z_n$ are monotone independent arcsine random variables where $Y_j$ and $Z_j$ have the same variance.  The discrete-time analogue of \eqref{eq:Lindebergthing} is that
\begin{multline}
F_{Y_1+\dots+Y_k}^{(n)} - F_{Z_1+\dots+Z_k}^{(n)} \nonumber \\
= \sum_{j=1}^k \left( F_{Y_1+\dots+Y_{j-1}}^{(n)} \circ F_{Y_j}^{(n)} \circ F_{Z_{j+1} + \dots + Z_k}^{(n)} - F_{Y_1+\dots+Y_{j-1}}^{(n)} \circ F_{Z_j}^{(n)} \circ F_{Z_{j+1} + \dots + Z_k}^{(n)} \right)
\end{multline}
In other words, the difference between $F_{\sum Y_j}^{(n)}(z) - F_{\sum Z_j}^{(n)}(z)$ can be estimated by the sum of the differences when we swap out $Y_j$ for $Z_j$.  Let $\alpha_j$ and $\beta_j$ be the generalized laws such that
\begin{equation}
F_{X_j}(z) = z - G_{\alpha_j}(z), \qquad F_{Y_j}(z) = z - G_{\beta_j}(z).
\end{equation}
Since $Y_j$ and $Z_j$ have the same variance, a similar argument to Lemma \ref{lem:CLTestimate} shows that
\begin{align}
F_{X_j}(z) - F_{Y_j}(z) &= -G_{\alpha_j}(z) + G_{\beta_j}(z) \nonumber \\
&= O\left( \norm*{\alpha_j((z^*)^{-1} X^2 z^{-1})}^{1/2} + \norm*{\beta_j((z^*)^{-1} X^2 z^{-1})}^{1/2} \right).
\end{align}
The quantity $C_2$ is analogous to $\alpha_j(z^{-1} X^2 z^{-1})$ and $\beta_j(z^{-1} X^2 z^{-1})$, and by a power series computation
\begin{equation}
\alpha_j(z^{-1} X^2 z^{-1}) = - E(Y_jz^{-1}Y_j^2 z^{-1} Y_j) + [E(Y_jz^{-1} Y_j)]^2,
\end{equation}
and a similar result holds for $Z_j$.  (On the left $X$ is the dummy variable for $\alpha_j$ and on the right, $Y_j$ is our given random variable.)  Thus, the constant $C_2$ in \ref{eq:CLTconstants} used for the estimates in Theorem \ref{thm:CLT} is analogous to a fourth-moment estimate on $Y_j$ and $Z_j$ in the discrete setting.  To prove the classical CLT by the exchange method, only third-moment bounds are required (see for instance Terence Tao's online note \cite{Tao2015}), but we lost one degree in Lemma \ref{lem:CLTestimate} by using Cauchy-Schwarz.


\section*{Acknowledgements}

Dimitri Shlyakhtenko provided oversight and editorial feedback.  Mario Bonk's lectures and discussion were invaluable to my understanding the classical theory of Loewner chains.  I thank Mario Bonk, Vivian Healey, Steffen Rohde, and Sebastian Schlei{\ss}inger for useful conversations.  I thank Daniel Hoff for some proofreading and the anonymous referee for detailed comments that helped improve this paper.  I acknowledge the support of the NSF grants DMS-1500035 and DMS-1344970 as well as the UCLA Graduate Dean's Scholarship.

\bibliographystyle{elsarticle-num}
\bibliography{OVLT}

\begin{thebibliography}{10}
\expandafter\ifx\csname url\endcsname\relax
  \def\url#1{\texttt{#1}}\fi
\expandafter\ifx\csname urlprefix\endcsname\relax\def\urlprefix{URL }\fi
\expandafter\ifx\csname href\endcsname\relax
  \def\href#1#2{#2} \def\path#1{#1}\fi

\bibitem{Bauer2004}
R.~O. Bauer, {L}{\"o}wner's equation from a noncommutative probability
  perspective, Journal of Theoretical Probability 17~(2) (2004) 435--457.
\newblock \href {http://dx.doi.org/10.1023/B:JOTP.0000020702.23996.8f}
  {\path{doi:10.1023/B:JOTP.0000020702.23996.8f}}.

\bibitem{Schleissinger2017}
S.~Schlei{\ss}inger, The chordal {L}oewner equation and monotone probability
  theory, Infinite-dimensional Analysis, Quantum Probability, and Related
  Topics 20~(3).
\newblock \href {http://dx.doi.org/10.1142/S0219025717500163}
  {\path{doi:10.1142/S0219025717500163}}.

\bibitem{Loewner1923}
K.~L{\"o}wner, Untersuchungen {\"u}ber schlichte konforme abbildungen des
  einheitskreises. i, Mathematische Annalen 89~(1) (1923) 103--121.
\newblock \href {http://dx.doi.org/10.1007/BF01448091}
  {\path{doi:10.1007/BF01448091}}.

\bibitem{Pommerenke1975}
C.~Pommerenke, G.~Jensen, Univalent functions, Vol.~25, Vandenhoeck und
  Ruprecht, 1975.

\bibitem{deBranges1985}
L.~de~Branges, A proof of the {B}ieberbach conjecture, Acta Math. 154~(1-2)
  (1985) 137--152.
\newblock \href {http://dx.doi.org/10.1007/BF02392821}
  {\path{doi:10.1007/BF02392821}}.

\bibitem{Bauer2005}
R.~O. Bauer, Chordal {L}oewner families and univalent {C}auchy transforms,
  Journal of Mathematical Analysis and Applications 302~(2) (2005) 484 -- 501.
\newblock \href {http://dx.doi.org/10.1016/j.jmaa.2004.08.017}
  {\path{doi:10.1016/j.jmaa.2004.08.017}}.

\bibitem{GHKK2013}
I.~Graham, H.~Hamada, G.~Kohr, M.~Kohr, Univalent subordination chains in
  reflexive complex {B}anach spaces, Contemporary Mathematics 591.
\newblock \href {http://dx.doi.org/10.1090/conm/591/11829}
  {\path{doi:10.1090/conm/591/11829}}.

\bibitem{Voiculescu1986}
D.~Voiculescu, Addition of certain non-commuting random variables, Journal of
  Functional Analysis 66~(3) (1986) 323 -- 346.
\newblock \href {http://dx.doi.org/10.1016/0022-1236(86)90062-5}
  {\path{doi:10.1016/0022-1236(86)90062-5}}.

\bibitem{Voiculescu1993}
D.~Voiculescu, The analogues of entropy and {F}isher's information in free
  probability, {I}, Comm. Math. Phys. 155~(1) (1993) 71--92.
\newblock \href {http://dx.doi.org/10.1007/BF02100050}
  {\path{doi:10.1007/BF02100050}}.

\bibitem{Biane1998}
P.~Biane, Processes with free increments, Mathematische Zeitschrift 227~(1)
  (1998) 143--174.
\newblock \href {http://dx.doi.org/10.1007/PL00004363}
  {\path{doi:10.1007/PL00004363}}.

\bibitem{Muraki2000}
N.~Muraki, Monotonic convolution and monotone {L}{\'e}vy-{H}in{\v c}in formula,
  preprint (2000).

\bibitem{Muraki2001}
N.~Muraki, Monotonic independence, monotonic central limit theorem, and
  monotonic law of small numbers, Infinite Dimensional Analysis, Quantum
  Probability, and Related Topics 04.
\newblock \href {http://dx.doi.org/10.1142/S0219025701000334}
  {\path{doi:10.1142/S0219025701000334}}.

\bibitem{Hasebe2010-1}
T.~Hasebe, Monotone convolution and monotone infinite divisibility from complex
  analytic viewpoints, Infin. Dimens. Anal. Quantum Probab. Relat. Top. 13~(1)
  (2010) 111--131.
\newblock \href {http://dx.doi.org/10.1142/S0219025710003973}
  {\path{doi:10.1142/S0219025710003973}}.

\bibitem{HS2011-2}
T.~Hasebe, H.~Saigo, Joint cumulants for natural independence, Elect. Commun.
  Probab. 16 (2011) 491--506.
\newblock \href {http://dx.doi.org/10.1214/ECP.v16-1647}
  {\path{doi:10.1214/ECP.v16-1647}}.

\bibitem{HS2014}
T.~Hasebe, H.~Saigo, On operator-valued monotone independence, Nagoya Math. J.
  215 (2014) 151--167.
\newblock \href {http://dx.doi.org/10.1215/00277630-2741151}
  {\path{doi:10.1215/00277630-2741151}}.

\bibitem{Voiculescu1995}
D.~Voiculescu, Operations on certain non-commutative operator-valued random
  variables, in: Recent advances in operator algebras, no. 232 in
  Ast{\'e}risque, Societe mathematique de {F}rance, 1995, pp. 243--275.

\bibitem{Speicher1998}
R.~Speicher, Combinatorial theory of the free product with amalgamation and
  operator-valued free probability theory, Mem. Amer. Math. Soc. 132~(627).
\newblock \href {http://dx.doi.org/10.1090/memo/0627}
  {\path{doi:10.1090/memo/0627}}.

\bibitem{HT2005}
U.~Haagerup, S.~Thorbj{\o}rnsen, A new application of random matrices:
  $\text{Ext}(c_{\text{red}}^*(f_2))$ is not a group, Annals of Mathematics 162
  (2005) 711--775.
\newblock \href {http://dx.doi.org/10.4007/annals.2005.162.711}
  {\path{doi:10.4007/annals.2005.162.711}}.

\bibitem{Anderson2013}
G.~W. Anderson, Convergence of the largest singular value of a polynomial in
  independent wigner matrices, Ann. Probab. 41~(3B) (2013) 2103--2181.
\newblock \href {http://dx.doi.org/10.1214/11-AOP739}
  {\path{doi:10.1214/11-AOP739}}.

\bibitem{Williams2017}
J.~D. Williams, Analytic function theory for operator-valued free probability,
  Journal f{\"u}r die reine und angewandte Mathematik (Crelles Journal) 2017
  (2017) 119--149.
\newblock \href {http://dx.doi.org/10.1515/crelle-2014-0106}
  {\path{doi:10.1515/crelle-2014-0106}}.

\bibitem{HMS2015}
J.~W. Helton, T.~Mai, R.~Speicher, Applications of realizations (aka
  linearizations) to free probability, preprint at arXiv:1511.05330 (2015).

\bibitem{Voiculescu2000}
D.~Voiculescu, The coalgebra of the difference quotient and free probability,
  Internat. Math. Res. Notices 2000~(2) (2000) 79--106.
\newblock \href {http://dx.doi.org/10.1155/S1073792800000064}
  {\path{doi:10.1155/S1073792800000064}}.

\bibitem{Voiculescu2004}
D.~Voiculescu, Free analysis questions {I}: duality transform for the coalgebra
  of $\partial_{X:B}$, Internat. Math. Res. Notices 2004~(16) (2004) 793--822.
\newblock \href
  {http://arxiv.org/abs//oup/backfile/content_public/journal/imrn/2004/16/10.1155/s1073792804132443/2/2004-16-793.pdf}
  {\path{arXiv:/oup/backfile/content_public/journal/imrn/2004/16/10.1155/s1073792804132443/2/2004-16-793.pdf}},
  \href {http://dx.doi.org/10.1155/S1073792804132443}
  {\path{doi:10.1155/S1073792804132443}}.

\bibitem{Voiculescu2010}
D.~Voiculescu, Free analysis questions ii: The grassmannian completion and the
  series expansions at the origin, Journal f{\"u}r die reine und angewandte
  Mathematik (Crelles Journal) 645 (2010) 155--236.
\newblock \href {http://dx.doi.org/10.1515/crelle.2010.063}
  {\path{doi:10.1515/crelle.2010.063}}.

\bibitem{KVV2014}
D.~S. Kaliuzhnyi-Verbovetskyi, V.~Vinnikov, Foundations of Free Non-Commutative
  Function Theory, Vol. 199 of Mathematical Surveys and Monographs, American
  Mathematical Society, 2014.
\newblock \href {http://dx.doi.org/10.1090/surv/199}
  {\path{doi:10.1090/surv/199}}.

\bibitem{VDN1992}
D.~Voiculescu, K.~J. Dykema, A.~Nica, Free Random Variables, Vol.~1 of CRM
  Monograph Series, American Mathematical Society, 1992.
\newblock \href {http://dx.doi.org/10.1090/crmm/001}
  {\path{doi:10.1090/crmm/001}}.

\bibitem{MS2013}
T.~Mai, R.~Speicher, Operator-valued and multivariate free berry-esseen
  theorems, in: P.~Eichelsbacher, G.~Elsner, H.~K{\"o}sters, M.~L{\"o}we,
  F.~Merkl, S.~Rolles (Eds.), Limit Theorems in Probability, Statistics and
  Number Theory: In Honor of Friedrich G{\"o}tze, Springer Berlin Heidelberg,
  Berlin, Heidelberg, 2013, pp. 113--140.
\newblock \href {http://dx.doi.org/10.1007/978-3-642-36068-8_7}
  {\path{doi:10.1007/978-3-642-36068-8_7}}.

\bibitem{BMS2013}
S.~T. Belinschi, T.~Mai, R.~Speicher, Analytic subordination theory of
  operator-valued free additive convolution and the solution of a general
  random matrix problem, Journal f{\"u}r die reine und angewandte Mathematik
  (Crelles Journal)\href {http://dx.doi.org/10.1515/crelle-2014-0138}
  {\path{doi:10.1515/crelle-2014-0138}}.

\bibitem{PV2013}
M.~Popa, V.~Vinnikov, Non-commutative functions and the non-commutative
  {L}{\'e}vy-{H}in{\v c}in formula, Adv. Math. 236 (2013) 131--157.
\newblock \href {http://dx.doi.org/10.1016/j.aim.2012.12.013}
  {\path{doi:10.1016/j.aim.2012.12.013}}.

\bibitem{Popa2008a}
M.~Popa, A combinatorial approach to monotonic independence over a
  ${C}^*$-algebra, Pacific Journal of Mathematics 237 (2008) 299--325.
\newblock \href {http://dx.doi.org/10.2140/pjm.2008.237.299}
  {\path{doi:10.2140/pjm.2008.237.299}}.

\bibitem{BPV2013}
S.~T. Belinschi, M.~Popa, V.~Vinnikov, On the operator-valued analogues of the
  semicircle, arcsine and {B}ernoulli laws, Journal of Operator Theory 70~(1)
  (2013) 239--258.
\newblock \href {http://dx.doi.org/10.7900/jot.2011jun24.1963}
  {\path{doi:10.7900/jot.2011jun24.1963}}.

\bibitem{AW2016}
M.~D. Anshelevich, J.~D. Williams, Operator-valued monotone convolution
  semigroups and an extension of the {B}ercovici-{P}ata bijection, Documenta
  Mathematica 21 (2016) 841--871.

\bibitem{Lu1997}
Y.~Lu, An interacting free fock space and the arcsine law, Probability and
  Math. Stat. 17~(1) (1997) 149--166.

\bibitem{Muraki1997}
N.~Muraki, Noncommutative brownian motion in monotone fock space, Commun. Math.
  Phys. 183 (1997) 557--570.
\newblock \href {http://dx.doi.org/10.1007/s002200050043}
  {\path{doi:10.1007/s002200050043}}.

\bibitem{Zorn1945a}
M.~A. Zorn, Characterization of analytic functions in {B}anach spaces, Ann. of
  Math. 2.
\newblock \href {http://dx.doi.org/10.2307/1969198}
  {\path{doi:10.2307/1969198}}.

\bibitem{Zorn1945b}
M.~A. Zorn, {G}{\^a}teaux differentiability and essential boundedness, Duke
  Math. J. 12 (1945) 579Ð583.
\newblock \href {http://dx.doi.org/10.1215/S0012-7094-45-01252-X}
  {\path{doi:10.1215/S0012-7094-45-01252-X}}.

\bibitem{Zorn1946}
M.~A. Zorn, Derivatives and {F}r{\'e}chet differentials, Bull. Amer. Math. Soc.
  52 (1946) 133--137.
\newblock \href {http://dx.doi.org/10.1090/S0002-9904-1946-08524-9}
  {\path{doi:10.1090/S0002-9904-1946-08524-9}}.

\bibitem{DP1940}
N.~Dunford, B.~Pettis, Linear operators on summable functions, Trans. Amer.
  Math. Soc. 47 (1940) 323--392.
\newblock \href {http://dx.doi.org/10.1090/S0002-9947-1940-0002020-4}
  {\path{doi:10.1090/S0002-9947-1940-0002020-4}}.

\bibitem{Phillips1940}
R.~S. Phillips, On linear transformations, Trans. Amer. Math. Soc. 48 (1940)
  516--541.

\bibitem{MU1971}
S.~Moedomo, J.~Uhl, Radon-nikodym theorem for the bochner and pettis integral,
  Pacific. J. Math. 38 (1971) 531Ð536.
\newblock \href {http://dx.doi.org/10.2140/pjm.1971.38.531}
  {\path{doi:10.2140/pjm.1971.38.531}}.

\bibitem{DU1977}
J.~Diestel, J.~Uhl, Vector Measures, no.~15 in Mathematical Surveys and
  Monographs, Amer. Math. Soc., 1977.
\newblock \href {http://dx.doi.org/10.1090/surv/015}
  {\path{doi:10.1090/surv/015}}.

\bibitem{Blackadar2006}
B.~Blackadar, Operator Algebras: Theory of ${C}^*$-algebras and von {N}eumann
  algebras, Vol. 122 of Encyclopaedia of Mathematical Sciences,
  Springer-Verlag, Berlin, Heidelberg, 2006.
\newblock \href {http://dx.doi.org/10.1007/3-540-28517-2}
  {\path{doi:10.1007/3-540-28517-2}}.

\bibitem{Lance1995}
E.~C. Lance, {H}ilbert ${C}^*$-Modules: A Toolkit for Operator Algebraists,
  London Mathematical Society Lecture Note Series, Cambridge University Press,
  Cambridge, 1995.
\newblock \href {http://dx.doi.org/10.1017/CBO9780511526206}
  {\path{doi:10.1017/CBO9780511526206}}.

\bibitem{Taylor1972}
J.~L. Taylor, A general framework for a multi-operator functional calculus,
  Advances in Math. 9 (1972) 183--252.
\newblock \href {http://dx.doi.org/10.1016/0001-8708(72)90017-5}
  {\path{doi:10.1016/0001-8708(72)90017-5}}.

\bibitem{Taylor1973}
J.~L. Taylor, Functions of several noncommuting variables, Bull.\ Amer.\ Math.\
  Soc. 79 (1973) 1--34.
\newblock \href {http://dx.doi.org/10.1090/S0002-9904-1973-13077-0}
  {\path{doi:10.1090/S0002-9904-1973-13077-0}}.

\bibitem{AKV2013}
G.~Abduvalieva, D.~S. Kaliuzhnyi-Verbovetskyi, Fixed point theorems for
  noncommutative functions, Journal of Mathematical Analysis and Applications
  401~(1) (2013) 436 -- 446.
\newblock \href {http://dx.doi.org/10.1016/j.jmaa.2012.12.038}
  {\path{doi:10.1016/j.jmaa.2012.12.038}}.

\bibitem{AKV2015}
G.~Abduvalieva, D.~S. Kaliuzhnyi-Verbovetskyi, Implicit/inverse function
  theorems for free noncommutative functions, Journal of Functional Analysis
  269~(9) (2015) 2813 -- 2844.
\newblock \href {http://dx.doi.org/10.1016/j.jfa.2015.07.011}
  {\path{doi:10.1016/j.jfa.2015.07.011}}.

\bibitem{AM2016}
J.~Agler, J.~E. McCarthy, The implicit function theorem and free algebraic
  sets, Transactions of the American Mathematical Society 368 (2016)
  3157--3175.
\newblock \href {http://dx.doi.org/10.1090/tran/6546}
  {\path{doi:10.1090/tran/6546}}.

\bibitem{Hasebe2010-2}
T.~Hasebe, Monotone convolution semigroups, Studia Math. 200 (2010) 175--199.
\newblock \href {http://dx.doi.org/10.4064/sm200-2-5}
  {\path{doi:10.4064/sm200-2-5}}.

\bibitem{Lenczewski2007}
R.~Lenczewski, Decompositions of the free additive convolution, Journal of
  Functional Analysis 246~(2) (2007) 330--365.
\newblock \href {http://dx.doi.org/https://doi.org/10.1016/j.jfa.2007.01.010}
  {\path{doi:https://doi.org/10.1016/j.jfa.2007.01.010}}.

\bibitem{Nica2009}
A.~Nica, Multi-variable subordination distributions for free additive
  convolution, Journal of Functional Analysis 257~(2) (2009) 428 -- 463.
\newblock \href {http://dx.doi.org/10.1016/j.jfa.2008.12.022}
  {\path{doi:10.1016/j.jfa.2008.12.022}}.

\bibitem{NS2006}
A.~Nica, R.~Speicher, Lectures on the Combinatorics of Free Probability, no.
  335 in London Mathematical Society Lecture Note Series, Cambridge University
  Press, 2006.
\newblock \href {http://dx.doi.org/10.1017/CBO9780511735127}
  {\path{doi:10.1017/CBO9780511735127}}.

\bibitem{RR1994}
M.~Rosenblum, J.~Rovnyak, Topics in Hardy classes and univalent functions,
  BirkhŠuser Advanced Texts: Basler LehrbŸcher, BirkhŠuser, 1994.

\bibitem{CDG2009}
M.~Contreras, S.~Diaz-Madrigal, P.~Gumenyuk, {L}oewner chains in the unit disk,
  Revista Matematica Iberoamericana 26 (2010) 975--1012.
\newblock \href {http://dx.doi.org/10.4171/RMI/624}
  {\path{doi:10.4171/RMI/624}}.

\bibitem{ABFN2013}
M.~Anshelevich, S.~T. Belinschi, M.~F{'e}vrier, A.~Nica, Convolution powers in
  the operator-valued framework, Trans.\ Am.\ Math.\ Soc. 365 (2013)
  2063--2097.
\newblock \href {http://dx.doi.org/10.1090/S0002-9947-2012-05736-9}
  {\path{doi:10.1090/S0002-9947-2012-05736-9}}.

\bibitem{AGZ2009}
G.~W. Anderson, A.~Guionnet, O.~Zeitouni, An Introduction to Random Matrices,
  Cambridge Studies in Advanced Mathematics, Cambridge University Press, 2009.
\newblock \href {http://dx.doi.org/10.1017/CBO9780511801334}
  {\path{doi:10.1017/CBO9780511801334}}.

\bibitem{Lawler2005}
G.~F. Lawler, Conformally invariant processes in the plane, Vol. 114 of
  Mathematical Surveys and Monographs, American Mathematical Society,
  Providence, 2005.
\newblock \href {http://dx.doi.org/10.1090/surv/114}
  {\path{doi:10.1090/surv/114}}.

\bibitem{Komura1967}
Y.~Komura, Nonlinear semi-groups in {H}ilbert space, J. Math. Soc. Japan 19~(4)
  (1967) 493--507.
\newblock \href {http://dx.doi.org/10.2969/jmsj/01940493}
  {\path{doi:10.2969/jmsj/01940493}}.

\bibitem{Tao2015}
T.~Tao,
  \href{https://terrytao.wordpress.com/2015/11/02/275a-notes-4-the-central-limit-theorem/}{Math
  275{A}, {N}otes 4: The central limit theorem}, online lecture notes.
\newline\urlprefix\url{https://terrytao.wordpress.com/2015/11/02/275a-notes-4-the-central-limit-theorem/}

\end{thebibliography}

\end{document}